\documentclass[11pt]{article}
\usepackage{fullpage}
\usepackage{amsmath,amssymb,amsthm}
\usepackage[T1]{fontenc}
\usepackage{comment}
\usepackage{stmaryrd}
\usepackage[english]{babel}
\usepackage[left=3cm,right=3cm,top=3cm,bottom=4cm]{geometry}

\usepackage{tikz,tikz-cd}
\usetikzlibrary{positioning}
\usetikzlibrary{calc}
\usetikzlibrary{decorations.pathmorphing}
\usepackage{braket}
\usepackage{mathtools,stmaryrd, mathdots}

\SetSymbolFont{stmry}{bold}{U}{stmry}{m}{n}
\usepackage{lmodern,bm,bbm,mathrsfs}
\usepackage{manfnt}

\usepackage{enumitem} 

\usepackage[scr=esstix]{mathalfa}

\usepackage{hyperref}
\newtheorem{theorem}{Theorem}[section]
\newtheorem{corollary}[theorem]{Corollary}
\newtheorem{lemma}[theorem]{Lemma}
\newtheorem{proposition}[theorem]{Proposition}
\newtheorem{conjecture}[theorem]{Conjecture}
\newtheorem{question}[theorem]{Question}

\theoremstyle{definition}
\newtheorem{definition}[theorem]{Definition}
\newtheorem{remark}[theorem]{Remark}
\newtheorem{claim}[theorem]{Claim}
\newtheorem{construction}[theorem]{Construction}
\newtheorem{variant}[theorem]{Variant}
\newtheorem{example}[theorem]{Example}

\renewcommand{\bar}[1]{\overline{#1}}

\renewcommand{\hat}[1]{\widehat{#1}}

\renewcommand{\tilde}[1]{\widetilde{#1}}

\newcommand{\bA}{\mathbb{A}}
\newcommand{\bC}{\mathbb{C}}
\newcommand{\bD}{\mathbb{D}}
\newcommand{\bE}{\mathbb{E}}
\newcommand{\bF}{\mathbb{F}}
\newcommand{\bG}{\mathbb{G}}

\newcommand{\bP}{\mathbb{P}}

\newcommand{\bS}{\mathbb{S}}

\newcommand{\bZ}{\mathbb{Z}}
\newcommand{\ZT}{{\mathbb{Z}_2}}

\newcommand{\fC}{\mathfrak{C}}

\newcommand{\ND}{\operatorname{N}^{\operatorname{D}}}
\newcommand{\Nhc}{\operatorname{N}^{\operatorname{hc}}}

\newcommand{\sR}{\mathsf{R}}
\newcommand{\sS}{\mathsf{S}}

\newcommand{\vir}{\mathrm{vir}}

\newcommand{\dercat}{D_{\mathit{QCoh}}}

\newcommand{\FonR}{\mathcal{R}_{F,\triangleleft}^{\gen,\Delta}}

\newcommand{\Gproj}{\Pi}
\newcommand{\Gprojtw}{\Pi^{\operatorname{tw}}}
\newcommand{\Gsub}{\Gamma}

\DeclareMathOperator{\Aff}{Aff}
\DeclareMathOperator{\Ann}{Ann}
\DeclareMathOperator{\Aut}{Aut}
\DeclareMathOperator{\End}{End}

\DeclareMathOperator{\Cl}{Cl}
\DeclareMathOperator{\Coh}{Coh}

\DeclareMathOperator{\cone}{cone}

\DeclareMathOperator{\Ex}{Ex}

\DeclareMathOperator{\GL}{GL}
\DeclareMathOperator{\SL}{SL}
\DeclareMathOperator{\G}{G}
\DeclareMathOperator{\gen}{gen}
\DeclareMathOperator{\Grpd}{Grpd}
\DeclareMathOperator{\Hilb}{Hilb}

\DeclareMathOperator{\Hom}{Hom}
\DeclareMathOperator{\id}{id}

\DeclareMathOperator{\Ker}{Ker}

\DeclareMathOperator{\lip}{S\Gamma}
\DeclareMathOperator{\ob}{ob}
\DeclareMathOperator{\Path}{Path}
\DeclareMathOperator{\PGL}{PGL}

\DeclareMathOperator{\pr}{pr}

\DeclareMathOperator{\Sd}{Sd}
\DeclareMathOperator{\SO}{SO}
\DeclareMathOperator{\sn}{sn}
\DeclareMathOperator{\OO}{O}

\DeclareMathOperator{\Spec}{Spec}
\DeclareMathOperator{\Spin}{Spin}
\DeclareMathOperator{\Spincat}{\mathcal{S}\mathit{pin}}

\DeclarePairedDelimiter{\abs}{\lvert}{\rvert}

\makeatletter
\providecommand{\leftsquigarrow}{%
	\mathrel{\mathpalette\reflect@squig\relax}%
}
\newcommand{\reflect@squig}[2]{%
	\reflectbox{$\m@th#1\rightsquigarrow$}%
}
\makeatother

\tikzset{%
  vertex/.style={shape=circle,fill=black,minimum size=6pt,inner sep=0},
  framing/.style={shape=rectangle,fill=black,minimum size=6pt,inner sep=0},
  baseline={([yshift=-0.8ex]current bounding box.center)}
}

\title{Spin structures on perfect complexes}
\author{Nikolas Kuhn}
\date{\today}

\begin{document}

\maketitle

\begin{abstract}
  We define spin structures on perfect complexes outside of characteristic two, generalizing the usual notion for vector bundles. We give an explicit local characterization of spin structures, and show that for an oriented quadratic complex $E$ on an algebraic stack, spin structures on $E$ are parametrized by a  degree $2$ gerbe. As an application, we show how to lift the K-theory class of Oh--Thomas in DT4 theory to a genuine (twisted) sheaf.    
\end{abstract}

Let $\mathcal{M}$ be a scheme or algebraic stack away from characteristic $2$. 
\begin{definition}\label{def:quad-comp}
A \emph{quadratic complex} on $\mathcal{M}$ is given by a perfect complex $\mathbb{E}$ on $\mathcal{M}$ together with an isomorphism $\theta: \bE\xrightarrow{\sim} \bE^{\vee}$ to its dual, such that $\theta^{\vee} = \theta$.	
\end{definition}
When $\bE$ is represented by a vector bundle in degree $0$, this is the same as giving an $\OO(m)$-principal bundle, where $m$ is the rank of $\bE$. We may say that $\theta$ gives an \emph{orthogonal} structure on $\bE$. Taking the determinant of $\theta$, one has a natural map $\det \theta: \det\bE \simeq (\det\bE)^{\vee}$.
\begin{definition}
	An \emph{orientation} on the quadratic complex $\bE$ is an isomorphism $\varphi:\det\bE\to \mathcal{O}_M$ such that $\varphi^{\vee} \circ \varphi = \det \theta$.
\end{definition} 
When $\bE$ is a vector bundle, this is the same as giving a reduction of structure group from $\OO(m)$ to $\SO(m)$. We may say that $(\theta,\varphi)$ give a \emph{special orthogonal} structure on $\bE$. 

The spin group $\Spin(m)$ is the unique non-trivial double covering of $\SO(m)$. 
In this paper, we ask and answer the following

\begin{question}\label{que:spin}
	Given an oriented quadratic complex $\bE$, what does it mean to give a \emph{spin structure} on $\bE$? 
\end{question}

Our approach to Question \ref{que:spin} is to give a definition that works locally for suitable ``self-dual'' representatives of the complex $\bE$. We then show that this local definition is essentially independent of choices, and can be globalized by declaring a spin structure to be a compatible system of choices for all local representatives simultaneously. 

The notion of spin structure presented here arose naturally as an obstruction class to the problem of constructing a virtual structure sheaf in DT4 theory. The problem is resolved as follows: Assuming an orientation has been chosen, a virtual structure sheaf always exists as a \emph{twisted} sheaf with respect to the $\ZT$-gerbe parametrizing spin structures on the obstruction complex.  
\tableofcontents

\section{Introduction}
\subsection{Main Results}
The main goal of this paper is to give a good definition of spin structure on a perfect complex over an arbitrary algebraic stack $\mathcal{M}$ (outside of characteristic $2$). This is achieved in Definition \ref{def:spin-structure}. As one might hope, its properties mirror closely the case of spin structures on vector bundles. In particular, a spin structure on a complex $\bE$ over $\mathcal{M}$ induces an underlying structure of oriented quadratic complex on $\bE$. Conversely, given an oriented quadratic complex $\bE$, there is a natural $\ZT:=\bZ/2\bZ$-gerbe $\mathcal{M}^{\operatorname{sp}}$ parametrizing spin structures on $\bE$ (Corollary \ref{cor:spin-gerbe}). In particular, we have a characteristic class
\[o_{\operatorname{sp}}(\bE)\in H^2(\mathcal{M},\ZT)\]
of oriented quadratic complexes, obstructing the existence of a spin structure on $\bE$ (Definition \ref{def:obstr-class}). If this class is zero, the set of spin structures on $\bE$ up to isomorphism is a torsor under $H^1(\mathcal{M}, \ZT)$.
We give a detailed summary of our strategy in \S \ref{sec:sketch}, but here is a very rough sketch: To define a spin structure on a complex on an algebraic stack, we give a definition that works in the case of an affine scheme (Definition \ref{def:spin-structure-affine}) and show that it satisfies descent (Proposition \ref{prop:spin-gerbe}). In the affine case, our definition essentially boils down to picking a good \emph{self-dual} representative $E^{\bullet}$ of $\bE$ (Definition \ref{def:self-dual}) whose middle term is an oriented quadratic \emph{vector bundle} $E$. We then declare a spin structure on $\bE$ to be roughly a ``twisted'' spin structure on $E$, where the twist is by the determinant of the positive part of the representing complex $E^{\bullet}$ (Definition \ref{def:spin-functor}). That the definition obtained in this way is sufficiently independent of choice of representative is the main technical challenge that one needs to address. Most of \S\S \ref{sec:self-dual-cat} and \ref{sec:spin-functor} are dedicated to dealing with this issue, the key result being Corollary \ref{cor:restriction-iso}.
\paragraph{Construction of a DT4 twisted virtual structure sheaf.}
In their paper \cite{ot_counting}, Oh--Thomas construct a $K$-theoretic \emph{twisted virtual structure sheaf} $[\widehat{\mathcal{O}}_M^{\vir}]$ on moduli spaces $M$ of stable sheaves of Calabi--Yau fourfolds. These moduli spaces are characterized by carrying an obstruction theory $\bE[1]\to L_M$ with $\bE$ a quadratic complex \footnote{Which in many cases of interest admits a \emph{canonical} orientation \cite{ju_bord}}, and satisfying an additional \emph{isotropy condition} -- we will call such an object a \emph{DT4 obstruction theory}. The $K$-theory class of Oh--Thomas has the slightly curious feature that -- by construction -- it lives only in $K_0(M)[1/2]$, i.e. it is only well-defined after inverting $2$ in the coefficients of $K$-theory. 

In work of Kool--Rennemo \cite{kr_mag}, the authors make crucial use of the fact that in special cases there exists a line bundle $L$ on $M$ with formal square root $L^{1/2}$ such that $[\mathcal{O}_M^{\vir}]:=[\widehat{\mathcal{O}}_M^{\vir}]\otimes L^{1/2}$ admits a refinement to a $\bZ_2$-graded sheaf $(\mathcal{O}_M^{\vir,+}, \mathcal{O}_M^{\vir,-})$. In such cases $[\mathcal{O}_M^{\vir}] = [\mathcal{O}_M^{\vir,+}] - [\mathcal{O}_M^{\vir,-}]$, and this is a class in integral $K$-theory. 

The definition of spin structure undertaken in the present work arose as an obstruction to lifting the class of Oh--Thomas to a genuine ($\ZT$-graded) sheaf in full generality. Our main result here may be stated as follows (cf. Theorem \ref{thm:DT4-local}, although the comparison with Oh--Thomas happens only via Theorem \ref{thm:twisted-str-sheaf} and Theorem \ref{thm:str-shvs-compatible}.)
\begin{theorem}
	Let $M$ be a quasi-projective scheme with an oriented DT4 obstruction theory $\bE$. Then a choice of spin structure on $\bE$ gives rise to a lift of $[\widehat{\mathcal{O}}_M^{\vir}]$ to a $\bZ_2$-graded sheaf. 
\end{theorem}
One important advantage of working with sheaves instead of $K$-theory classes is that they satisfy smooth descent and our construction of virtual structure sheaf generalizes to algebraic stacks $\mathcal{M}$ with a DT4 obstruction theory carrying a spin structure. In particular, one may always assume that a spin structure exists by working on the associated gerbe $\mathcal{M}^{\operatorname{sp}}$, on which a virtual structure sheaf exists as a $1$-twisted sheaf (Theorem \ref{thm:twisted-str-sheaf}). 

That it is necessary to invert $2$ in the construction of $[\widehat{\mathcal{O}}^{\vir}_M]$ is now explained by the fact that, in general, a distinguished (or even any) choice of spin structure may not exist, so that a canonical virtual structure sheaf exists only on the spin gerbe. To recover the Oh--Thomas class from the twisted sheaf, we develop in Appendix \ref{app:untw} a general  \emph{untwisted push forward} procedure for $K$-theory classes along a $\mu_n$ gerbe. 

\subsection{Further Directions}

\paragraph{Characteristic $2$.}
The approach taken in this paper runs into two serious issues in characteristic two: Firstly, in the construction of symmetric representatives in \S \ref{sec:symm-rep}, we divide by $2$ to find a representative of $\theta$ that is self-dual on the level of complexes. Just as in the classical algebraic theory of spin groups, if one wants to work in all characteristics, it becomes necessary to distinguish between the notions of ``quadratic structure'' and ``symmetric bilinear pairing'', and Definition \ref{def:quad-comp} used here corresponds to the latter. Secondly, in showing independence of representatives, we use in \S \ref{sec:spin-functor} that, away from characteristic $2$, \'etale cohomoloy with $\ZT$-coefficients is invariant under taking products with affine spaces. This does not hold anymore in characteristic $2$. Additionally, the relevant finite group scheme in that case would be $\mu_2$ in place of $\ZT$.

Nevertheless, we expect that there should still be a good theory of spin structures on complexes in all characteristics, mirroring the classical case.
  
\paragraph{$-2$-shifted symplectic stacks.}
In more established enumerative theories, such as curve counting or sheaf-counting on surfaces and CY3s, a virtual structure sheaf is obtained from a $2$-term \emph{perfect obstruction theory} on some scheme $X$. In almost all interesting cases, the perfect obstruction theory comes from an underlying \emph{quasi-smooth} derived scheme $\mathcal{X}$, of which $X$ is the classical truncation, by restricting the cotangent complex of $\mathcal{X}$. In these cases, the virtual structure sheaf on $X$ is obtained as the collection of cohomology groups of the genuine structure sheaf $\mathcal{O}_{\mathcal{X}}$. 

	A similar picture is expected to hold for the twisted virtual structure sheaf constructed in \S \ref{sec:str-sheaf}: The interesting examples of DT4 obstruction theories are obtained as classical truncations of derived schemes or stacks with a \emph{$-2$-shifted symplectic structure} \cite{ptvv_shifted}. By analogy, there should be some enhanced construction on the derived stack itself, from which the virtual structure sheaf in our sense can be recovered. This is also a special case of a conjecture of D. Joyce regarding derived Lagrangians over $-1$-shifted symplectic stacks.
	We only give a rough version here, without defining all terms.
\begin{conjecture}\label{conj:der-stack}
	Let $\mathcal{X}$ be a (possibly higher) derived algebraic stack over $\bC$ with a $-2$-shifted symplectic structure. Then a choice of spin structure on the $1$-shifted tangent bundle induces an element of a category of $2$-periodic complexes on $\mathcal{X}$.
\end{conjecture}

\paragraph{Derived Clifford Algebras.}
Following a suggestion by Nick Rozenblyum, an alternative approach to spin structures on complexes is via module categories over a derived version of Clifford algebras. Relevant prior work on this concept has been done by Preygel \cite[\S 9]{preygel}, who in particular gives a construction of Clifford algebras that is likely to generalize from bundles to complexes.\footnote{A construction of derived Clifford algebra has also been suggested by Vezzosi \cite{vezzosi_quad}, although it does not appear to be the correct one.} 
In characteristic zero, one can also mirror the usual construction of Clifford algebras using generators and relations in a differential graded setting (cf. \cite[Remark 2.1.7]{lurie2011dag}). 

One advantage of this fully ``derived'' approach is that it lends itself more easily to genuine derived constructions (working in the stable infinity category of complexes, rather than its homotopy category) that can be used directly on derived schemes and stacks. In particular, we hope is that it provides an inroad to solving Conjecture \ref{conj:der-stack} (this is ongoing work with Emile Bouaziz).

Another possible advantage is that -- in analogy to the classical situation -- an approach using derived Clifford algebras can be expected to include the characteristic $2$ case without additional difficulties. 

\paragraph{Existence of Spin structures in DT4 theory.}
Due to their importance of constructing genuine structure sheaves in DT4 theory, the following question is natural.
\begin{question}
	Do moduli spaces of sheaves on a Calabi--Yau fourfold $Y$ (sometimes) admit spin structures? If so, canonical ones?
\end{question}

Computations indicate that already $\Hilb^n Y$ does not admit a spin structure for $n>1$  -- as was pointed out to us by Arkadij Bojko, Woonam Lim, Felix Thimm and Jørgen Rennemo. A direct argument establishing this fact via reduction to $Y=\bC^4$ was explained to us by Rennemo. 

A promising alternate question -- inspired by the application to factorization in \cite{kr_mag} -- would be whether there exist ``twisted'' spin structures with twists given by natural line bundles on the moduli space (see Remark \ref{rem:spinc-structures}). We thank Thimm and Rennemo for pointing this out.

\paragraph{Generalized principal bundles.}
The category of perfect complexes provides a natural enlargement of the category of $\GL_n$-principal bundles (for varying $n$), which is useful e.g. for many applications in moduli theory. Similarly, one has natural enlargements of the categories of principal bundles for other classes of algebraic groups, the last of which is provided in this paper.   
\begin{center}
	\begin{tabular}{c|c}
	Structure group	& Generalized principal bundle\\ 
		\hline 
	$\GL$ & perfect complex \\
	$\SL$ & oriented perfect complex \\
	$\OO$  & quadratic (self-dual) complex \\
	$\SO$ & oriented quadratic complex \\
	$\operatorname{Sp}$ & anti-self-dual complex \\	
	$\PGL$ & twisted perfect complex \\
	$\Spin$ &  complex with spin structure 
	\end{tabular}
\end{center}
This perspective has previously been employed for the groups $\PGL$ by Lieblich \cite{lie_mod} and for the orthogonal and symplectic groups by Bu \cite{bu_enumI}, \cite{bu_enumII}. However, note that all the entries on the right hand side are ad-hoc.

\begin{question}
	Do categories of generalized principal bundles exist for all reductive (or just linear) algebraic groups $G$? Is there a conceptual way of defining them for all $G$?
\end{question}

\subsection{Sketch of the construction}\label{sec:sketch}
We explain the main ideas in the case that $X$ is a quasi-projective scheme over $\bC$, and where $(\bE,\theta)$ be a quadratic complex on $X$ with $\bE$ having perfect amplitude in $[-1,1]$. 
\paragraph{Self-dual Representatives.}
By a construction of Oh--Thomas, we can find a \emph{self-dual representative} of $\bE$. Here, a three-term complex $E^{\bullet}$ of locally free sheaves is \emph{self-dual}, if it is endowed with an isomorphism $\theta_{\bullet}: E^{\bullet} \to E_{\bullet}:= (E^{\bullet})^{\vee}$ that is self-dual in the sense that $\theta_i^{\vee} = \theta_{-i}$ for all $i\in \bZ$. Equivalently, $E^{\bullet}$ is of the form 
\begin{equation}\label{eq:self-dual-3term}
	B_{-1}\xrightarrow{d_{-1}} E\xrightarrow{d_0} B^{1},
\end{equation}
for vector bundles $B_{-1},E,B^1$, where $E$ carries an orthogonal structure, $B_{-1}$ is identified with $(B^1)^{\vee}$, and where $d_{-1} = d_0^{\vee}$ under these identifications. A self-dual complex represents $\bE$, if it is endowed with a specified isomorphism $\eta:[E^{\bullet}]\to \bE$ in $D(X)$ that intertwines $[\theta_{\bullet}]$ and $\theta$, i.e. such that we have an equality of maps  $\eta^{\vee}\circ [\theta_{\bullet}] = \theta \circ \eta$ in $D(X)$. 

\paragraph{Isotropic Reductions.}
Now fix the self-dual representative $E^{\bullet}$. Suppose $K\subset B_{-1}$ is a sub-vector bundle, i.e. that the quotient sheaf is again locally free. Suppose further that the composition $K\to B_{-1}\to E$ is again injective. In particular, the acyclic complex 
\[K^{\bullet}:= [K\to K]\] 
in degrees $[-1,0]$ is a subcomplex of $E^{\bullet}$. Dually, we have a degreewise surjection $E^{\bullet}\to K_{\bullet}:=(K^{\bullet})^{\vee}$. Together, these fit in the diagram
\begin{equation}\label{eq:K-reduction-diag}
	\begin{tikzcd}
		K\ar[r,equals]\ar[d]&K\ar[d]\ar[r]&0 \ar[d] \\
		B_{-1}\ar[r]\ar[d]& E\ar[r]\ar[d] & B^1\ar[d]\\
		0 \ar[r]& K^{\vee} \ar[r, equals] & K^{\vee}
	\end{tikzcd}
\end{equation}
One sees that the vertical composition $K\to E\to K^{\vee}$ is zero. In particular, $K\subset K^{\perp} = \Ker(E\to K^{\vee})$ is contained in its orthogonal complement, hence is an isotropic sub-bundle. The bundle $K^{\perp}/K$ is a quadratic bundle in its own right, and we obtain a new self-dual representative of $\bE$ by taking vertical subquotients of \eqref{eq:K-reduction-diag}:
\begin{equation}\label{eq:iso-red-3term}	
	B_{-1}/K \to K^{\perp}/K \to (B_{-1}/K)^{\vee}.
\end{equation}
We call \eqref{eq:iso-red-3term} the \emph{isotropic reduction of $E^{\bullet}$ along $K$}. More generally, we say that 
\begin{equation*}
	F^{\bullet} = [A_{-1}\to F \to A^1]
\end{equation*} is an isotropic reduction of $E^{\bullet}$, if it is equipped with an isomorphism of self-dual representatives (i.e. compatible with the self-duality maps and the identifications with $\bE$) from a complex \eqref{eq:iso-red-3term} for some $K$.

The collection of self-dual representatives of $\bE$ naturally forms a category $\mathcal{R}= \mathcal{R}_\bE$ in which morphisms from $F^{\bullet}$ to $E^{\bullet}$ are the isotropic reductions from $E^{\bullet}$ to $F^{\bullet}$. One can show that the category $\mathcal{R}$ is \emph{connected}, i.e. any two self-dual representatives are related by a sequence of isotropic reductions (cf. Corollary \ref{cor:repcat-connected}). 
\paragraph{Orientations.}
Given a self-dual representative $E^{\bullet}$, giving an orientation on $\bE$ is the same as giving one on the quadratic bundle $E$ via the isomorphism 
\begin{equation}\label{eq:det-isom}
	\det \bE \simeq \det E\otimes \det B_{-1}^{\vee} \otimes \det (B^1)^{\vee}\simeq \det E
\end{equation}
induced by $\eta$ and where we use the pairing 
\begin{align*}
	\det B_{-1}^{\vee} \otimes \det (B^1)^{\vee} = \det B_{-1}^{\vee}\otimes \det B_{-1} &\to \mathcal{O}\\
	b_1^*\wedge \cdots \wedge b_n^*\otimes b_n\wedge \cdots \wedge b_1& \mapsto (-1)^n .
\end{align*}
An isotropic sub-bundle $K\subseteq K^{\perp}\subseteq E$ induces an isomorphism
\begin{equation*}
	\det E \simeq \det (K^{\perp}/K \oplus K\oplus K^{\vee}) \simeq \det K^{\perp}/K.
\end{equation*}
This isomorphism is compatible with the identifications obtained from applying \eqref{eq:det-isom} to $E^{\bullet}$ and its reduction \eqref{eq:iso-red-3term} respectively.
\paragraph{Associated principal bundles.}
To any special orthogonal bundle $E$ of rank $m$, let $P_{E}^{\SO}$ denote the associated $\SO(m)$ principal bundle. The choice of an isotropic rank $k$ sub-bundle $K\subset E$ induces a reduction of structure group of $P_E^{\SO}$ to the subgroup $\G(m, k)\subset \SO(m)$ preserving a rank $k$ isotropic subspace of the standard oriented quadratic bundle.\footnote{To simplify the discussion, we restrict to the cases $m\neq 0$ and $m> 2k$. }  See \S \ref{sec:standard-forms} in the Appendix for conventions. We denote the associated principal bundle by $P^{\G}_{E,K}$.
Finally, an isomorphism of special orthogonal bundles $K^{\perp}/K  \to F$ induces a diagram of principal bundles 
\begin{equation}\label{eq:torsor-diag}
	P_F^{\SO}\leftarrow P_{E,K}^{\G}\to P_E^{\SO}
\end{equation}
which is equivariant over the diagram of structure groups
\begin{equation}\label{eq:group-diag}
	\SO(m-2k)\xleftarrow{\Gproj_{m,k}} \G(m,k) \subseteq \SO(m).
\end{equation}
More concretely, the elements of $\G(m,k)$ are block-diagonal matrices of the form 
\[\begin{bmatrix}
	A & * & * \\
	0 & M & * \\
	0 & 0 & (A^T)^{-1}
\end{bmatrix}\]
where $A\in \GL(k)$ and $M\in \SO(m-2k)$. The map $\Gproj$ is simply restriction to the middle block. 

In particular, any isotropic reduction between self-dual representatives \eqref{eq:iso-red-3term} of an oriented quadratic complex gives rise to a diagram \eqref{eq:torsor-diag} by considering the middle terms. 
\paragraph{Spin structures with a twist.}
One may hope that a good notion of spin structure on the oriented quadratic complex $\bE$ can be defined simply by declaring it to be a spin structure on the middle term of any self-dual representative $E^{\bullet}$ (as the analogous statement holds for orientations). If this is well-behaved, then given a spin structure $\sigma$ on the middle term $E$ of $E^{\bullet}$, and an isotropic reduction $F^{\bullet}$ of $E^{\bullet}$, we should be able to produce a spin structure $\tau$ on $F^{\bullet}$. The following strategy suggests itself. The given data provides the following diagram of principal bundles with solid arrows
\begin{equation*}
	\begin{tikzcd}
		*\ar[d, dashed]&? \ar[l, dashed]\ar[r,dashed]\ar[d,dashed]& P_{\sigma}^{\Spin}\ar[d,"2:1"] \\
		P_F^{\SO}&P_{E,K}^{\G}\ar[l]\ar[r] & P_E^{\SO}.
	\end{tikzcd}
\end{equation*}
If we could fill out the upper row and dashed arrows in a way that all vertical arrows are double coverings, then the term $*$ would give a spin structure on $F$, as desired. This could easily be achieved (via base change, followed by induction of structure group) if it held that the spin groups fit into a corresponding diagram of double coverings extending \eqref{eq:group-diag}
\begin{equation*}
	\begin{tikzcd}
		\Spin(m-2k)\ar[d,"2:1"]& ?? \ar[l,dashed]\ar[r,dashed]\ar[d,dashed,"2:1"]& \Spin(m)\ar[d,"2:1"] \\
		\SO(m-2k) &\G(m,k)\ar[l]\ar[r] & \SO(m)
	\end{tikzcd}
\end{equation*}
However, such a diagram cannot exist: Up to homotopy, we have $\G(m,k)\simeq \SO(m-2k)\times \GL(k)$, and passing to fundamental groups in \eqref{eq:group-diag}, we get \footnote{Say $m-2k \geq 3$ here, for convenience.}
\begin{equation*}
	\begin{tikzcd}[row sep = tiny]
		\bZ_{2}&\bZ_2\times \bZ\ar[l]\ar[r]& \bZ_2 \\
		\overline{a}&(\overline{a},b)\ar[l, mapsto]\ar[r, mapsto] & \overline{a}+\overline{b}.
	\end{tikzcd}
\end{equation*}
The preimages of the trivial subgroup from each side give are the subgroups $\langle \overline{1},1\rangle$ and $\langle 0,1 \rangle$ respectively, which are distinct. 

By restricting only to the middle term of $E^{\bullet}$, we have discarded too much information. The remaining piece contributes a twist with which the above discussion goes through. For $m\geq 3$, define the (complex) $\Spin^{\bC}$-group $\Spin^{\bC}(m)$ as the double cover of $\SO(m)\times \bC^*$ defined by the sub-group $\langle \overline{1}, 1\rangle \subset \bZ_2 \times \bZ$. In particular, it restricts to $\Spin(m)$ over $\SO(m)\times\{\id\}$ and to a non-trivial double covering over $\{\id\}\times \bC^*$.

\begin{definition}\label{def:spin-str-rep}
	A spin \emph{structure on the representative $E^{\bullet}$} is a $\Spin^{\bC}(m)$-structure on the $\SO(m)\times \bC^*$-principal bundle $P_E^{\SO}\times P_{\det B^1}^{\bC^*}$.  
\end{definition}

In place of \eqref{eq:group-diag}, consider the modified diagram
\begin{equation}\label{eq:group-diag-twisted}
	\SO(m-2k)\times \bC^* \xleftarrow{\Gprojtw_{m,k}} \G(m,k)\times \bC^* \to \SO(m)\times \bC^*
\end{equation}
where 
\[\Gprojtw_{m,k}: \left( \begin{bmatrix}
	A & * & * \\ 0 & M & * \\ 0 & 0 & (A^T)^{-1}
\end{bmatrix}, t\right) \mapsto (M, t \cdot \det A).\]
The corresponding diagram of fundamental groups is
\begin{equation*}
	\begin{tikzcd}[row sep = tiny]
		\bZ_{2}\times \bZ&\bZ_2\times \bZ \times \bZ\ar[l]\ar[r]& \bZ_2 \times \bZ \\
		(\overline{a}, b + c)&(\overline{a},b, c)\ar[l, mapsto]\ar[r, mapsto] & (\overline{a}+\overline{b}, c).
	\end{tikzcd}
\end{equation*}
One checks that the respective subgroups defining the $\Spin^{\bC}$-groups have the
same preimage in $\bZ_2\times \bZ\times \bZ$, consisting of those triples $(\overline{a}, b , c)$ for which $\overline{a} + \overline{b} + \overline{c} = 0 $ in $\bZ_2$. 

Letting $\widetilde{\G(m,k)\times \bC^*}\to \G(m,k)\times \bC^*$ denote the associated double covering group, we therefore do have a diagram
\begin{equation}\label{eq:diag-groups-extended}
	\begin{tikzcd}
		\Spin^{\bC}(m-2k)\ar[d,"2:1"]& \widetilde{\G(m,k)\times \bC^*} \ar[l,dashed]\ar[r,dashed]\ar[d,dashed,"2:1"]& \Spin^{\bC}(m)\ar[d,"2:1"] \\
		\SO(m-2k) \times \bC^* &\G(m,k)\times \bC^*\ar[l]\ar[r] & \SO(m)\times \bC^*.
	\end{tikzcd}
\end{equation}

\paragraph{Transfer of spin structures.}
Given an isotropic reduction from $E^{\bullet}$ to $F^{\bullet}$, with
\[F^{\bullet} = [A_{-1}\to F\to A^1],\] we have a natural isomorphism 
\[\det A^1 \simeq \det B^1\otimes  \det K \] using which we can extend \eqref{eq:torsor-diag} to a diagram of principal bundles over \eqref{eq:group-diag-twisted}
\begin{equation}\label{diag:principal-bundles-pairs}
	P_{F}^{\SO}\times P_{\det A^1}^{\bC^*} \leftarrow P_{E,K}^{\G}\times P_{\det B^1}^{\bC^*}\to P_{E}^{\SO}\times P_{\det B^1}^{\bC^*}
\end{equation}
Now, a spin structure on $E^{\bullet}$ gives rise to the diagram of solid arrows
\begin{equation}\label{diag:spin-structure-transfer}
	\begin{tikzcd}
		P_{\tau}^{\Spin^{\bC}}\ar[d, dashed, "2:1"]&\tilde{P} \ar[l, dashed]\ar[r,dashed]\ar[d,dashed, "2:1"]& P_{\sigma}^{\Spin^{\bC}}\ar[d,"2:1"] \\
		P_F^{\SO}\times P_{\det A^1}^{\bC^*}&P_{E,K}^{\G}\times P_{\det B^1}^{\bC^*}\ar[l]\ar[r] & P_E^{\SO}\times P_{\det B^1}^{\bC^*}.
	\end{tikzcd}
\end{equation}
By base change, and extension of structure group respectively, we obtain the terms $\tilde{P}$ and $P_{\tau}^{\Spin^{\bC}}$ as well as the dashed arrows, making this a diagram of principal bundles over \eqref{eq:diag-groups-extended}.
The reverse procedure produces a spin structure on $E^{\bullet}$, starting from one on $F^{\bullet}$. 

Regarding an isotropic reduction as a morphism $\xi:F^{\bullet}\rightsquigarrow E^{\bullet}$ in the category $\mathcal{R}$, we denote the spin structure on $F^{\bullet}$ obtained by the above construction by $\xi^*\sigma$. 

\paragraph{Spin structures on quadratic complexes, affine case.}
The collection of spin structures on a self-dual representative $E^{\bullet}$ forms a groupoid, denoted $\bS(E^{\bullet})$. Let $\Grpd$ denote the category of groupoids, and $\Grpd^{\simeq}$ its sub-category of equivalences between groupoids. 
\begin{theorem}[cf. Definition \ref{def:spin-functor}]
	There is a natural contravariant pseudo-functor $\bS:\mathcal{R}\to \Grpd$ that sends an object $E^{\bullet}$ to $\bS(E^{\bullet})$ and an isotropic reduction $\xi:F^{\bullet}\to E^{\bullet}$ to $\xi^*$. Moreover, it factors through the subcategory $\Grpd^{\simeq}$.   
\end{theorem}

Via the Grothendieck construction \cite[Part 1, \S 3.1]{FGAE}, the pseudo-functor $\bS$ determines a category $\mathcal{S} = \mathcal{S}_{\bE}$, fibered in groupoids over $\mathcal{R}$, whose fiber over an object $\bE$ is the groupoid $\bS(E^{\bullet})$. 

This allows us to give a definition of spin structure on a perfect complex without picking out a single representative.
\begin{definition}\label{def:spin-str-indep} 
	A \emph{spin structure on the oriented quadratic complex $\bE$} is a section of the functor $\mathcal{S}\to \mathcal{R}$. 
\end{definition}
Effectively, a spin structure on $\bE$ is a choice of spin structure $\sigma_{E^{\bullet}}$ on \emph{each} representative $E^{\bullet}$ together with identifications $\xi^*\sigma_{E^{\bullet}}\xrightarrow{\sim} \sigma_{F^{\bullet}}$ for each isotropic reduction, satisfying natural compatibilities. 
The collection of spin structures on $\bE$ forms a groupoid, which we denote by $\bS(\bE)$. 

One main result of this paper is that Definitions \ref{def:spin-str-rep} and \ref{def:spin-str-indep} are compatible, i.e. locally nothing is lost by working with a single representative:
\begin{theorem}[cf. Corollary \ref{cor:restriction-iso}]
	For any choice of representative $E^{\bullet}$, the natural restriction map $\bS(\bE)\to \bS(E^{\bullet})$ that evaluates a section of $\mathcal{S}\to \mathcal{R}$ at $E^{\bullet}$ is an equivalence of groupoids. 
\end{theorem}

\paragraph{Spin structures on quadratic complexes}
Let now $\bE$ be an oriented quadratic complex on an algebraic stack $\mathcal{M}$.  Let $\Aff_{\mathcal{M}}$ denote the category of affine schemes over $\mathcal{M}$ whose objects are pairs $(X,x)$ where $X$ is an affine scheme and $x\in \mathcal{M}(X)$. To any $(x: X \to \mathcal{M})\in \Aff_{\mathcal{M}}$, we have the oriented quadratic complex $x^*\bE$ on $X$ and can form the associated categories $\mathcal{S}_{x^*\bE}\to \mathcal{R}_{x^*\bE}$. Since all constructions are compatible with pullbacks along maps of affine schemes, these assemble into fibered categories $\underline{\mathcal{S}}_{\bE}\to \underline{\mathcal{R}}_{\bE}$ over $\Aff_{\mathcal{M}}$ (cf. Definition \ref{def:R-cat-stack}, Definition \ref{def:S-cat-stack} and Lemma \ref{lem:fibered-cats-stacks}). 
\begin{definition}\label{def:spin-str-stack-intro}
	A \emph{spin structure} on $\bE$ is a section of the forgetful map $\underline{\mathcal{S}}_{\bE}\to \underline{\mathcal{R}}_{\bE}$. 
\end{definition}
If the stack $\mathcal{M}$ is the functor of points of an affine scheme, one can show that this agrees with Definition \ref{def:spin-str-indep} up to natural equivalences (Theorem \ref{thm:spin-str-defs-compatible}). 

Using that Definition \ref{def:spin-str-stack-intro} makes sense for any scheme and behaves well under pullback, we obtain a new category fibered over schemes $\mathcal{M}^{S}$, whose $T$-points consist of pairs $(t,\sigma)$, where $t:T\to \mathcal{M}$, and where $\sigma\in \bS(t^*\bE)$. The category $\mathcal{M}^{S}$ parametrizes spin structures on $\bE$. We have
\begin{theorem}[Proposition \ref{prop:spin-gerbe} \& Corollary \ref{cor:spin-gerbe}]
	The category $\mathcal{M}^{S}$ fibered over schemes is an algebraic stack. The forgetful map $\mathcal{M}^{S}\to \mathcal{M}$ exhibits $\mathcal{M}^{S}$ as a $\ZT$-gerbe over $\mathcal{M}$.  
\end{theorem}

\subsection{Notation and Conventions}
We will use symbols $A^{\bullet}, B^{\bullet}, C^{\bullet}$ to denote complexes of quasi-coherent sheaves on a scheme or algebraic stack. Except where stated otherwise, all complexes will be bounded complexes of locally free sheaves. We reserve letters $E^{\bullet}, F^{\bullet}, G^{\bullet}$ to denote \emph{self-dual} complexes as defined in the beginning of \S \ref{sec:self-dual-cat}.

If $A^{\bullet}$ is a complex of locally free sheaves, we use the notation $A_{\bullet}$ to denote the dual of $A^{\bullet}$, with $i$-th entry $A_i:= A^{-i}$. We define the differential of $A_{\bullet}$ to be the dual of the one on $A^{\bullet}$, \emph{without} intervention of signs. This deviates from common conventions, such as described in \cite[\href{https://stacks.math.columbia.edu/tag/0FNG}{Section 0FNG}]{stacks-project}. For a complex $A^{\bullet}$, we let $A_+^{\bullet}$ (resp. $A_-^{\bullet}$) denote its bad truncation in degrees $\geq 1$ and by $A_-^{\bullet}$ its bad truncation in degrees $\leq -1$.
For sign conventions regarding determinants, we follow \cite[p. 8 \& pp. 27-28]{ot_counting}.

We use $\dercat(\mathcal{X})$ to denote the derived category of quasi-coherent sheaves on a stack $\mathcal{X}$ -- mostly we will only be dealing with the full subcategory of perfect complexes.

To address set-theoretic issues, we follow the approach of the stacks project by fixing an appropriately large set of schemes, and similarly for each affine scheme in this set a set of finite locally free sheaves on them.  

Throughout, all stacks and schemes will be over $\Spec \bZ[1/2]$.  

\subsection{Acknowledgements}
The line of research undertaken here was sparked by discussions with Felix Thimm regarding the possibility of a DT4 virtual structure sheaf, and by a presentation of Martijn Kool at Oslo STEW 2023. Our results can be seen as developing ideas already present in \cite{ot_counting} and, particularly, \cite{kr_mag}. 
We thank Nick Rozenblyum for discussions about spin structures and derived Clifford algebras, and for pointing out Conjecture \ref{conj:der-stack}. 
We thank Henry Liu for helpful discussions, particularly regarding $K$-groups of spaces, and Nikolai Opdan and Håkon Kolderup for discussions about functors of infinity-categories.  
Throughout, and particularly in the beginning stages of this project, we have benefitted immensely from frequent conversations with Jørgen Rennemo, who contributed several suggestions putting the constructions undertaken here in a more conceptual light. We are particularly grateful for suggesting Question \ref{que:spin} as a guiding principle. 

The author further thanks Kavli IPMU, Academia Sinica and the Lorentz Center for their hospitality in Summer 2024 while part of this paper was being written. 
This research was funded by Research Council of Norway grant number 302277 - ”Orthogonal gauge duality and non-commutative geometry”.
\section{The category of self-dual representatives}\label{sec:self-dual-cat}

Let $X$ denote an affine scheme and let $(\bE, \theta)$ be a quadratic complex on $X$. \footnote{More generally, one may assume $X$ to be quasi-projective over some affine scheme.}

Following Oh--Thomas \cite{ot_counting}, we say that finite length complex of locally free sheaves $E^{\bullet}$ is \emph{self-dual} if it is equipped with a map $\theta_{\bullet}:E^{\bullet}\to E_{\bullet}$ to its dual, satisfying $\theta_{\bullet}^{\vee} = \theta_{\bullet}$ under the identification $E_{\bullet}^{\vee}  \simeq E^{\bullet}$. Equivalently, a self-dual complex is given by a quadratic bundle $(E,q)$ and a finite length complex of vector bundles $A^{\bullet}$ in degrees $\geq 1$, together with a map $a:E\to A^{1}$ satisfying $d\circ a = 0$ and $a\circ a^{\vee} = 0$. Then one recovers a self-dual complex $E^{\bullet}$ as
\[E^{\bullet}:= [A_{\bullet}\xrightarrow{a^{\vee}} E \xrightarrow{a} A^{\bullet}].\] 
\begin{definition}\label{def:self-dual}
 A \emph{self-dual representative} of $\bE$ consists of a self-dual complex $(E^{\bullet}, \theta_{\bullet})$ together with an isomorphism $\alpha_E:E^{\bullet}\to \bE$ in the derived category, such that $\theta = \alpha_E\circ [\theta_{\bullet}]\circ \alpha_E^{-1}$.
\end{definition}
In \cite[Proposition 4.1]{ot_counting}, Oh--Thomas show that -- at least if $\bE$ has amplitude in $[-1,1]$ -- one can represent $(\bE, \theta)$ by a self-dual complex, and go on to show that any two self-dual representatives can be related via a notion of \emph{isotropic reduction} -- at least if one identifies complexes that are related via a deformation over $\bA^1$. 

In this section, we significantly elaborate on their results: Define a category $\mathcal{R}$ whose objects are self-dual representatives of $(\bE,\theta)$ and whose morphisms are isotropic reductions. In \S \ref{sec:symm-rep}, we define a splicing construction and use it to show that self-dual representatives exist (i.e. $\mathcal{R}$ is non-empty) without any boundedness assumption. In \S \ref{sec:iso-red}, we define isotropic reductions in general (Lemma \ref{def:iso-red}) and show that any two representatives are connected via isotropic reductions, i.e. the category $\mathcal{R}$ is connected (Corollary \ref{cor:repcat-connected}). 

For our purpose of defining spin structures, we need a certain ``higher'' connectedness result for $\mathcal{R}$. In \S \ref{sec:families}, we show that given an arbitrary family of isotropic reductions, one can always, in a certain sense, ``contract'' it, which for us means find a deformation over $\bA^1$ to a constant family.   
 In \S \ref{sec:simpl-struc}, we use the notion of family of isotropic reductions to define a simplicial enrichment $\mathcal{R}^{\Delta}$ of $\mathcal{R}$ and show that this enriched category is weakly contractible.

\subsection{Existence of symmetric representatives} \label{sec:symm-rep}

Since $\bE$ is a perfect complex on a quasi-projective scheme, we may represent $\bE$ by a finite-length complex of locally free sheaves 
\[A^{\bullet}= [A^{-m}\to \cdots \to A^{-1}\to A^0\to A^1\to  \cdots\to A^m]\] for some $m\geq  0$.

Similarly, we may represent the map $\theta: \bE\to \bE^{\vee}$ on the level of complexes by a roof 
\[A^{\bullet}\xleftarrow{\varphi_{\bullet}} B^{\bullet}\xrightarrow{\theta'_{\bullet}} A_{\bullet},\] where $B^{\bullet}$ is again a finite length complex of locally free sheaves with the indicated quasi-isomorphism to $A^{\bullet}$. We are free to replace $A^{\bullet}$ by $B^{\bullet}$ as our chosen representative, and one checks that $\varphi^{\vee}_{\bullet}\circ \theta'_{\bullet}$ is a representative of $\theta$. 
Thus, without loss of generality, we may assume that we have chosen $A^{\bullet}$ together with a map of complexes $\theta_{\bullet}:A^{\bullet}\to A_{\bullet}$ representing $\theta$. Replacing $\theta_{\bullet}$ by $(\theta_{\bullet}+\theta_{\bullet}^{\vee})/2)$ if necessary, we may assume that $\theta_{\bullet}$ is \emph{self-dual}, i.e. that $\theta_i^{\vee} = \theta_{-i}$.

Since $\theta_{\bullet}$ represents the isomorphism $\theta$ in $\dercat(X)$, it is a quasi-isomorphism. 
We have arranged a quasi-isomorphism of complexes
\begin{equation}\label{eq:self-dual-representative-splicing-setup}
	\begin{tikzcd}
		\cdots\ar[r]&A^{-1}\ar[r]\ar[d, "\theta_{1^{\vee}}"]&A^0\ar[r]\ar[d, "\theta_0 = \theta_0^{\vee}"]& A^1\ar[r]\ar[d, "\theta_1"]& \cdots \\
		\cdots\ar[r]&A_{-1}\ar[r]&A_0\ar[r] & A_{1}\ar[r] & \cdots
	\end{tikzcd}
\end{equation}

As in \cite[Eq.~(52)]{ot_counting}, we obtain a self-dual representative from \eqref{eq:self-dual-representative-splicing-setup} by conjoining the positive entries of the upper and negative entries of the lower row along the $0$-th term, (see Construction \ref{constr:self-dual-complex}). This uses a \emph{splicing} construction, which we now introduce more generally.

\paragraph{Splicing. }
Let $f_{\bullet}:A^{\bullet}\to B^{\bullet}$ be a quasi-isomorphism of complexes (for now in any abelian category). We construct a new complex $A_f^{\bullet}$, together with quasi-isomorphisms \[A^{\bullet}\xrightarrow{\eta^-_{\bullet}(f)} A_f^{\bullet}\xrightarrow{\eta^+_{\bullet}(f)} B^{\bullet}\] factoring $f_{\bullet}$, as follows:   
\begin{construction}\label{constr:splicing}
	Since $f_{\bullet}$  is a quasi-isomorphism, the associated mapping cone $\cone(f_{\bullet})$ is acyclic. In particular, the sequence 
	\begin{equation}\label{eq:cone-exact}
		B^{-2}\oplus A^{-1} \xrightarrow[c_{-1}(f)]{\begin{bmatrix}
				d& -f_{-1}\\0 &d 
		\end{bmatrix}} B^{-1}\oplus A^0\xrightarrow[c_0(f)]{\begin{bmatrix}
				d & f_0 \\ 0 & d	
		\end{bmatrix}} B^0\oplus A^1 \xrightarrow[c_1(f)]{\begin{bmatrix}
				d & -f_1 \\ 0 & d
		\end{bmatrix}} B^1\oplus A^{2}
	\end{equation}
	is exact (the maps $c_i(f)$ are the differentials in $\cone(f_{\bullet})$ up to some sign changes). This implies that $c_0$ descends to an isomorphism
	
	\begin{equation}\label{eq:coker-ker-isom}
		\operatorname{Coker}\left(B^{-2}\oplus A^{-1}\xrightarrow{c_{-1}(f)}B^{-1}\oplus A^0\right) \xrightarrow{\overline{c_0(f)}} \operatorname{Ker}\left(B^0\oplus A^1\xrightarrow{c_1(f)}B^1\oplus A^2\right) = \colon A_f^0.
	\end{equation}
	
	By construction, $A_f^0$ comes with maps 
	\[ B^{-1} \to  A_f^0 \to A^1\]
	induced by the inclusion $B^{-1}\to B^{-1}\oplus A^0$ and projection $B^0\oplus A^1\to A^1$.
	Similarly, we have maps
	\begin{equation}\label{eq:f_zero_fac}
		A^0 \xrightarrow{f^-_0} A^0_f\xrightarrow{f^+_0} B^0
	\end{equation}
	factoring $f_0$.
	
	We form the complex $A_f^{\bullet}$ as
	\begin{equation}
		A_f^{\bullet}:= [\cdots \to B^{-2}\to B^{-1}\to E_f\to A^1\to A^2\to \cdots].
	\end{equation}
	This is indeed a complex: The compositions $B^{-2}\to B^{-1}\to A_f^0$ and $A^0_f\to A^1\to A^2$ factor through $c_{-1}(f)$ and $c_1(f)$ respectively, hence are zero by construction. The composition $B^{-1}\to A^1$ is the lower left coordinate of the matrix defining $c_0(f)$ in \eqref{eq:cone-exact}, which is zero.
	
	Define the maps $\eta^-_{\bullet}(f): A^{\bullet}\to A_f^{\bullet}$ and $\eta^+_{\bullet}(f):A_f^{\bullet}\to B^{\bullet}$ as in the following diagram:
	
	\begin{equation*}
		\begin{tikzcd}
			\cdots\ar[r]&A^{-1}\ar[r]\ar[d, "f_{-1}"]&A^0\ar[r]\ar[d,"f_0^-"]& A^1\ar[r]\ar[d, equals]& \cdots \\
			\cdots\ar[r]&B^{-1}\ar[r]\ar[d, equals]&A_f^0\ar[r]\ar[d, "f_0^+"]& A^1\ar[r]\ar[d, "f_1"]& \cdots \\
			\cdots\ar[r]&B^{-1}\ar[r]&B^0\ar[r] & B^{1}\ar[r] & \cdots
		\end{tikzcd}
	\end{equation*} 
	Using \eqref{eq:cone-exact}, one checks that $\eta^-_{\bullet}(f), \eta^+_{\bullet}(f)$ indeed define maps of complexes and, using \eqref{eq:f_zero_fac}, that  
	\begin{equation}\label{eq:etas-factor-f}
		f_{\bullet} = \eta^+_{\bullet}(f)\circ \eta^-_{\bullet}(f).
	\end{equation} To see that they are quasi-isomorphisms, we only need to check in degrees $-1,0$ and $1$.  It suffices to check for, say, $\eta^-_{\bullet}(f)$. Writing out the mapping cone and writing $A_f^0=\operatorname{Coker}(c_{-1}(f))$, one has
	\[\cdots \to B^{-2}\oplus A^{-1}\xrightarrow{c_{-1}(f)} B^{-1}\oplus A^0 \to \operatorname{Coker}(c_{-1}(f))\oplus A^1 \to A^1\oplus A^2 \to A^2\oplus A^3 \to \cdots\]
	which one checks to be exact by a direct calculation. 
\end{construction} 

For our purposes, we are working with finite length complexes of locally free sheaves. This is preserved by splicing:
\begin{lemma}
	Suppose that $A^{\bullet}$ and $B^{\bullet}$ in Construction \ref{constr:splicing} are finite length complexes of vector bundles. Then so is $A_f^{\bullet}$.  
\end{lemma}
\begin{proof}
	The only thing to be checked is that the middle term $A_f^0$ is a vector bundle. But $A_f^0 = \operatorname{Ker} c_{1}(f)$, where $c_1(f)$ can be identified with a differential in the acyclic complex $\cone(f_{\bullet})$. By our assumptions, $\cone(f)$ is itself a finite length complex of vector bundles. In particular, the kernel of any of its differentials $d^{C}_i$ sits in a long exact sequence in which all other terms are locally free:
	\[0\to \operatorname{Ker} d_i^C\to C_i\to C_{i+1}\to \cdots \to C_M\to 0.\]
	It follows form this that $\Ker d_i^C$ has vanishing Tor-dimension, hence is locally free. In particular, this applies ot $A_f^0$.  
\end{proof}

We record the following properties of Construction \ref{constr:splicing}:
\begin{lemma}\label{lem:splicing-properties}
	\begin{enumerate}[label = \roman*)]
		\item \label{lem:splicing-propertiesi} The construction of $A_f^{\bullet}, \eta^-_{\bullet}(f)$ and $\eta^+_{\bullet}(f)$ is functorial in $f_{\bullet}$.
		\item \label{lem:splicing-propertiesii} If $A^{\bullet}$ and $B^{\bullet}$ are complexes of locally free sheaves, we have a natural identification 
		\begin{equation}\label{eq:duality-isom}
			A_{f^{\vee}}^{\bullet} \xrightarrow{\sim} (A^{\bullet}_{f})^{\vee}.
		\end{equation} 
		It makes the following diagram commute
		\begin{equation}\label{eq:etas-diagram}
			\begin{tikzcd}[row sep = small, column sep = large]
				&A_{f^{\vee}}^{\bullet} \ar[dr, "\eta_{\bullet}^+(f^{\vee})"]\ar[dd,"\sim"]&  \\
				B_{\bullet}\ar[ur, "\eta_{\bullet}^-(f^{\vee})"]\ar[dr,"(\eta_{\bullet}^+(f))^{\vee}"']& &A_{\bullet}\\
				&(A_f^{\bullet})^{\vee}\ar[ur, "(\eta_{\bullet}^-(f))^{\vee}"'] & \,
			\end{tikzcd}
		\end{equation}
		
		\item \label{lem:splicing-propertiesiii} The homomorphism $A_{f}^{\bullet}\to (A_{f^{\vee}}^{\bullet})^{\vee}$ obtained by applying \eqref{eq:duality-isom} to $f_{\bullet}^{\vee}$ in place of $f_{\bullet}$ is the same as the one obtained by dualizing \eqref{eq:duality-isom} applied to $f_{\bullet}$ itself. 
	\end{enumerate}
\end{lemma}
\begin{proof}
	Point \ref{lem:splicing-propertiesi} is straightforward from the construction. 
	In \ref{lem:splicing-propertiesii}, the map \eqref{eq:duality-isom} is given by identity morphisms in degrees $\neq 0$. To define the map $A^0_{f^{\vee}}\to (A_f^0)^{\vee}$, note that dualizing the sequence \eqref{eq:cone-exact} and identifying $(B^{i} \oplus A^{i+1})^{\vee}\simeq A_{-i-1} \oplus B_{-i}$ recovers $\eqref{eq:cone-exact}$ with $f_{\bullet}$ replaced by $f_{\bullet}^{\vee}$. 
	
	This isomorphism descends to \eqref{eq:coker-ker-isom}, resulting in a commutative diagram
	\begin{equation*}
		\begin{tikzcd}
			(A_f^{0})^{\vee}\ar[r,equals]&\Ker(c_1(f))^{\vee}\ar[r,"\overline{c_0(f)}^{\vee}"]\ar[d, "\sim"]& \operatorname{Coker}(c_{-1}(f))^{\vee}\ar[d, "\sim"]& \\
			&\operatorname{Coker}(c_{-1}(f^{\vee}))\ar[r, "\bar{c_0(f^{\vee})}"] & \Ker(c_1(f^{\vee}))\ar[r,equals]&A_{f^{\vee}}^{0}
		\end{tikzcd}
	\end{equation*}
	The induced isomorphism $A_{f^\vee}^0\to (A_f^{0})^{\vee}$ is the one we want. 
	Moreover, since dualizing this diagram again transposes it, point \ref{lem:splicing-propertiesiii} follows immediately. 
	Commutativity of \eqref{eq:etas-diagram} can again be checked in degree zero and is straightforward.  
\end{proof}

We can now finish the construction of a self-dual representative of a quadratic complex
\begin{construction}\label{constr:self-dual-complex}
	Starting from the self-dual map $\theta_{\bullet}: A^{\bullet}\to A_{\bullet}$ as in \eqref{eq:f_zero_fac}, form the complex $A_{\theta}^{\bullet}$, which comes with the quasi-isomorphisms 
	\[A^{\bullet}\xrightarrow{\eta_{\bullet}^-(\theta) } A_{\theta}^{\bullet} \xrightarrow{\eta^+_{\bullet}(\theta)}A_{\bullet}.\]
	Since $\theta_{\bullet}^{\vee} = \theta_{_{\bullet}}$, the morphism \eqref{eq:duality-isom} of Lemma \ref{lem:splicing-properties} \ref{lem:splicing-propertiesii} provides a map \[\theta'_{\bullet}:A_{\theta}^{\bullet}\to A_{\theta}^{\bullet},\]
	which by \ref{lem:splicing-propertiesiii} of that lemma is self-dual. Thus, $A_{\theta}^{\bullet}$ has the structure of self-dual complex. To see that this is compatible with the map $\theta_{\bullet}$, note that the outer edges of diagram \eqref{eq:etas-diagram} compose to $\theta_{\bullet}$, which yields the commutative diagram   
	\begin{equation}\label{eq:self-dual-compatible}
		\begin{tikzcd}
			A^{\bullet}\ar[r,"\eta_{\bullet}^-(\theta)"]\ar[d, "\theta_{\bullet}"]& A_{\theta}^{\bullet}\ar[d, "\theta_{\bullet}'"] \\
			A_{\bullet}\ar[r, leftarrow, "(\eta^-_{\bullet}(\theta))^{\vee}"] & (A_{\theta}^{\bullet})^{\vee}.
		\end{tikzcd}
	\end{equation}
\end{construction}

\paragraph{More on splicing.}
In Construction \ref{constr:splicing}, it suffices to specify $A^{\bullet}$ and $f^{\bullet}$ in positive degrees only (or alternatively, $B^{\bullet}$ and $f^{\bullet}$ in negative degrees). This gives rise to the following stronger variant, which we will use repeatedly.
\begin{construction}\label{constr:splicing-truncated}
	Let $f^+_{\bullet}:A_+^{\bullet}\to B^{\bullet}$ be a morphism of complexes, where $A^{\bullet}$ is concentrated in degrees $\geq 1$. Suppose that $f^+_{\bullet}$ induces isomorphisms on cohomology sheaves $h^i$ for $i\geq 2$ and a surjection on $h^1$. 
	Define 
	\[A_f^0 := \Ker c_1,\]
	where 
	\[c_1= B^0\oplus A^1 \xrightarrow{\begin{bmatrix}
			d & -f_1 \\ 0 & d
	\end{bmatrix}} B^1\oplus A^2.\]
	and consider the maps 
	\[B^{-1}\xrightarrow{(d,0)} A_f^0 \xrightarrow{\operatorname{pr}_2} A^1.\]
	We define the complex $A_f^{\bullet}$ as 
	\[\cdots \to B^{-2}\to B^{-1}\to A_f^0\to A^1\to A^2\to \cdots \]
	It comes with a natural quasi-isomorphism of complexes
	
	\begin{equation*}
		\begin{tikzcd}
			\cdots\ar[r]&B^{-1}\ar[r]\ar[d, equals]&A^0_f\ar[r]\ar[d,"\operatorname{pr}_1"]& A^1\ar[r]\ar[d, "f_1"]& \cdots \\
			\cdots\ar[r]&B^{-1}\ar[r]&B^0\ar[r] & B^{1}\ar[r] & \cdots
		\end{tikzcd}
	\end{equation*} 
\end{construction}

As before, if $A^{\bullet}$ and $B^{\bullet}$ are bounded complexes of vector bundles, then so is $A_f^{\bullet}$, and the construction is functorial in $f_{\bullet}^+$. 

To conclude, we prove some lemmas regarding how splicing behaves with respect to chain homotopies. For this, let $f_{\bullet},g_{\bullet}:A^{\bullet}\to B^{\bullet}$ be two maps that are related by a Chain homotopy $h_{\bullet}:A^{\bullet}\to B^{\bullet}[-1]$ (note that $h_{\bullet}$ is \emph{not} a map of complexes), in the sense that for all $i$, we have $f_i = g_i + d h_i + h_{i+1} d$. We write this as $f_{\bullet} = g_{\bullet} + d_{\bullet} h_{\bullet} + h_{\bullet} d_{\bullet}$. Assume that either $f_{\bullet}, g_{\bullet}$ are quasi-isomorphisms or that they satisfy the assumptions of Construction \ref{constr:splicing-truncated}.

In either situation, the homotopy $h_{\bullet}$ induces an isomorphism $\alpha^h_{\bullet}:A_f^{\bullet}\to A_g^{\bullet}$: 
\begin{construction}\label{constr:homotopy-isom}
	Define a map of complexes $\alpha^h_{\bullet}:A_f^{\bullet}\to A_g^{\bullet}$ with components $\alpha^h_i$ as follows:
	\begin{itemize}
		\item In degrees $i\neq 0$, define $\alpha_i^h$ to be the identity on $A^i$ ($i>0$) and $B^i$ ($i<0)$ respectively.
		\item To define $\alpha^h_0$, note that the chain homotopy $h_{\bullet}$ induces a map of cones
		\[\cone(f_{\bullet})\to \cone(g_{\bullet}),\]
		and hence a map $A_f^0\to A_g^0$ of their sheaves of zero-cycles. We take $\alpha_0^h$ to be this map. To make the signs completely explicit, we have
		
		\begin{align*}
			A_f^0 &= \Ker \left(\begin{bmatrix}
				d & -f_1 \\ 0 & d
			\end{bmatrix}: B^0\oplus A^1\to B^1\oplus A^2\right);\\
			\quad A_g^0 &= \Ker \left(\begin{bmatrix}
				d & -g_1 \\ 0 & d
			\end{bmatrix}: B^0\oplus A^1\to B^1\oplus A^2\right)
		\end{align*}
		and $\alpha^h_0$ is the restriction of
		\[B^0\oplus A^1\xrightarrow{\begin{bmatrix}
				\id & h_1 \\0 &\id 
		\end{bmatrix}} B^0\oplus A^1.\] 
		Clearly, $\alpha_0^h$ is an isomorphism. 
		
	\end{itemize} 
	
	One checks that $\alpha_{\bullet}^h$ defines a map of complexes, hence an isomorphism of complexes.  
	
	Further, there are natural homotopies for the following diagram 
	\begin{equation*}
		\begin{tikzcd}[row sep = small]
			&A_f^{\bullet}\ar[dr]\ar[dd,"\alpha_h" {yshift=0.5ex}]&  \\
			A^{\bullet}\ar[ur]\ar[dr]\ar[r, Rightarrow, shorten <=1.5ex, shorten >=1.5ex]& \ar[r, Rightarrow, shorten >=1.5ex, shorten <=1.5ex]&B^{\bullet}\\
			&A_g^{\bullet}\ar[ur] & \,
		\end{tikzcd}
	\end{equation*}
	These are, respectively,
	\begin{equation*}
		\begin{tikzcd}
			\cdots\ar[r]&A^{-1}\ar[r]\ar[dl, "h_{-1}"]&A^0\ar[r]\ar[dl,"h_0"]& A^1\ar[r]\ar[dl, "0" ]& \cdots \ar[dl,"0"]\\
			\cdots\ar[r]&B^{-1}\ar[r]&A_g^0\ar[r] & A^{1}\ar[r] & \cdots
		\end{tikzcd}
	\end{equation*} 
	and 
	\begin{equation*}
		\begin{tikzcd}
			\cdots\ar[r]&B^{-1}\ar[r]\ar[dl, "0"]&A_f^0\ar[r]\ar[dl,"0"]& A^1\ar[r]\ar[dl, "h_1" ]& \cdots \ar[dl,"h_2"]\\
			\cdots\ar[r]&B^{-1}\ar[r]&B^0\ar[r] & B^1\ar[r] & \cdots
		\end{tikzcd}
	\end{equation*} 
	
\end{construction}

\begin{lemma}\label{lem:dual-homotopies}
	Suppose that $A^{\bullet}$ and $B^{\bullet}$ are bounded complexes of vector bundles. Denote $h_{\bullet}^{\vee}$ the dual homotopy of $h_{\bullet}$, whose term at degree $i$ is $h_{1-i}^{\vee}$. Then, under the identifications $(A_f^{\bullet})^{\vee}\simeq A_{f^{\vee}}^{\bullet}$ and $(A_g^{\bullet})^{\vee}\simeq A_{g^{\vee}}^{\bullet}$ given by Lemma \ref{lem:splicing-properties} \ref{lem:splicing-propertiesii}, we have $(\alpha^h_{\bullet})^{\vee} = \alpha^{h^{\vee}}_{\bullet}$. In other words, the following diagram commutes
	\begin{equation*}
		\begin{tikzcd}
			(A_{f}^{\bullet})^{\vee}\ar[d]& (A_g^{\bullet})^{\vee}\ar[l,"(\alpha^h_{\bullet})^{\vee}"']\ar[d] \\
			A_{f^{\vee}}^{\bullet}\ar[r, "\alpha_{\bullet}^{h^{\vee}}"] & A_{g^{\vee}}^{\bullet}\,
		\end{tikzcd}
	\end{equation*} 
\end{lemma}
\begin{proof}
	All maps are term-wise identities except in degree $0$. Thus, we are reduced to checking commutativity in degree $0$. By construction of the vertical isomorphisms, the isomorphism $(A_f)^{\vee}\to A_{f^{\vee}}$ is induced from the map  
	\[B_0\oplus A_{-1}\xrightarrow{\begin{bmatrix}
			d^{\vee} & 0 \\ f_0^{\vee} & d^{\vee}
	\end{bmatrix}} B_1\oplus A_0,\]
	and similarly for $f$ replaced by $g$. The claimed commutativity then follows from  commutativity of the following diagram 
	\begin{equation*}
		\begin{tikzcd}[ampersand replacement=\&, row sep = large]
			B_0\oplus A_{-1}\ar[d, "{\begin{bmatrix}
					d^{\vee} & 0 \\ f_0^{\vee} & d^{\vee}
			\end{bmatrix}}"']\& B_0\oplus A_{-1} \ar[d, "{\begin{bmatrix}
					d^{\vee} & 0 \\ g_0^{\vee} & d^{\vee}
			\end{bmatrix}}"]\ar[l, "{\begin{bmatrix}
					\id & 0 \\ h_1^{\vee} & \id
			\end{bmatrix}}"'] \\
			B_1\oplus A_0\ar[r,"{\begin{bmatrix}
					\id & 0 \\ h_0^* & \id
			\end{bmatrix}}"]\& B_1\oplus A_0,
		\end{tikzcd}
	\end{equation*}

	It reduces to the equation $g_0^{\vee} =f_0^{\vee} + h_0^{\vee}\circ d^{\vee} + d^{\vee} \circ h_1^{\vee} $, which holds by assumption. 
	
\end{proof}

We will use the following basic result several times
\begin{lemma}\label{lem:isomorphic}
	Let $f_{\bullet}:A^{\bullet}\to B^{\bullet}$ be a quasi-isomorphism of complexes that is an isomorphism in each degree $i\neq 0$. Then $f_{\bullet}$ is an isomorphism of complexes.  
\end{lemma}
\subsection{Isotropic Reductions}\label{sec:iso-red}

We define morphisms in the category of representatives and show that it is connected. 
\begin{definition}\label{def:isotropic-red}
	Let $F^{\bullet}$ and $E^{\bullet}$ be self-dual representatives of $\bE$. An \emph{isotropic reduction} from $E^{\bullet}$ to $F^{\bullet}$ consists of a complex $A^{\bullet}$ of locally free sheaves, together with quasi-isomorphisms of complexes 
	\begin{equation}\label{eq:iso-red-diag}
		F^{\bullet} \xleftarrow{f_{\bullet}}A^{\bullet} \xrightarrow{e_{\bullet}} E^{\bullet},
	\end{equation}
	compatible with the isomorphisms to $\bE$, such that the following hold
	\begin{enumerate}[label = \roman*)]
		\item $f_{\bullet}$ is an isomorphism in degrees $\geq 1$ and $e_{\bullet}$ is an isomorphism in degrees $\leq -1$,\label{item:iso-redi}
		\item \label{item:iso-redii} the induced map $A^{1}\xrightarrow{(e_1,d)} E^1\oplus A^2$ is an inclusion of vector bundles (i.e. has surjective dual),
		\item the following diagram commutes:\label{item:iso-rediii}
		\begin{equation*}
			\begin{tikzcd}
				F^{\bullet}\ar[r, leftarrow, "f_{\bullet}"]\ar[d, "f_{\bullet}^{\vee}"]& A^{\bullet}\ar[d, "e_{\bullet}"] \\
				A_{\bullet}\ar[r, leftarrow ,"e_{\bullet}^{\vee}",] & E^{\bullet}.
			\end{tikzcd}
		\end{equation*} 
	\end{enumerate}
	
	We will indicate denote the data of an isotropic reduction from $E^{\bullet}$ to $F^{\bullet}$ by letters $\xi,\zeta$ and a squiggly arrow: \footnote{Throughout, we will silently identify isotropic reductions that are equivalent, in the sense that there is an isomorphism $A_1^{\bullet}\to A_2^{\bullet}$ compatible with the rest of the data.}
	\[\xi: F^{\bullet} \rightsquigarrow E^{\bullet}.\]
\end{definition}
Here, condition \ref{item:isotropic-constrii} guarantees that in degree zero, the diagram $F^0\leftarrow {A}^0 \to E^0$ is an isotropic reduction of quadratic bundles (Lemma \ref{lem:iso-red-gives-iso-red}).
To simplify some constructions, we will also consider a generalization, where one doesn't impose this condition (Variant \ref{var:generalized-reduction}).  

A large supply of isotropic reductions comes from the following result, which is a substantial strengthening of the technique in \cite[p. 33]{ot_counting}
\begin{construction}\label{constr:isotr-red}
	Let $E^{\bullet}$ be a self-dual representative of $\bE$. Let $A_+^{\bullet}$ be a complex of locally free sheaves concentrated in degrees $\geq 1$ and let $e^+_{\bullet}: A_{+}^{\bullet}\to E^{\bullet}$ be a morphism satisfying
	\begin{enumerate}[label = \roman*)]
		\item The map $e^+_{\bullet}$ is an isomorphism on cohomology sheaves $h^i$ for $i\geq 2$ and surjective on $h^1$. 
		\label{item:isotropic-constri}
		\item The induced map $A^1\xrightarrow{(e_1,d)}E^1\oplus A^2$ is an inclusion of vector bundles.\label{item:isotropic-constrii}
	\end{enumerate}
	
	We obtain a complex $A^{\bullet}:=A_{e}^{\bullet}$ together with a map $e^{\bullet}: A^{\bullet}\to E^{\bullet}$ via Construction \ref{constr:splicing-truncated}. 
	Applying Construction \ref{constr:self-dual-complex} to the map $e_{\bullet}^{\vee}\circ  e_{\bullet}: A^{\bullet}\to A_{\bullet}$, we obtain a self-dual complex $F_{e}^{\bullet}:= A_{e^{\vee}\circ e}^{\bullet}$, together with a map $f^e_{\bullet}:=\eta^-_{{\bullet}}(e^{\vee}\circ e)$.
	We claim that the maps  
	\[F^{\bullet}_e\xleftarrow{f_{\bullet}^e} A^{\bullet}\xrightarrow{e_{\bullet}}E^{\bullet}\]
	define an isotropic reduction (with the induced identification $F^{\bullet}_e\xrightarrow{\sim}\bE$).
	Indeed, properties \ref{item:iso-redi} and \ref{item:iso-redii} are immediate from the construction and our assumption \ref{item:isotropic-constrii}, while \ref{item:iso-rediii} follows from combining \eqref{eq:self-dual-compatible} and \eqref{eq:etas-factor-f}.

\end{construction}
\begin{remark}\label{rem:iso-red-comparison}
	Let $E^{\bullet}$ and $F^{\bullet}$ be self-dual representatives of $\bE$. 
	Giving an isotropic reduction from $E^{\bullet}$ to $F^{\bullet}$ is equivalent to giving a map of complexes $e^+_{\bullet}:A^{\bullet}_+ \to E^{\bullet}$ satisfying the conditions of Construction \ref{constr:isotr-red}, and an isomorphism of complexes $\varphi_{\bullet}:F_{e}^{\bullet}\to F^{\bullet}$ (compatible with the maps to $\bE$) satisfying
	\begin{enumerate}[label=\roman*)]
		\setcounter{enumi}{2}
		\item $\varphi_{\bullet}$ is an isomorphism of self-dual complexes, i.e. $\varphi_{\bullet}^{\vee} = \varphi_{\bullet}^{-1}$\label{item:isotropic-constriii}
	\end{enumerate}
	Indeed, since the notion of isotropic reduction is preserved under isomorphisms of self-dual complexes, taking $f_{\bullet}:=\varphi_{\bullet}\circ f^e_{\bullet}$ gives an isotropic reduction from $E^{\bullet}$ to $F^{\bullet}$.

	Conversely, given an isotropic reduction \eqref{eq:iso-red-diag}, we take $(A^{\bullet}_+, e^+_{\bullet})$ to be the restrictions of $A^{\bullet}$ and $e_{\bullet}$ to positive degrees. 
	Since $e_{\bullet}$ is an isomorphism in degrees $\leq -1$, one has that the induced map $A^{\bullet}\to A_{e}^{\bullet}$ is an isomorphism of complexes, and we may assume that $A^{\bullet} = A_{e}^{\bullet}$. Then, since $f_{\bullet}^{\vee}\circ f_{\bullet} = e_{\bullet}^{\vee}\circ e_{\bullet}$, we have induced maps of complexes
	\[F_e^{\bullet} = A^{\bullet}_{e^{\vee}\circ e}= A^{\bullet}_{f^{\vee}\circ f} \leftarrow A_{f}^{\bullet}\to F^{\bullet}.\]
	The two maps here are quasi-isomorphisms of complexes, and isomorphisms in degrees $\neq 0$, hence isomorphisms of complexes by Lemma \ref{lem:isomorphic}.
	Using Lemma \ref{lem:splicing-properties}, one checks that the induced isomorphism $\varphi_{\bullet}: F_e^{\bullet}\to F^{\bullet}$ satisfies \ref{item:isotropic-constriii}.
\end{remark}

\begin{variant}\label{var:generalized-reduction}
	We define a \emph{generalized isotropic reduction} by leaving out \ref{item:iso-redii} in Definition \ref{def:isotropic-red}. Then Construction \ref{constr:isotr-red} and Remark \ref{rem:iso-red-comparison} go through for generalized isotropic reductions if one drops property \ref{item:iso-redii} throughout. We will denote generalized isotropic reductions also as $F^{\bullet}\rightsquigarrow E^{\bullet}$ and make it clear which notion we are talking about. 
\end{variant}

\begin{example}\label{ex:reductions}
	\begin{enumerate}
		\item \label{item:reductionsi} Let $K^{\bullet}$ be an acyclic complex of locally free sheaves in degrees $\geq 0$, so that
		\[K^{\bullet}\oplus K_{\bullet}\]
		is a self-dual acyclic complex. 
		
		Then any self-dual complex $E^{\bullet}$ is naturally an isotropic reduction of $E^{\bullet}\oplus K^{\bullet}\oplus K_{\bullet}$, given by the maps \[ E^{\bullet}\leftarrow E^{\bullet}\oplus K_{\bullet}\to E^{\bullet}\oplus K_{\bullet}\oplus K^{\bullet}\]
		which are the natural projection and inclusion respectively. 
		
		\item \label{item:reductionsii} Let 
		\[\zeta: F^{\bullet}\xleftarrow{f_{\bullet}}A^{\bullet} \xrightarrow{e_{\bullet}}E^{\bullet}\]
		be a generalized isotropic reduction. Take $K^{\bullet}= [F^1\xrightarrow{\id} F^1]$ in degrees $[0,1]$. Then we have an induced isotropic reduction
		\begin{equation}\label{eq:make-reduction}
			\hat{\zeta}:F^{\bullet}\leftarrow A^{\bullet}\oplus K_{\bullet} \to E^{\bullet}\oplus K_{\bullet}\oplus K^{\bullet} 
		\end{equation}
		which in degrees $-1,0$ and $1$ respectively is given by 
		\begin{equation*}
			\begin{tikzcd}[column sep=large]
				F_{-1}& A^{-1}\oplus F_{-1} \ar[l,"{f_{-1}\oplus \id}"']\ar[r,"{e_{-1}\oplus \id }"]& E_{-1}\oplus F_{-1} \\
				F_0 & A^0\oplus F_{-1} \ar[l,"{f_0\oplus d}"'] \ar[r, "{(e_0, f_1\circ d)\oplus  \id}"]& E^0\oplus F^1 \oplus F_{-1} \\
				F^1 & A^1\ar[l,"f_1"']\ar[r,"{(e_1,f_1)}"] & {E^1\oplus F^1}
			\end{tikzcd}
		\end{equation*}
		\item \label{item:reductionsiii} Let $\zeta$ be as in the previous point, and let $F^{\bullet}_+$ be the bad truncation of $F$ in degrees $\geq 1$, and $F_-^{\bullet}$, the truncation in degrees $\leq -1$ (note that $(F^{\bullet}_+)^{\vee}\simeq F_-^{\bullet}$). Let $K^{\bullet}:=\cone(F_+^{\bullet}\xrightarrow{\id}F_+^{\bullet})$. Then we have an induced isotropic reduction 
		\[\widehat{\zeta}:F^{\bullet} \leftarrow A^{\bullet}\oplus K_{\bullet} \rightarrow E^{\bullet}\oplus K_{\bullet}\oplus K^{\bullet}.\]
		Here, the induced map $K_{\bullet}\simeq \cone(F_-^{\bullet}\xrightarrow{\id} F_-^{\bullet})[-1]\to F^{\bullet}$ is induced from the universal property of mapping cones by the map $F_-^{\bullet}[-1]\xrightarrow{d} F^{\bullet}$, which is just the restriction of the differential, and the homotopy $d\Rightarrow 0$ given by the inclusion $F_-^{\bullet}[-1]\to F^{\bullet}[-1]$. The map $A^{\bullet}\to K^{\bullet}$ factorizes as $A^{\bullet}\to F^{\bullet}\to K^{\bullet}$, where the latter is dual to the map $K_{\bullet}\to F^{\bullet}$ just constructed. 
		\item Any isomorphism of self-dual complexes $\varphi_{\bullet}:F^{\bullet}\to E^{\bullet}$ canonically gives rise to an isotropic reduction (which we will also denote $\varphi_{\bullet}$):
		\[F^{\bullet} \xleftarrow{=}F^{\bullet}\xrightarrow{\varphi_{\bullet}}E^{\bullet}\]
		
	\end{enumerate}
\end{example}

We come to the first key result of this section: Any two self-dual representatives can be related via isotropic reductions. 

\begin{proposition}[cf. Eq. (75) and thereafter in \cite{ot_counting}]\label{prop:iso-red}
	Let $E^{\bullet}$ and $F^{\bullet}$ be self-dual representatives of $\mathbb{E}$. 
	There exists a self-dual representative $G^{\bullet}$ of $\bE$, together with generalized isotropic reductions 
	\[ F^{\bullet}\leftsquigarrow G^{\bullet} \rightsquigarrow E^{\bullet}.\]
\end{proposition} 
In view of Example \ref{ex:reductions} points \ref{item:reductionsi} and \ref{item:reductionsii},  we obtain the following
\begin{corollary}\label{cor:repcat-connected}
	Any two self-dual representatives $F^{\bullet}$ and $E^{\bullet}$ are connected via a sequence of (genuine) isotropic reductions.
	\begin{equation*}
		\begin{tikzcd}[row sep = small]
			&    & G^{\bullet}\ar[dr, rightsquigarrow]\ar[dl, rightsquigarrow]&    &					   \\
			E^{\bullet } \ar[r, rightsquigarrow]& E'^{\bullet}     &  							 &  F'^{\bullet}  & F^{\bullet}\ar[l, rightsquigarrow]
		\end{tikzcd}
	\end{equation*}	
\end{corollary}

\begin{proof}[Proof of Proposition \ref{prop:iso-red}]
	Since $E^{\bullet}$ and $F^{\bullet}$ represent the same object $\bE$ in the derived category, we can find a bounded complex of vector bundles $A^{\bullet}$ and quasi-isomorphisms $e_{\bullet}:A^{\bullet}\to E^{\bullet}$ and $f_{\bullet}:A^{\bullet}\to F^{\bullet}$ such that $\alpha_F\circ [f_{\bullet}] = \alpha_E\circ [e_{\bullet}]$. It follows that $e_{\bullet}^{\vee}\circ e_{\bullet}$ and $f_{\bullet}^{\vee}\circ f_{\bullet}$ represent the same map $A^{\bullet}\to A_{\bullet}$ in the derived category. Therefore, we can  find a further bounded complex of vector bundles $B^{\bullet}$ together with a quasi-isomorphism $B^{\bullet}\xrightarrow{a_{\bullet}} A^{\bullet}$, such that $e^{\vee}_{\bullet}\circ e_{\bullet}\circ a_{\bullet}$ and $f_{\bullet}\circ f_{\bullet}\circ a_{\bullet}$ are chain homotopic. Consequently, the maps $a_{\bullet}^{\vee}\circ e_{\bullet}^{\vee}\circ e_{\bullet}\circ a_{\bullet}$ and $a_{\bullet}^{\vee}\circ f_{\bullet}^{\vee}\circ f_{\bullet}\circ a_{\bullet}$ are chain homotopic. Replacing $A^{\bullet}$ by $B^{\bullet}$, we may therefore assume without loss of generality that there exists a chain homotopy $h_{\bullet}$ relating $e_{\bullet}^{\vee}\circ e_{\bullet}$ and $f_{\bullet}^{\vee}\circ f_{\bullet}$, i.e. morphisms $h_i:A^i\to A_{i-1}$ satisfying \[e_{-i}^{\vee}\circ e_i = f_{-i}^{\vee}\circ f_i + d \circ h_i + h_{i+1}\circ d\] 
	By dualizing, we also have
	\[e_{-i}^{\vee}\circ e_i = f_{-i}^{\vee}\circ f_i + d \circ h^{\vee}_{-i+1} + h_{-i}^{\vee}\circ d.\]
	By replacing each $h_i$ by $(h_{i}+h_{-i+1})/2$, we may assume that the homotopy $h_{\bullet}$ is \emph{self-dual} in the sense that $h_{i}^{\vee} = h_{1-i}$ for all $i$.

	By Construction \ref{constr:homotopy-isom}, the homotopy $h_{\bullet}$ induces an isomorphism of complexes $\alpha_h:A_{e}\to A_{f}$. We claim that the diagram
	\begin{equation*}
		\begin{tikzcd}
			A_e^{\bullet}\ar[r]\ar[d]& A_f^{\bullet}\ar[d] \\
			(A_e^{\bullet})^{\vee} & (A_f^{\bullet})^{\vee}\ar[l]
		\end{tikzcd}
	\end{equation*} 
	commutes. By self-duality of $h_{\bullet}$, this follows immediately from Lemma \ref{lem:dual-homotopies}. 
\end{proof}

\paragraph{Composition.}
We show how to compose (generalized) isotropic reductions. Suppose we have generalized isotropic reductions \[\xi:G^{\bullet}\rightsquigarrow F^{\bullet}\quad \mbox{ and } \quad \zeta:F^{\bullet} \rightsquigarrow E^{\bullet}.\] 
These correspond to the following commutative diagram of solid arrows, which are all quasi-isomorphisms
\begin{equation}\label{eq:composition-diag}
	\begin{tikzcd}[sep = small]
		& & A_{\zeta\xi }^{\bullet}\ar[dr,dashed]\ar[dl,dashed] & &\\
		& A_{\xi}^{\bullet}\ar[dl,, "g_{\bullet}"']\ar[dr, "f_{\bullet}"]& & A_{\zeta}^{\bullet}\ar[dl, "f'_{\bullet}"'] \ar[dr, "e_{\bullet}"]& \\
		G^{\bullet}\ar[dr]& & F^{\bullet}\ar[dl]\ar[dr]& & E^{\bullet} \ar[dl]\\
		&(A_{\xi}^{\bullet})^{\vee} & &(A_{\zeta}^{\bullet})^{\vee} & 
	\end{tikzcd}
\end{equation}
We claim that we can canonically fill in $A_{\zeta\xi }^{\bullet}$ and the dashed arrows, so that the induced maps define a generalized isotropic reduction $G^{\bullet}\rightsquigarrow E^{\bullet}$. 

As in Remark \ref{rem:iso-red-comparison}, consider the maps $e^+_{\bullet}:F_+^{\bullet}\to E^{\bullet}$ and $f^+_{\bullet}: G_+^{\bullet}\to F^{\bullet}$ obtained by restricting $e_{\bullet}$ and $f_{\bullet}$ to positive degree. Composing degree-wise, we get a map $c^+_{\bullet}:G_+^{\bullet}\to E^{\bullet}$. One checks that it is an isomorphism on $h^i$ for $i\geq 2$ and surjective on $h^1$. Moreover, if $e_{\bullet}^+$ and $f_{\bullet}^+$ satisfy \ref{item:isotropic-constrii} of Construction \ref{constr:isotr-red}, then so does $c^+_{\bullet}$ as one can check directly. 

Hence the version of Construction \ref{constr:isotr-red} for generalized isotropic reductions applies to $c_{\bullet}^+$ and we get an isotropic reduction 
\[A_{c^{\vee}\circ c}^{\bullet}\xleftarrow{\eta^-_{\bullet}(c)} A_{c}^{\bullet}\xrightarrow{c_{\bullet}} E^{\bullet}.\] 
We take $A_{\zeta\xi}^{\bullet}:=A_{c}^{\bullet}$. To conclude, we show how it fits into the diagram \eqref{eq:composition-diag}.

Note that up to isomorphism, $A^{\bullet}_{\xi}$ and $A^{\bullet}_{\zeta}$ are obtained by splicing along $f^+_{\bullet}$ and $e_{\bullet}^+$ respectively. 
By functoriality of the splicing construction, we get a map 
\begin{equation}	\label{eq:composition-first-map} 	
	A_{c}^{\bullet}\xrightarrow{c'_{\bullet}} A_{\zeta}^{\bullet} 
\end{equation}	
which moreover satisfies $c_{\bullet} = e_{\bullet}\circ c_{\bullet}'$.
Since $e_i$ is an isomorphism for $i\geq 0$, we have a unique lift  $b_{\bullet}^+:G^{\bullet}_+\to A_{\zeta}^{\bullet}$ of $c_{\bullet}^+$ along $e_{\bullet}$. Then, again from functoriality, we get maps 
\begin{equation}\label{eq:composition-second-map}
	A_{c}^{\bullet}\xleftarrow{\sim} A_b^{\bullet}\to A^{\bullet}_{\xi}.
\end{equation} 
Here, the first arrow is an isomorphism of complexes by Lemma \ref{lem:isomorphic}. 
We take \eqref{eq:composition-second-map} and \eqref{eq:composition-first-map}  as the dashed maps in the diagram \eqref{eq:composition-diag}. We claim that this makes the diagram commute. This can be argued by functoriality of the splicing construction. To illustrate this, let $\operatorname{Spl}$ denote the splicing functor. Then the upper part of \eqref{eq:composition-diag} can be re-rewritten as
\begin{equation}
	\begin{tikzcd}[row sep = small, column sep = tiny]
		& & \operatorname{Spl}(G_+^{\bullet} \to E^{\bullet})\ar[dr,dashed]\ar[dl,dashed] & & \\
		& \operatorname{Spl}(G_+^{\bullet}\to F^{\bullet})\ar[dl,, "g_{\bullet}"']\ar[dr, "f_{\bullet}"]& & \operatorname{Spl}(F_+^{\bullet}\to E^{\bullet})\ar[dl, "f'_{\bullet}"'] \ar[dr, "e_{\bullet}"]& \\
		G^{\bullet}& & \operatorname{Spl}(F_+^{\bullet}\to F^{\bullet})& & E^{\bullet} 
	\end{tikzcd}
\end{equation}
Moreover, the map $f_{\bullet}'$ arises via functoriality as the composition
\[A_{\zeta}^{\bullet} = \operatorname{Spl}(F_+^{\bullet}\to E^{\bullet})\to  \operatorname{Spl}(F_+^{\bullet}\to (A^{\bullet}_{\zeta})^{\vee})\xleftarrow{\sim} \operatorname{Spl}(F_+^{\bullet}\to F^{\bullet}) = F^{\bullet}.\]
In view of \eqref{eq:composition-second-map}, the commutativity then comes down to commutativity of the following diagram, which is immediate
\begin{equation*}
	\begin{tikzcd}
		\operatorname{Spl}(G_+^{\bullet}\to A_{\zeta}^{\bullet})\ar[r]\ar[d, "\sim"]&  \operatorname{Spl} (G_+^{\bullet}\to F^{\bullet})\ar[r]&  \operatorname{Spl}(F_+^{\bullet}\to F^{\bullet}) \ar[d, "\sim"]\\
		\operatorname{Spl}(G_+^{\bullet}\to E^{\bullet})\ar[r] &  \operatorname{Spl}(F_+^{\bullet}\to E^{\bullet}) \ar[r]& \operatorname{Spl}(F_+^{\bullet}\to (A_{\zeta}^{\bullet})^{\vee}).
	\end{tikzcd}
\end{equation*}

This establishes commutativity of \eqref{eq:composition-diag} for our choice of dashed maps. Dualizing \eqref{eq:composition-diag}, we find that moreover the diagram 
\begin{equation*}
	\begin{tikzcd}
		G^{\bullet}\ar[r, leftarrow]\ar[d]& A_{c}^{\bullet}\ar[d] \\
		(A_c^{\bullet})^{\vee}\ar[r, leftarrow] &  E^{\bullet}
	\end{tikzcd}
\end{equation*}
commutes, and one sees that we indeed have a generalized isotropic reduction. Moreover, if we start with genuine isotropic reductions, $\xi, \zeta$ we obtain another one.   

\begin{definition}
	We will denote the (generalized) isotropic reduction $G^{\bullet}\rightsquigarrow E^{\bullet}$ obtained in this way by $\zeta\circ \xi$, and call it the \emph{composition} of $\xi$ and $\zeta$. 
\end{definition}	

This makes (generalized) isotropic reduction into the morphisms of a category:
\begin{lemma}
	Composition of generalized isotropic reductions is associative.
\end{lemma}
\begin{proof}[Proof sketch]
	Say, one has $\xi_1:J^{\bullet}\rightsquigarrow G^{\bullet}$, $\xi_2:G^{\bullet}\rightsquigarrow F^{\bullet}$ and $\xi_3:F^{\bullet}\rightsquigarrow E^{\bullet}.$
	The composition $(\xi_3\circ \xi_2)\circ \xi_1$ corresponds to a diagram
	
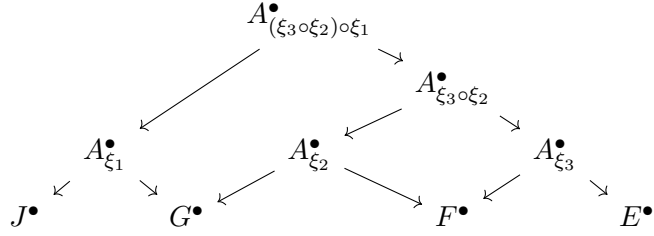
\begin{equation*}
		\begin{tikzcd}[sep = tiny]
			& & &A_{(\xi_3\circ \xi_2)\circ \xi_1}^{\bullet}\ar[ddll]\ar[dr] & & & \\
			& & & & A_{\xi_3\circ \xi_2 }^{\bullet}\ar[dr]\ar[dl] & &\\
			& A_{\xi_1}^{\bullet}\ar[dl]\ar[dr]& & A_{\xi_2}^{\bullet}\ar[dl]\ar[dr]& & A_{\xi_3}^{\bullet}\ar[dl] \ar[dr]& \\
			J^{\bullet}& &G^{\bullet}& & F^{\bullet}& & E^{\bullet} 
		\end{tikzcd}
	\end{equation*}
	One has an analogous diagram for $\xi_3\circ (\xi_2\circ \xi_1)$, with a complex $A^{\bullet}_{\xi_3\circ (\xi_2\circ \xi_1)}$ at the top. Now, the key point is that both of these are obtained from a refined diagram (which is unique up to unique isomorphism)
	
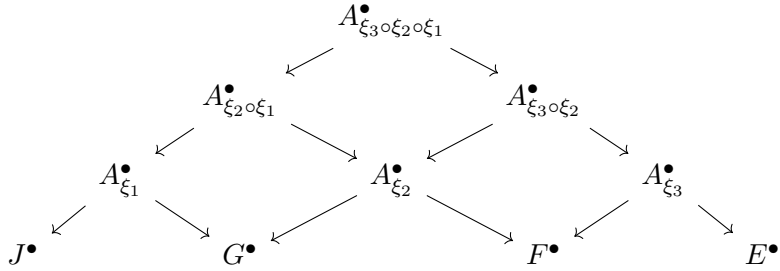
\begin{equation*}
		\begin{tikzcd}[sep = small]
			& & &A_{\xi_3\circ \xi_2\circ \xi_1}^{\bullet}\ar[dl]\ar[dr] & & & \\
			& &A_{\xi_2\circ \xi_1}^{\bullet} \ar[dl]\ar[dr]& & A_{\xi_3\circ \xi_2 }^{\bullet}\ar[dr]\ar[dl] & &\\
			& A_{\xi_1}^{\bullet}\ar[dl]\ar[dr]& & A_{\xi_2}^{\bullet}\ar[dl]\ar[dr]& & A_{\xi_3}^{\bullet}\ar[dl] \ar[dr]& \\
			J^{\bullet}& &G^{\bullet}& & F^{\bullet}& & E^{\bullet} 
		\end{tikzcd}
	\end{equation*}
	
\end{proof}

\begin{definition}\label{def:iso-red}
	We let $\mathcal{R}^{\gen}$ denote the category whose objects are self-dual representatives of $\bE$ and whose morphisms are generalized isotropic reductions. We let $\mathcal{R}$ denote the subcategory with the same objects and whose morphisms are isotropic reductions.
\end{definition}	
From the construction of composition in $\mathcal{R}^{\gen}$, one also has
\begin{lemma}\label{lem:reduction-compose-equal}
	Let $\zeta: F^{\bullet}\rightsquigarrow E^{\bullet}$ be a generalized isotropic reduction, and let  $a^+_{\bullet}:A_+^{\bullet}\to F^{\bullet}$ be a map satisfying \ref{item:isotropic-constri} of Construction \ref{constr:isotr-red}. Let $\zeta:G^{\bullet}\rightsquigarrow F^{\bullet}$ be the induced generalized isotropic reduction. Let moreover $\zeta':G'^{\bullet}\rightsquigarrow E^{\bullet}$ be the generalized isotropic reduction obtained by applying Construction \ref{constr:isotr-red} to the composition 
	\[b_{\bullet}^+: A_{+}^{\bullet}\xrightarrow{a^+_{\bullet}} F_+^{\bullet}\to E^{\bullet},\]
	where the last map is obtained from $\xi$. 
	
	Then, there is a canonical isomorphism of self-dual complexes $G'^{\bullet}\simeq G^{\bullet}$ under which $\zeta'$ is identified with $\zeta \circ \xi$. 
\end{lemma}	

\begin{lemma}\label{lem:iso-red-gives-iso-red}
	Let 
	\[F^{\bullet} \xleftarrow{f_{\bullet}} A^{\bullet} \xrightarrow{e_{\bullet}}E^{\bullet}\]
	be an isotropic reduction. Then 
	\begin{enumerate}[label = \roman*)]
		\item \label{item:isoiso-i} \[K:= \Ker(e^{\vee}_0:E^0\to A_0) \]
		is an isotropic sub-bundle of $E^0$. The morphism $e_0:A^0\to E^0$ is an isomorphism onto $K^{\perp}\subseteq E$, and the composition 
		\[K^{\perp}\simeq A^0\xrightarrow{f_0}F^0\]
		is surjective with kernel $K$. 
		\item \label{item:isoiso-ii}Let $\sigma_{\geq 1}(e_{\bullet}):A^{\bullet}_+\to E^{\bullet}_+$ be the bad truncation of $e_{\bullet}$ to positive degrees. Then, the mapping cone $\cone(\sigma_{\geq 1} e_{\bullet})$ is isomorphic in $\dercat(X)$ to the complex $K^{\vee}[-1]$. 
	\end{enumerate}
\end{lemma}

\begin{proof}
	For \ref{item:isoiso-i}, recall that $A^0= \Ker(E^0\oplus A^1\to E^1\oplus A^2)$.
	Applying the snake lemma to the diagram 
	\begin{equation*}
		\begin{tikzcd}
			0\ar[r]&A^1\ar[r]\ar[d,"{(e_1,d_1)}"]& E^0\oplus A^1\ar[r]\ar[d]& E^0\ar[d]\ar[r]&0 \\
			0\ar[r]&E^1\oplus A^2\ar[r] & E^1\oplus A^2\ar[r] & 0\ar[r]&0
		\end{tikzcd}
	\end{equation*}
	shows that the induced map $A^0\to E^0$ is injective and has the same quotient $K^{\vee}$ as $A^1\to E^1\oplus A^2$, which is a vector bundle by \ref{item:iso-redii} of Definition \ref{def:isotropic-red}. This also shows that $A^0\simeq K^{\perp}$.  Moreover, by inspection of connecting morphism, one finds that the natural morphisms fit into a commutative diagram 
	\begin{equation*}
		\begin{tikzcd}
			E^0\ar[r, "d_0"]\ar[d]&E^1\ar[d] \\
			K^{\vee}\ar[r, equal] & K^{\vee}.
		\end{tikzcd}
	\end{equation*}
	This, and its dual, show that the composition $K\to E^0\to K^{\vee}$ vanishes, so that $K\subseteq K^{\perp}$ and $K$ is isotropic. Using Remark \ref{rem:iso-red-comparison}, we have an isomorphism of $F^{\bullet}$ with the quadratic complex $A_{e^{\vee}\circ e}$ obtained by splicing. This means that $F^0 \simeq \Ker(A_0 \oplus A^1\to A_11\oplus A^2)$, where the morphism is the middle column of the diagram   
	
	\begin{equation*}
		\begin{tikzcd}[ampersand replacement=\&, sep = large]
			0\ar[r]\&A^1\ar[r]\ar[d,"{(-e_1^{\vee}e_1,d_1)}"]\& A_0\oplus A^1\ar[r]\ar[d, "{\begin{bmatrix}
					d & -e_1^{\vee}e_1 \\
					0 & d
			\end{bmatrix}}"]\& A_0\ar[d]\ar[r]\&0 \\
			0\ar[r]\&A_1\oplus A^2\ar[r] \& A_1\oplus A^2\ar[r] \& 0\ar[r]\&0
		\end{tikzcd}
	\end{equation*}
	The left column can be recognized as $A^1\to E^1\oplus A^2$. So, by another application of the snake Lemma, we see that $f_0^{\vee}:F^0\to A_0$ is injective with cokernel $K^{\vee}$. This, combined with the dual construction yield the isomorphism $F^0\simeq K^{\perp}/K$.
	
	For \ref{item:isoiso-ii}, it is clear that $\cone(\sigma_{\geq 1} e_{\bullet})$ has non-trivial cohomology sheaves at most in degrees $\geq 0$. Moreover, using $\sigma_{\geq _i}$ to denote the bad truncation functor, it fits into an exact sequence of complexes
	\[0\to  \cone(\sigma_{\geq 1} e_{\bullet})\to \sigma_{\geq 0}(\cone(e_{\bullet}^+)) \to E^0[0] \to 0.\]
	Here $e^+_{\bullet}$ is as in Construction \ref{constr:isotr-red}. By  \ref{item:isotropic-constri} there and the induced long exact sequence, we conclude that $h^i(\cone(\sigma_{\geq 1} e_{\bullet})) = 0$ for $i\geq 2$. Using further that $A^0$ is defined in terms of $e_{\bullet}^+$, we obtain the exact sequence 
	\[0\to h^0(\cone(\sigma_{\geq 1} e_{\bullet})) \to A^0 \xrightarrow{e_0} E^0\to h^1(\cone(\sigma_{\geq 1} e_{\bullet})) \to 0.\]
	By what we have already shown, the $h^0$-term vanishes and the $h^1$-term is identified with $K^{\vee}$.   
\end{proof}

To conclude this sub-section, we explain how Construction \ref{constr:isotr-red} interacts with chain homotopies.
\begin{lemma}\label{lem:homotopy-reduction-isom}
	Let $E^{\bullet}$ be a symmetric representative of $\bE$, and let $a^+_{\bullet}, b^+_{\bullet}:A^{\bullet}_+\to E^{\bullet}$ be two maps satisfying condition \ref{item:isotropic-constri} of Construction \ref{constr:isotr-red}. Let $\xi:F^{\bullet}\rightsquigarrow {E^{\bullet}}$ and $\zeta:G^{\bullet}\rightsquigarrow E^{\bullet}$ be the generalized isotropic reductions associated to $a_{\bullet}^+$ and $b_{\bullet}^+$ respectively. Then any chain homotopy $h_{\bullet}$ relating $a^+_{\bullet}$ and $b_{\bullet}^+$ gives rise to an isomorphism of self-dual representatives $\beta^h_{\bullet}:F^{\bullet}\to G^{\bullet}$. Moreover, it restricts to $\id_{A^i}$ in positive degrees, and $\beta^0_{\bullet}$ is the identity. 
\end{lemma}
\begin{proof}
	Recall that by construction, we first form 
	\[a_{\bullet}: A_a^{\bullet}\to E^{\bullet}\]
	by splicing along $a^+_{\bullet}$, and obtain $F^{\bullet}$ by splicing along $a^{\vee}\circ a$, and similarly for $b^+_{\bullet}$.
	
	By Construction \ref{constr:homotopy-isom}, we obtain an isomorphism $\alpha^h_{\bullet}:A_a^{\bullet}\to A_b^{\bullet}$, together with a chain homotopy $H_{\bullet}: a_{\bullet} \Rightarrow b_{\bullet}\circ \alpha^h_{\bullet}$. Write $c_{\bullet}:= b_{\bullet}\circ \alpha^h_{\bullet}$. Using $H_{\bullet}$, and its dual, we obtain a \emph{self-dual} homotopy $I_{\bullet}: a^{\vee}\circ a_{\bullet}\Rightarrow c^{\vee}_{\bullet}\circ c_{\bullet}$, namely
	\[I_{\bullet}:=H^{\vee}_{\bullet}a_{\bullet} + a_{\bullet}^{\vee}H_{\bullet} + H_{\bullet}^{\vee}d_{\bullet}H_{\bullet} + \frac{1}{2}(d_{\bullet}H^{\vee}_{\bullet}H_{\bullet} + H^{\vee}_{\bullet}H_{\bullet} d_{\bullet}).\]
	Applying Construction \ref{constr:homotopy-isom}, one obtains an isomorphism of complexes
	\[\alpha^I_{\bullet}:A_{a^{\vee}\circ a}^{\bullet}\to A_{c^{\vee}\circ c}^{\bullet},\]
	compatible with the identifications with $\bE$.
	By Lemma \ref{lem:dual-homotopies}, it also respects the self-dual structures. Together with the identification $A_{c^{\vee}\circ c}^{\bullet}\simeq A_{b^{\vee}\circ b}^{\bullet}$ obtained from $\alpha^{h}_{\bullet}$, this gives the desired isomorphism $F^{\bullet}\to G^{\bullet}$.  
	
\end{proof}

\subsection{Families of morphisms}\label{sec:families}
Let $T$ be an affine scheme, and let $F^{\bullet},E^{\bullet}\in \mathcal{R}$ be self-dual representatives of $\bE$. 
\begin{definition}\label{def:family-reductions}
	A \emph{family of (generalized) isotropic reductions} from $E^{\bullet}$ to $F^{\bullet}$ over $T$ is given by a (generalized) isotropic reduction from  $\operatorname{pr}_2^*E^{\bullet}$ to $\operatorname{pr}_2^*F^{\bullet}$, regarded as self-dual representatives of the quadratic complex $\operatorname{pr}_2^*\bE$ on $T\times X$. 
\end{definition} 
Families of morphisms can be pulled back along arbitrary morphisms of affine schemes $T_1\to T_2$.

The notion of family allows us to relate morphisms:
\begin{proposition}	\label{prop:connect-reductions}
	Let $\xi_0,\xi_1:F^{\bullet}\rightsquigarrow E^{\bullet}$ be morphisms in $\mathcal{R}^{\gen}$. Then there exists a morphism $\upsilon:G^{\bullet}\rightsquigarrow F^{\bullet}$ in $\mathcal{R}^{\gen}$, and a family of generalized isotropic reductions $\zeta_t$ from $E^{\bullet}$ to $G^{\bullet}$ over $\bA^1$, such that $\zeta_i = \xi_i \circ \upsilon $ for $i=0,1$. 
\end{proposition}
\begin{proof}
	By definition, for $i=0,1$ the morphism $\xi_i$ is given by a diagram
	\[F^{\bullet}\xleftarrow{f_{i,\bullet}} A_{\xi_i}^{\bullet} \xrightarrow{e_{i,\bullet}} E^{\bullet}.\]
	
	In particular, we have the two truncated maps $e^{+}_{0,\bullet},  e^{+}_{1,\bullet}:F_+^{\bullet}\to E^{\bullet}$. These induce the \emph{same} map in the derived category, namely the bad truncation $F_+^{\bullet}\to F^{\bullet}$ composed with $\alpha_{E}^{-1}\circ \alpha_F$. Thus, we can choose a map of complexes $a_{\bullet}^+:A_+^{\bullet}\to F_+^{\bullet}$ surjective on $h^1$ and an isomorphism on $h^i$ for $i\geq 2$ with $A_+^{\bullet}$ again concentrated in positive degrees, such that the compositions $b^+_{0,\bullet}:= e^+_{0,\bullet}\circ  a^+_{\bullet}$ and $b^+_{1,\bullet}:=a^+_{\bullet}\circ e^+_{1,\bullet}$ are related by a chain homotopy $h_{\bullet}: A^{\bullet}_+\to E^{\bullet}[-1]$\footnote{We give a proof that works for quasi-projective $X$. The proof in the affine case can be simplified somewhat.}, i.e.
	\[b^+_{1,\bullet} = b^+_{0,\bullet} +d_{\bullet}h_{\bullet} + h_{\bullet} d_{\bullet}.\]
	
	Let $\upsilon: G^{\bullet} \rightsquigarrow F^{\bullet}$ denote the generalized isotropic reduction associated to $a_{\bullet}^+$. For $i=0,1$, let $\zeta_i: G_i^{\bullet}\rightsquigarrow E^{\bullet}$ denote the generalized isotropic reduction associated to $b^+_{i, \bullet}$.
	By Lemma \ref{lem:reduction-compose-equal}, there is an isomorphism $\varphi_{i,\bullet}:G_i^{\bullet}\simeq G^{\bullet}$ identifying $\zeta_i$ with $\xi_i \circ \upsilon$.  
	
	We will now use $h_{\bullet}$ to interpolate between $\zeta_0$ and $\zeta_1$:
	Consider the morphism  of complexes $b^+_{t,\bullet}: \pr_{2}^*A_+^{\bullet}\to \pr_2^*E^{\bullet}$ on $\bA^1\times X$, given by 
	\[b^+_{t,\bullet}:=\pr_2^*b^+_{0,\bullet} + t (d_{\bullet} \pr_2^*h_{\bullet} + \pr_2^*h_{\bullet} d),\]
	where $t$ is the coordinate on $\bA^1$. 
	
	Then $b^+_{t,\bullet}$ restricts to $b^+_{0,\bullet}$ over $t=0$ and to $b^+_{1,\bullet}$ over $t=1$. Moreover, the maps $\pr_2^*b_{0,\bullet}^+$ and $b_{t,\bullet}^+$ are chain homotopic via the chain homotopy
	\[h_{t,\bullet}:=t\pr_2^*h_{\bullet}.\]
	Let $\zeta_{t}:G^{t}_{\bA^1}\rightsquigarrow \pr_2^*E^{\bullet}$ be the  generalized isotropic reduction on $\bA^1\times X$ associated to $b_{t,\bullet}^+$. Note that its restrictions to $0,1\in \bA^1$ are canonically identified with $\zeta_0$ and $\zeta_1$ respectively. 
	
	By Lemma \ref{lem:homotopy-reduction-isom}, there is an isomorphism of self-dual complexes $\Phi_{t,\bullet}:\pr_2^*G_0^{\bullet}\to G_t^{\bullet}$ induced by $h_{t,\bullet}$, which restricts to the identity at $t=0$.
	
	Using the identification $G^{\bullet}\simeq G_0^{\bullet}$, we have now constructed a family of morphisms $\zeta_t\circ \Phi_{t,\bullet}\circ (\phi_{0,\bullet})^{-1}:G^{\bullet}\rightsquigarrow E^{\bullet}$ in $\mathcal{R}^{\gen}$ over $\bA_1$, such that:
	\begin{itemize}
		\item At $t=0$, we recover $\xi_0\circ \upsilon = \zeta_1\circ (\varphi_{0, \bullet})^{-1}$,
		\item At $t=1$, we recover $ \zeta_2 \circ \Phi_{1,\bullet } \circ (\varphi_{0,\bullet})^{-1}$  
	\end{itemize}
	
	Now, if for some reason, one had $\Phi_{1,\bullet } \circ (\varphi_{0,\bullet})^{-1} = \varphi_{1,\bullet}^{-1}$ as maps $G^{\bullet}\to G_1^{\bullet}$, then we would be done. In general, we only know that both maps are isomorphisms of self-dual complexes, agree in degrees $\neq 0$, and have the same image in the derived category. Hence, setting
	\[\psi_t:=t \varphi_{1,\bullet}\circ \Phi_{1,\bullet}\circ (\varphi_{0,\bullet})^{-1} + (1-t) \id_{G^{\bullet}}.\]
	
	and replacing $\zeta_t$ with $\zeta_t\circ \psi_{t\bullet}$ gives us the desired family.
\end{proof}

We can further generalize the previous result: Given any \emph{family} of generalized isotropic reductions over some $T$, there exists a single $\upsilon:G^{\bullet}\rightsquigarrow F^{\bullet}$, such that after precomposition with $\upsilon$, the family becomes contractible. 

In the following proposition, let $\bA^n$ denote affine space over any base ring over which $X$ is defined (e.g. the integers). All fiber products are to be understood over this base. 

\begin{proposition}	\label{prop:connect-families}
	Let $\xi_T$ be a family of generalized isotropic reductions $F^{\bullet}\rightsquigarrow E^{\bullet}$ over an affine scheme $T$, and let $\xi:F^{\bullet}\rightsquigarrow E^{\bullet}$ be another generalized isotropic reduction. 
	Then there exists a further $\upsilon:G^{\bullet}\rightsquigarrow F^{\bullet}$ in $\mathcal{R}^{\gen}$, and a family $\zeta_{T\times \mathbb{A}^1}$ of generalized isotropic reductions from $E^{\bullet}$ to $G^{\bullet}$, such that $\zeta_{T\times \{1\}}$ is equal to the constant family with value $\xi\circ \upsilon$, and such that $\zeta_{T\times \{0\}}$ is equal to the family $\xi_T \circ \upsilon$. 
\end{proposition}

\begin{proof}
	The proof is essentially the same as the one of Proposition \ref{prop:connect-reductions} with $X$ replaced by $T\times X$ with one additional subtlety:
	
	To guarantee that one can choose $G^{\bullet}\rightsquigarrow F^{\bullet}$ on $X$ rather than $T\times X$, one proceeds as follows:
	Suppose one has chosen $A_{T,+}^{\bullet}\to \pr_2^*F_+^{\bullet}$ (on $T\times X$) as in the proof of Proposition \ref{prop:connect-reductions} such that the two respective compositions $A_{T,+}^{\bullet}\to \pr_2^*E$ become homotopic, 
	
	Now we can choose $a^+_{\bullet}:A_{+}^{\bullet}\to F_+^{\bullet}$ on $X$ with $A_+^{\bullet}$ sufficiently negative\footnote{Assuming $X$ is quasi-projective.}, such that its pullback to $T\times X$ possesses a lift $\pr_2^*A^{\bullet}\to A_{T,+}^{\bullet}$ up to homotopy -- the argument here uses is the same for lifting maps from projective resolutions along quasi-isomorphisms and uses vanishing of finitely many elements in sheaf cohomology groups. 
	Then everything goes through with $\upsilon: G^{\bullet}\rightsquigarrow E^{\bullet}$ constructed from $a^+_{\bullet}$.
\end{proof}

By taking the argument in the proof of Proposition \ref{prop:connect-reductions} in a different direction, we obtain the following connectedness result
\begin{proposition}\label{prop:simply-connected}
	The category $\mathcal{R}$ is weakly simply connected, in the sense that any functor $\mathcal{R}\to \mathcal{G}$ into a groupoid $\mathcal{G}$ is equivalent to a constant functor. 
\end{proposition}
\begin{proof}
	We've seen in Corollary \ref{cor:repcat-connected} that $\mathcal{R}$ is connected. Let $F:\mathcal{R}\to \mathcal{G}$ be a functor into a groupoid. 
	It suffices to show that any pair of morphisms $\xi_0,\xi_1:F^{\bullet}\rightsquigarrow E^{\bullet}$ become equal under $F$. 
	Let $\xi_0,\xi_1$ be given. Using Example \ref{ex:reductions} \ref{item:reductionsi}, It suffices to show that $\xi_i\oplus K^{\bullet}\oplus K_{\bullet}$ are identified under $F$ for a single acyclic complex $K^{\bullet}$ in non-negative degrees.
	Furthermore, it suffices to show that $F(\xi_0\circ \zeta) = F(\xi_1 \circ\zeta)$ for a single morphism $\zeta:G^{\bullet}\rightsquigarrow F^{\bullet}\oplus K$ in $\mathcal{R}$.
	
	Exactly as in the proof of Proposition \ref{prop:connect-reductions}, write out $\xi_i$ as
	\[F^{\bullet}\xleftarrow{f_{i,\bullet}}A^{\bullet}_{\xi_i}\xrightarrow{e_{i,\bullet}}E^{\bullet},\]
	and choose a quasi-isomorphism $a_{\bullet}:A_+^{\bullet}\to F_+^{\bullet}$ together with a chain-homotopy $h_{\bullet}:A_+^{\bullet}\to E^{\bullet}[-1]$ satisfying
	\[b_{1,{\bullet}}^+ = b_{0,\bullet}^+ +d_{\bullet}h_{\bullet}+h_{\bullet}d_{\bullet}.\]
	Here, $b_{i,\bullet}^+ = a_{\bullet}^+\circ a_{i,\bullet}^+$.
	
	Let $v:G^{\bullet}\rightsquigarrow F^{\bullet}$ be the generalized isotropic reduction associated to $a_{\bullet}^+$, and $\zeta_i:G_i^{\bullet}\rightsquigarrow E^{\bullet}$ the one associated to $b_{i,+}^{\bullet}$ for $i=0,1$. Lemma \ref{lem:reduction-compose-equal} provides isomorphisms $\varphi_{i,\bullet}:G_i^{\bullet}\simeq G^{\bullet}$. 

	 Explicitly, in degree zero,  we have 
	\[G^0_i = h^0(A_{-2}\oplus E_{-1}\xrightarrow{\begin{bmatrix}
			d &  -b_{i,1}^{\vee} \\
			0 & d \\
			0& 0			
	\end{bmatrix}} A_{-1}\oplus E\oplus A^1 \xrightarrow{\begin{bmatrix}
			0& d & -b_{i,1}\\
			0& 0 & d
	\end{bmatrix}} E^1\oplus A^2).\]
	Lemma \ref{lem:homotopy-reduction-isom} provides us with a further isomorphism of self-dual representatives $\Phi_{\bullet}:G_0^{\bullet}\to G_1^{\bullet}$. It is given by the identity in degrees $\neq 0$, and in degree zero it is induced by
	\begin{equation}\label{eq:isom-formula}
		A_{-1}\oplus E\oplus A^1\xrightarrow{\begin{bmatrix}
				\id & -h^*& -\frac{1}{2}h^*h\\
				& \id & h\\
				& & \id
		\end{bmatrix}} A_{-1}\oplus E\oplus A^1.
	\end{equation}
	Writing out $\zeta_i$ as $G^{\bullet}_i\leftarrow A^{\bullet}_i\to E^{\bullet}$, we also have an automorphism $\phi^{\bullet}:A^{\bullet}_0\to A_1^{\bullet}$, and homotopies provided by Construction \ref{constr:homotopy-isom}
	\begin{equation*}
		\begin{tikzcd}[row sep = small]
			G_0^{\bullet}\ar[dd,"\sim"',"\Phi_{\bullet}"]& A_0^{\bullet}\ar[dr]\ar[l]\ar[dd,"\sim"', "\varphi_{\bullet}"]\ar[ddl,Rightarrow]&  \\
			&\, \ar[r,Rightarrow, shorten >=0.9ex, shorten <=0.9ex] &E^{\bullet}\\
			G_1^{\bullet}&	A_1^{\bullet}\ar[l]\ar[ur] & \,
		\end{tikzcd}
	\end{equation*}
	
	We apply Example \ref{ex:reductions} point \ref{item:reductionsiii} to $\zeta_i$ and $v$, and denote the resulting isotropic reduction by $\widehat{\zeta_i}$ and $\hat{v}$. Note that $\widehat{\zeta}_i$  Let $K^{\bullet} = \cone(\id_{G_+^{\bullet}})$, so that $\widetilde{\zeta_i}$ has target $E_K^{\bullet}:= K_{\bullet}\oplus E\oplus K^{\bullet}$.
	We define an automorphism $\Psi_{\bullet}$ of this complex as follows: Set $A^i = 0$ for $i\leq 0$. Then, in degree $i$, we have 
	\[E_K^i =  A_{-i-1} \oplus A_{-i}\oplus E^i\oplus  A^{i}\oplus A^{i+1}.\]
	We define 
	\[\Psi_i:=\begin{bmatrix}
		\id & \quad & - h_{i+1}^* & & -\frac{1}{2}h_{i+1}^*h_{i+1}\\
		\quad& \id& -d h_{i+1}^* - h_i^* d & & \\
		& & \id &dh_{i} + h_{i+1}d &h_{i+1}\\
		& & & \id& \\
		& &  & & \id 
	\end{bmatrix}.\] 
	
	We claim that this fits into a \emph{commutative} diagram in $\mathcal{R}$:
	\begin{equation}\label{eq:final-contraction-diag}
		\begin{tikzcd}[row sep = small]
			G_0^{\bullet}\ar[r, squiggly,"{\widehat{\zeta}_0}"]\ar[dd,"\sim"', "\Phi_{\bullet}"]&E_K^{\bullet}\ar[dd,"\sim"',"\Psi_{\bullet}"]&  \\
			& & E^{\bullet}\ar[ul,squiggly]\ar[dl,squiggly]\\ 
			G_1^{\bullet} \ar[r, squiggly,"{\widehat{\zeta}_1}"]&E_K^{\bullet}&
		\end{tikzcd}
	\end{equation}
	where both maps $E^{\bullet}\rightsquigarrow E_K^{\bullet}$ are the one of Example \ref{ex:reductions} point \ref{item:reductionsi}. This implies that $\widetilde{\zeta}_0 = \widetilde{\zeta}_1\circ \Phi_{\bullet}$ in the total localization of $\mathcal{R}$. 
	That the square in \eqref{eq:final-contraction-diag} commutes follows from the commutativity of the diagram 
	\begin{equation*}
		\begin{tikzcd}
			G_0^{\bullet}\ar[r]\ar[d,"\Phi_{\bullet}"]& K_{\bullet}\oplus A_0^{\bullet}\ar[r]\ar[d,"{\tilde{\varphi}_{\bullet}}"]& E_K\ar[d,"\Psi_{\bullet}"] \\
			G_1^{\bullet}\ar[r] & K_{\bullet}\oplus A_1^{\bullet}\ar[r]&E_K
		\end{tikzcd}
	\end{equation*} 
	where the map denoted by $\widetilde{\varphi}_{\bullet}$ is given by 
	\[A_{-i-1}\oplus A_{-i}\oplus A_0^{-i}\xrightarrow{\begin{bmatrix}
			\id& & -h_{i+1}^*|_{A_0^{-i}}\\
			& \id& -d h_{i+1}^* -h_i^* d\\
			& & \varphi_{\bullet}
	\end{bmatrix}}A_{-i-1}\oplus A_{-i}\oplus A_1^{-i}.\]
	The commutativity of the triangle in \eqref{eq:final-contraction-diag} is checked similarly.
\end{proof}
\subsection{Simplicial structure}\label{sec:simpl-struc}
We endow $\mathcal{R}$ and $\mathcal{R}^{\gen}$ with a simplicial enrichment \cite[\href{https://kerodon.net/tag/00JQ}{Subsection 00JQ}]{kerodon}.
For this, we will use the following model of $n$-simplices by affine spaces. 
\begin{definition}
	We define a co-simplicial object $\mathbb{A}^{\Delta}$ in affine schemes as follows:
	\begin{enumerate}
		\item Its value in degree $n$ is the affine space $\mathbb{A}^n$.
		\item For any morphism $\varphi: [m]\to [n]$ in the simplex category, we associate the morphism $\mathbb{A}^m\to \mathbb{A}^{n}$ given by
		\[(x_1,\ldots,x_m)\mapsto \sum_{i=0}^m x_i \mathsf{v}_{\varphi(i)},\]
		where $\mathsf{v}_0 = 0$ and $\mathsf{v}_{i} = e_{i}$ is the $i$-th coordinate for $i\geq 1$, and where $x_0:=1-\sum_{i=1}^n x_i$.
	\end{enumerate}
\end{definition}
We will collect more related construction and background in \S \ref{sec:affine-models}.

For $E^{\bullet},F^{\bullet}\in \mathcal{R}$, let $\underline{\Hom}_{\mathcal{R}}(F^{\bullet},E^{\bullet})$ denote the presheaf that sends a scheme $T$ to the set of families of maps $F^{\bullet}\to E^{\bullet}$ over $T$, and similarly for $\mathcal{R}^{\gen}$. 

\begin{definition}\label{def:simplicial-enrichment-reductions}
	Let $\tilde{\mathcal{R}}^{\Delta}$ denote the simplicial category whose objects are the objects of $\mathcal{R}$ and such that for $E^{\bullet}, F^{\bullet}\in \mathcal{R}$, the associated Hom- simplicial set is given by the composition 
	\[\Delta \xrightarrow{\mathbb{A}^{\Delta}} \operatorname{Schemes} \xrightarrow{\underline{\Hom}_{\mathcal{R}}(F^{\bullet},E^{\bullet}) } \operatorname{Set}.\] 	
	Similarly, define $\tilde{\mathcal{R}}^{\gen, \Delta}$ by replacing $\mathcal{R}$ with $\mathcal{R}^{\gen}$ in the above.	
\end{definition}

It is clear that $\tilde{\mathcal{R}}^{\Delta}$ defines a simplicial enrichment over $\mathcal{R}$, i.e. the morphisms $F^{\bullet}\to E^{\bullet}$ in $\mathcal{R}$ are precisely the $0$-simplices in $\Hom_{\mathcal{R}^{\Delta}}(F^{\bullet}, E^{\bullet})$, and similarly for $\tilde{\mathcal{R}}^{\gen, \Delta}$. Moreover, $\tilde{\mathcal{R}}^{\Delta}$ is a simplicial sub-category of $\tilde{\mathcal{R}}^{\gen,\Delta}$ in the obvious way. 

We now employ some basic tools from simplicial homotopy theory.
For technical reasons, it is more convenient to work with simplicial categories that are \emph{locally Kan}, in the sense that every Hom-simplicial set between two objects is a Kan complex \cite[\href{https://kerodon.net/tag/00JY}{Definition 00JY}]{kerodon}. There is an explicit functorial construction $\Ex^{\infty}$ which to any simplicial set $Y$ associates a Kan complex $\Ex^{\infty}(Y)$ together with an injective weak equivalence of simplicial sets $Y\to \Ex^{\infty}(Y)$ \cite[\href{https://kerodon.net/tag/00ZA}{Subsection 00ZA}]{kerodon}. Moreover, $\Ex^{\infty}$ preserves finite products. 

Given any simplicial category $\mathcal{C}^{\Delta}$, one can construct a locally Kan simplicial category $\mathcal{C}^{\Delta, \Ex^{\infty}}$ by declaring the objects of $\mathcal{C}^{\Delta,\Ex^{\infty}}$ to be the objects of $\mathcal{C}^{\Delta}$, and for any pair of objects $x,y$ taking 
\begin{equation}\label{eq:hom-spaces-locally-kan-replacement}
	\Hom_{ \mathcal{C}^{\Delta, \Ex^{\infty}}}(x,y) :=\Ex^{\infty}\left(\Hom_{\mathcal{C}^{\Delta}}(x,y)\right).
\end{equation}
\begin{definition}\label{def:kan-enrichment-reductions}
	We define 
	\[\mathcal{R}^{\Delta}:= \tilde{\mathcal{R}}^{\Delta, \Ex^{\infty}} \mbox{ and } \mathcal{R}^{\gen,\Delta}:= \tilde{\mathcal{R}}^{\gen,\Delta,\Ex^{\infty}}.\]
\end{definition} 

To any simplicial category $\mathcal{C}^{\Delta}$, one can associate a simplicial set $\Nhc(\mathcal{C}^{\Delta})$ via the \emph{homotopy-coherent nerve} construction \cite[\href{https://kerodon.net/tag/00KM}{Subsection 00KM}]{kerodon}. If the simplicial category is locally Kan, then its homotopy coherent nerve is an \emph{infinity category}. 

The reason for considering the simplicial enrichment are the following contractibility results. 

\begin{proposition}\label{prop:cofiltered}
	The $\infty$-category $\Nhc(\mathcal{R}^{\gen,\Delta})$ is \emph{cofiltered} (cf. \cite[Definition 5.3.1.7]{lurie2006higher} for the opposite notion). In particular, it is weakly contractible as a simplicial set.  
\end{proposition}

Intuitively, this means that the simplicial category obtained from $\mathcal{R}^{\gen,\Delta}$ by inverting all morphisms is equivalent to the ordinary category with one object and only the identity morphism.

\begin{proof}
	We will show that the \emph{simplicial category} $\mathcal{R}^{\gen,\Delta}$ is cofiltered in the following sense \cite[Definition 5.3.1.1 and Remark 5.3.1.4]{lurie2006higher}
	\begin{enumerate}
		\item For any finite collection of objects $E^{\bullet}_i$ in $\mathcal{R}^{\gen,\Delta}$, there exists an object $F^{\bullet}\in \mathcal{R}^{\gen,\Delta}$, and morphisms $F^{\bullet}\rightsquigarrow E_i^{\bullet}$ for each $i$.\label{item:cofilteredi}
		\item Given $F^{\bullet}, E^{\bullet}\in \mathcal{R}^{\gen,\Delta}$ and a map of simplicial sets $\varphi:\partial \Delta^{n+1}\to \Hom_{\mathcal{R}^{\gen,\Delta}}(F^{\bullet}, E^{\bullet})$, where $n\geq 0$,  there exists a morphism $a:G^{\bullet}\rightsquigarrow F^{\bullet}$, such that the induced map $a^*\varphi:\partial \Delta^{n+1}\to \Hom_{\mathcal{R}^{\gen,\Delta}}(F^{\bullet},E^{\bullet})$ is weakly contractible. \label{item:cofilteredii} 
	\end{enumerate}
	
	Then we can conclude by (the opposite of) \cite[Proposition 5.3.1.13 and Lemma 5.3.1.18]{lurie2006higher}, which state, respectively, that a locally Kan simplicial category $\mathcal{C}^{\Delta}$ is filtered if and only if $\Nhc(\mathcal{C}^{\Delta})$ is filtered as an infinity category, and that a filtered infinity category is weakly contractible.
	
	Point \ref{item:cofilteredi} follows by a repeated application of Proposition \ref{prop:iso-red}.  
	
	For point \ref{item:cofilteredii} we use Proposition \ref{prop:connect-families}: Let $\varphi:\partial\Delta^{n+1}\to \Hom_{\mathcal{R}^{\gen,\Delta}}(F^{\bullet}, E^{\bullet})$ be a map of simplicial sets. By definition, the target of $\varphi$ is $\Ex^{\infty}(\Hom_{\tilde{\mathcal{R}}^{\gen,\Delta}}(F^{\bullet}, E^{\bullet}))$. Now we use some standard properties of the $\Ex^{\infty}$-functor. 
	Since $\partial\Delta^{n+1}$ is a finite simplicial set, the map $\varphi$ factors through some finite stage $\operatorname{Ex}^k(\Hom_{\mathcal{R}^{\gen,\Delta}}(F^{\bullet}, E^{\bullet}))$, and therefore corresponds to an adjoint map 
	\[\operatorname{Sd}^k(\partial\Delta^n)\to \Hom_{\mathcal{R}^{\gen,\Delta}}(F^{\bullet},E^{\bullet}).\] 
	Now, by Proposition \ref{prop:geom-real} and Example \ref{ex:subdivs} below, we can find a geometric realization of $\operatorname{Sd}^k(\partial\Delta^{n+1})$ (in the language of that section) as a closed subscheme $\mathbb{A}(\operatorname{Sd}^k(\partial\Delta^{n+1}))\subseteq\bA^N$ for some $N$. By Proposition \ref{prop:geom-real-adjunction} there, the map $\varphi$ corresponds to a map $\mathbb{A}(\operatorname{Sd}^k(\partial\Delta^{n+1}))\to \underline{\Hom}_{\mathcal{R}^{\gen}}(F^{\bullet}, E^{\bullet})$ (also denoted $\varphi$). Proposition \ref{prop:connect-families} now states that we can find $G^{\bullet}\to F^{\bullet}$ and a map 
	\[\Phi: \mathbb{A}(\operatorname{Sd}^k(\partial\Delta^{n+1})) \times \mathbb{A}^1\to \underline{\Hom}_{\mathcal{R}^{\gen}}(G^{\bullet}, E^{\bullet})\]  
	restricting to $a^*\varphi$ at $0\in \mathbb{A}^1$ and to a constant map at $1\in \mathbb{A}^1$. 
	
	Via the last vertex map \cite[\href{https://kerodon.net/tag/00YR}{Subsection 00YR}]{kerodon} and functoriality of subdivision, we have morphisms
	\[\Sd^k(\partial\Delta^{n+1} \times \Delta^1)\to \Sd^k(\partial\Delta^{n+1})\times \Sd^k\Delta^1\to \Sd^k(partial\Delta^{n+1})\times \Delta^1.\]
	Taking $\bA(-)$ and composing with $\Phi$, we get a map
	\[\Psi: \bA(\Sd^k(\partial\Delta^{n+1} \times \Delta^1)) \to \underline{\Hom}_{\mathcal{R}^{\gen}}(G^{\bullet}, E^{\bullet}),\]
	which corresponds to a map 
	\[\psi: \partial \Delta^{n+1} \times \Delta^1 \to \Hom_{\mathcal{R}^{\gen,\Delta}}(G^{\bullet},E^{\bullet}).\]
	One checks that $\psi$ restricts to $a^*\varphi$ at $0\in \Delta^1$ and is constant at $1\in \Delta^1$. This shows that $a^*\varphi$ is homotopy equivalent to a constant map, hence trivial on homotopy groups. 
\end{proof}

Using a generalization of Example \ref{ex:reductions} which we state in Lemma \ref{lem:ex-red-generalized}, one can conclude the same for the category of isotropic reductions
\begin{theorem}\label{thm:contractible}
	The $\infty$-category $\Nhc(\mathcal{R}^{\Delta})$ is weakly contractible.
\end{theorem}
\begin{proof}
	Let $y$ be an arbitrary vertex of $\Nhc(\mathcal{R}^{\Delta})$ (which exists by \S \ref{sec:symm-rep}). As in the proof of \cite[Lemma 5.3.1.18]{lurie2006higher}, it is enough to show that the homotopy sets $\pi_i(\Nhc(\mathcal{R}^{\Delta}),y)$ are singletons for $i\geq 0$. 
	If not, there exists a map $\varphi:K\to \Nhc(\mathcal{R}^{\Delta})$ from a (pointed) finite simplicial set inducing a non-trivial map on $\pi_i$ for some $i$. 
	
	By Proposition \ref{prop:cofiltered}, we can extend $\varphi$ to a map $\varphi^{\triangleleft}:K^{\triangleleft}\to \Nhc(\mathcal{R}^{\gen,\Delta})$, which we may regard as a map $\varphi^{\triangleleft}:K\to  \Nhc(\mathcal{R}^{\gen,\Delta})_{F^{\bullet}/}$. 
	Note that composing $\varphi^{\triangleleft}$ with the forgetful map $\pi:\Nhc(\mathcal{R}^{\gen,\Delta})_{F^{\bullet}/}\to \Nhc(\mathcal{R}^{\gen,\Delta})$ recovers $\varphi$.
	We consider the map 
	\[\Omega_{F^{\bullet}}:\Nhc(\mathcal{R}^{\gen,\Delta})_{F^{\bullet}/}\to \Nhc(\mathcal{R}^{\gen, \Delta})_{F^{\bullet}/}\]
	 of Lemma \ref{lem:ex-red-generalized}.
	 By part \ref{item:exredgen2b} of the same lemma, we find that $\Omega_{F^{\bullet}}\circ \varphi^{\triangle}$ factors through $\Nhc(\mathcal{R}^{\Delta})_{F^{\bullet}/}$. We obtain a corresponding map $\tilde{\varphi}^{\triangleleft}:K^{\triangleleft}\to \Nhc(\mathcal{R}^{\Delta})$ whose restriction to $K$ is $\Sigma_{F^1}(\varphi)$ by Lemma \ref{lem:ex-red-generalized} \ref{item:exredgen2a}. 
	By Lemma \ref{lem:ex-red-generalized} \ref{item:exredgen1b}, we also have the map $\psi:\Nhc(\Xi_{F^1})\circ (\varphi\times \id_{\Delta^1}):K\times \Delta^1\to \Nhc(\mathcal{R}^{\Delta})$. 
	By glueing $\tilde{\varphi}^{\triangleleft}$ and $\psi$ along $K$, we obtain a map 
	\[\tilde{\psi}:K\times \Delta^1 \coprod_{K} K^{\triangleleft}\to \Nhc(\mathcal{R}^{\Delta}),\]
	whose restriction to $K\times\{0\}\subset K\times \Delta^1$ recovers $\varphi$. Since the domain of $\tilde{\psi}$ is weakly contractible, it induced trivial maps on homotopy sets. Hence, $\varphi$ induces trivial maps on homotopy sets, which produces the desired contradiction. 
\end{proof}

In the following, let $[1]$ denote the simplicial category associated to the partially ordered set $\{0,1\}$.
\begin{lemma}\label{lem:ex-red-generalized}
	\begin{enumerate}
		\item Let $K$ be a vector bundle on $X$, and consider the acyclic complex $K^{\bullet} := [K\to K]$ in degrees $[0,1]$. 
		\begin{enumerate}
			\item The operation $E^{\bullet}\mapsto E^{\bullet}\oplus K^{\bullet}\oplus K_{\bullet}$ defines a functor of simplicial categories $\Sigma_K:\mathcal{R}^{\gen, \Delta}\to  \mathcal{R}^{\gen, \Delta}$, which restricts to a functor $\mathcal{R}^{\Delta}\to \mathcal{R}^{\Delta}$. 
			\item \label{item:exredgen1b} There is a natural functor of simplicial categories 
			\[\Xi_K: \mathcal{R}^{\gen, \Delta}\times [1]\to \mathcal{R}^{gen,\Delta}\] that restricts to the identity over $0\in [1]$ and to $\Sigma_K$ over $1\in [1]$, and whose value at a morphism $(\id_{E^{\bullet}}, (0\to 1))$ is the isotropic reduction $\xi_{K}(E^{\bullet}):E^{\bullet}\rightsquigarrow \Sigma_{K}E^{\bullet}$ of Example \ref{ex:reductions} \ref{item:reductionsi}. 
			Moreover, $\Xi_K$ restricts to a functor 
			\[\mathcal{R}^{\Delta}\times [1]\to \mathcal{R}^{\Delta}.\]
		\end{enumerate}
		\item \label{item:exredgen2} Let $F^{\bullet}\in \mathcal{R}$ be arbitrary. There is a natural map of simplicial sets
		\[\Omega_{F^{\bullet}}:\Nhc(\mathcal{R}^{\gen,\Delta})_{F^{\bullet}/}\to \Nhc(\mathcal{R}^{\gen, \Delta})_{F^{\bullet}/}\]
		with the following properties
		\begin{enumerate}
			\item \label{item:exredgen2a} There is a commutative diagram
			\begin{equation*}
				\begin{tikzcd}
					\Nhc(\mathcal{R}^{\gen,\Delta})_{F^{\bullet}/}\ar[r,"\Omega_{F^{\bullet}}"]\ar[d,"\pi"]& \Nhc(\mathcal{R}^{\gen,\Delta})_{F^{\bullet}/}\ar[d,"\pi"] \\
					\Nhc(\mathcal{R}^{\gen,\Delta})\ar[r, "{\Nhc(\Sigma_{F^1})}"] & \Nhc(\mathcal{R}^{\gen,\Delta})\,
				\end{tikzcd}
			\end{equation*}
			where $\pi$ denotes the natural forgetful map from the under-category.
			\item \label{item:exredgen2b} It restricts to a map 
			\[\Nhc(\mathcal{R}^{\gen, \Delta})_{F^{\bullet}/}\supseteq \pi^{-1}(\Nhc(\mathcal{R}^{\Delta}))\to \Nhc(\mathcal{R}^{\Delta})_{F^{\bullet}/}.\]
		\end{enumerate}
		
	\end{enumerate}
\end{lemma}
\begin{proof}
	\begin{enumerate}
		\item 
		\begin{enumerate}
			\item 
			We first note that simply adding on $K^{\bullet}\oplus K_{\bullet}$ (respectively its pullback to some $\mathbb{A}^n\times X$) gives a well-defined functor $\Sigma_K:\tilde{\mathcal{R}}^{\gen,\Delta}\to \tilde{\mathcal{R}}^{\gen,\Delta}$. This extends to $\mathcal{R}^{\gen,\Delta}$ by the functoriality of the $\Ex^{\infty}$-construction. Moreover, it is straightforward to check that this preserves $\mathcal{R}^{\Delta}$. 
			\item 
			Giving $\Xi_K$ is the same as giving a natural transformation of the simplicial functors $\xi_K:\id_{\mathcal{R}^{\gen,\Delta}}\Rightarrow \Sigma_K$ with the specified values $\xi_K(E^{\bullet})$. It remains to show that the specified data indeed gives a natural transformation, i.e. that for any $F^{\bullet},E^{\bullet}\in \mathcal{R}^{\gen,\Delta}$ and any simplex $\zeta:\Delta^n\to \Hom_{\mathcal{R}^{\gen,\Delta}}(F^{\bullet}, E^{\bullet})$, one has 
			\[\xi_K(E^{\bullet})\circ \zeta = \zeta\circ \xi_K(F^{\bullet}).\]
			Unraveling the definitions (as in the proof of Proposition \ref{prop:cofiltered}), one has that $\zeta$ is given by a map $\operatorname{Sd}^k(\Delta^n)\to \Hom_{\tilde{\mathcal{R}}^{\gen,\Delta}}(F^{\bullet},E^{\bullet})$, or equivalently a family of generalized isotropic reductions $F^{\bullet}\rightsquigarrow E^{\bullet}$ over $\bA(\operatorname{Sd}^k(\Delta^n))$, and the compositions we need to compute take place as compositions of isotropic reductions on $\bA(\operatorname{Sd}^k(\Delta^n))\times X$.
			
			But for any generalized isotropic reduction \[\zeta_T:\pr_2^*F^{\bullet}\leftarrow{f^T_{\bullet}} A_T^{\bullet}\rightarrow{e^T_{\bullet}} \pr_2^*E^{\bullet}\] over some $T\times X$ for some affine scheme $T$, one checks that the compositions $\zeta_T \circ \pr_2^*\xi_{K}(F^{\bullet})$ and $\pr_2^*\xi_{K}(E^{\bullet})\circ \zeta_T$ are both equal to 
			\[\pr_2^*F^{\bullet} \xleftarrow{f^T_{\bullet}\oplus 0} A^{\bullet}\oplus K^{\bullet} \xrightarrow{e_{\bullet}\oplus \id_{K^{\bullet}}\oplus 0} E^{\bullet}\oplus K^{\bullet}\oplus K_{\bullet}.\]
			
			That $\Xi_K$ preserves isotropic reductions follows from the fact that both the identity and $\Sigma_K$, preserve isotropic reductions, and that each $\xi_{K}(E^{\bullet})$ is itself an isotropic reduction. 
		\end{enumerate}
		\item 
		
		Let $\mathcal{R}^{\gen,\Delta}_{F\triangleleft}$ denote the simplicial sub-category of $\mathcal{R}^{\gen,\Delta}\times [1]$ on the objects $(F^{\bullet},0)$ and $\{(E^{\bullet},1)\}_{E^{\bullet}\in \mathcal{R}^{\gen,\Delta}}$ and with morphism spaces 
		\[\Hom((E_1^{\bullet},i), (E_2^{\bullet},j)) = \begin{cases}
			\emptyset, & i>j\\
			\id_{F^{\bullet}}, & i = j = 0\\
			\Hom_{\mathcal{R}^{\gen, \Delta}}(E_1^{\bullet},E_2^{\bullet}), & j=1.
		\end{cases}\]

		We have a functor $\tilde{\Omega}: \mathcal{R}^{\gen,\Delta}_{F\triangleleft}\to  \mathcal{R}^{\gen,\Delta}_{F\triangleleft}$ that respects the projections to $[1]$, with
		
		\[\tilde{\Omega}|_{\{F^{\bullet}\}\times \{0\}} =\operatorname{Id}_{\{F^{\bullet}\}\times \{0\}};\qquad \tilde{\Omega}|_{\mathcal{R}^{\gen,\Delta}\times \{1\}} = \Sigma_{F^1}\times \id_{\{1\} }.\]
		and such that the morphism  
		\begin{equation*}
			\begin{tikzcd}
				\Hom_{\mathcal{R}^{\gen,\Delta}}(F^{\bullet},E^{\bullet})\ar[r, dashed]\ar[d, equals]& \Hom_{\mathcal{R}^{\gen,\Delta}}(F^{\bullet},\Sigma_{F^1}(E^{\bullet}))\ar[d, equals] \\
				\Hom_{\FonR}((F^{\bullet},0), (E^{\bullet},1)) \ar[r,"\tilde{\Omega}"] & \Hom_{\FonR}((F^{\bullet},0),(\Sigma_{F^1}E^{\bullet},1))
			\end{tikzcd}
		\end{equation*}
		is obtained by taking the dashed map to be induced by the morphism of presheaves on affine schemes 
		\begin{align}\label{eq:hat-functor-def}
			\widehat{ }:\underline{\Hom}_{\mathcal{R}^{\gen}}(F^{\bullet},E^{\bullet}) & \to \underline{\Hom}_{\mathcal{R}}(F^{\bullet},\Sigma_{F^1}E^{\bullet}) 
		\end{align}
		given on $T$-points by
		\begin{align*}
			\zeta_T &\mapsto \widehat{\zeta_T} 
		\end{align*}
		for $\widehat{\zeta_T}$ is as in \eqref{eq:make-reduction} of Example \ref{ex:reductions} \ref{item:reductionsii}, applied on $T\times X$. That this gives a well-defined functor $\tilde{\Omega}$ follows from the formula
		\[\widehat{\xi_T\circ \zeta_T}  = \Sigma_{F^1}(\xi_T)\circ \widehat{\zeta_T}.\]
		
		By definition, an $n$-simplex of $\Nhc(\mathcal{R}^{\gen,\Delta})_{F^{\bullet}/}$ is an $n+1$-simplex of $\Nhc(\mathcal{R}^{\gen,\Delta})$ whose zero-th vertex is $F^{\bullet}$. By definition of $\Nhc$, this is the same as a map of simplicial categories $\Path^{\Delta}(\Delta^{n+1}) \to \mathcal{R}^{\gen,\Delta}$, where $\Path^{\Delta}$ denotes the simplicial path category, sending the object $0$ of $\Path^{\Delta}(\Delta^{n+1})$ to $F^{\bullet}$. 
		
		This, in turn is the same as a map of simplicial categories over $[1]$
		\[\Path^{\Delta}(\Delta^{n+1})\to \FonR\]
		where we take $\Path^{\Delta}(\Delta^{n+1})\to [1]$ to be the unique map with $i\mapsto \max \{0,i\}$.
		Moreover, this correspondence is compatible with face and degeneracy maps. 
		We take $\Omega_{F^{\bullet}}$ to be the functor that -- under this correspondence -- sends 
		\[x: \Path^{\Delta}(\Delta^{n+1})\to \FonR\]
		to the composition $\tilde{\Omega}\circ x$.
		By construction, $\Omega_{F^{\bullet}}$ satisfies the desired property a).
		To see that it satisfies b), we note that \eqref{eq:hat-functor-def} in fact lands in $\underline{\Hom}_{\mathcal{R}}(F^{\bullet},E^{\bullet})$, and that $\Sigma_{F^1}$ preserves $\mathcal{R}^{\Delta}$.
	\end{enumerate}
\end{proof}
\subsection{Affine models for simplicial sets}\label{sec:affine-models}
We elaborate a bit on modeling certain finite simplicial sets by glueing affine spaces. 
Let $K$ be a finite simplicial set. Its \emph{category of simplices} of $K$, denoted $(\Delta/K)$ is the category whose objects are simplices $\Delta^n\to K$, and whose morphisms are morphisms in the simplex category compatible with the given map to $K$.  
The co-simplicial object $\bA^{\Delta}$ defines a functor $(\mathbb{A}^{\Delta}/K):(\Delta/K)\to \Aff$. 

We say that an affine scheme $\bA(K)$ \emph{is a geometric realization for} $K$, if it comes with morphisms $\bA^n\to \bA(K)$ for each $n$-simplex of $K$ which realize $\bA(K)$ as a colimit over $(\mathbb{A}^{\Delta}/K)$ in the category of schemes. 
We can find a nice class of simplicial for which we have geometric realizations:

\begin{definition}
	Let $K$ be a simplicial sub-set of $\Delta^N$. We define $\mathbb{A}(K)$ to be the union of the images of all non-degenerate simplices of $K$ with its reduced structure as a closed sub-scheme of $\mathbb{A}^N$. 
\end{definition} 

\begin{lemma}\label{lem:geom-real}
	The scheme $\mathbb{A}(K)$ is a geometric realization of $K$. 
\end{lemma}
\begin{proof}
	This is true if $K$ is empty, so assume otherwise. 
	
	We proceed by induction on the maximal dimension of a non-degenerate simplex of $K$, and on the number of maximal-dimensional simplices of $K$. 
	
	Let $(\Delta/K)^{\operatorname{nd}}\subseteq (\Delta/K)$ denote the full sub-category whose objects are the non-degenerate simplices of $K$. By the argument of \cite[Variant 4.2.3.15]{lurie2006higher}, the inclusion $(\Delta/K)^{\operatorname{nd}}\to (\Delta/K)$ is a cofinal functor (see e.g. \cite[\href{https://stacks.math.columbia.edu/tag/09WN}{Tag 09WN}]{stacks-project}). In particular, it suffices to consider the restriction of $(\bA^{\Delta}/K)$ to $(\Delta/K)^{\operatorname{nd}}$ to compute the colimit (and whether it exists).

	Let $x:\Delta^n\to K$ be a non-degenerate simplex of $K$ of maximal dimension, and let $K'\subset K$ be the simplicial subset generated by all other non-degenerate simplices. Since $K$ is a subset of $\Delta^N$, we have that $x$ is a monomorphism, and that we have a pushout-diagram 
	\begin{equation}\label{eq:simplex-glueing-diag}
		\begin{tikzcd}
			\partial \Delta^n \ar[r]\ar[d]& \Delta^n\ar[d, "x"] \\
			K'\ar[r] & K
		\end{tikzcd}
	\end{equation} 
	where the upper horizontal map is the natural boundary inclusion.
	
	By inductive assumption, we have the geometric realizations $\bA(\partial \Delta^n)$ and $\bA(K')$ inside $\bA^N$. It is then enough to show
	that the induced diagram 
	\begin{equation*}
		\begin{tikzcd}
			\bA(\partial \Delta^N)\ar[r]\ar[d]& \bA^n\ar[d] \\
			\bA(K')\ar[r] & \bA(K)
		\end{tikzcd}
	\end{equation*}
	is a pushout-diagram of schemes. 
	By \cite[Theorem 3.11]{schwede2005gluing} it is enough to show that the intersection of $\bA(K')$ with the image of $\bA^n$ inside $\bA^N$ is equal to $\bA(\partial\Delta^n)$. 
	Let $X_1, \ldots,X_N$ denote the coordinates on $\bA^N$, with $X_0:= 1 - X_1-\cdots -X_N$.
	Then for any subset $I\subseteq [0,N]$ the corresponding embedding $\bA^{\abs{I}}\to \bA^N$ has image cut out by the ideal $\mathfrak{a}_I:=(X_j)_{j\in [0,N]\setminus I}$. 
	Let $I_x$ be the subset corresponding to $x:\Delta^n\to K\to \bA^N$. Up to an affine symmetry of $\bA^N$ and re-ordering, we may assume that $I_x = [0,n]$, so that $\mathfrak{a}_{I_x}= (X_{n+1}, \ldots,X_N)$. Then $\bA(K')$ is cut out by the ideal 
	\[\mathfrak{a}_K:= \bigcap_{I\subset K'} \mathfrak{a}_I.\]
	By Lemma \ref{lem:ideal-intersection} below, we have: 
	\begin{equation}\label{eq:ideal-equality}
		(X_{n+1}, \ldots X_N) + \bigcap_{I\subset K'} \mathfrak{a}_I  = \bigcap_{I\subset K'} \left(\mathfrak{a}_I + (X_{n+1}, \ldots X_N)\right)
	\end{equation}
	
	Here, the minimal ideals appearing on the right hand side are exactly those where $[0,n]\setminus I$ is a single element, and since $\partial x\subset K'$, all of these occur. Thus, we may simplify \eqref{eq:ideal-equality} to 
	\[\bigcap_{i=0}^n (X_i,X_{n+1}, \ldots,X_N),\]
	which is exactly the ideal of $\mathbb{A}(\partial\Delta^n)$ in $\bA^N$  as desired. 	
\end{proof}
We can strengthen the result as follows:
\begin{proposition}\label{prop:geom-real}
	For any scheme $X$, the product $\bA(K)\times X$ is a co-limit of the functor \begin{align*}
		(\bA^{\Delta}/K)\times X:(\Delta/K)&\to \operatorname{Sch}\\
		x &\mapsto (\bA^{\Delta}/K)(x) \times X
	\end{align*}
\end{proposition}
\begin{proof}
	By the same arguments as in Lemma \ref{lem:geom-real}, one reduces to checking that for any diagram \eqref{eq:simplex-glueing-diag}, the induced diagram 
	\begin{equation*}
		\begin{tikzcd}
			\bA(\partial \Delta^n)\times X\ar[r]\ar[d]& \bA^n\times X\ar[d] \\
			\bA(K')\times X \ar[r] & \bA(K)\times X
		\end{tikzcd}
	\end{equation*} 
	is a pushout diagram. 	
	In the category of schemes, this follows from \cite[Theorem 3.11]{schwede2005gluing}. Indeed, as sub-schemes of $\bA^N\times X$, we have $\bA^n\cap \bA(K') = \bA(\partial\Delta^n)$: This can be checked affine-locally on $X$, and the affine case was treated in the proof of Lemma \ref{lem:geom-real}
\end{proof}

\begin{proposition}\label{prop:geom-real-adjunction}
	Let $X$ $\bE$ be as in \S \ref{sec:simpl-struc}. 
	For any simplicial set $K$ that is isomorphic to a simplicial subset of some $\Delta^N$, and for any two self-dual representatives $E^{\bullet},F^{\bullet}$of $\bE$ there is a bijection between maps  of simplicial sets
	\[K\to \Hom_{\tilde{\mathcal{R}}^{\gen,\Delta}}(F^{\bullet},E^{\bullet})\] and maps of pre-sheaves
	\[\mathbb{A}(K)\to \underline{\Hom}_{\mathcal{R}^{\gen}}(F^{\bullet}, E^{\bullet}).\]
	Moreover, this correspondence is functorial in $K$ and respects the sub-categories of genuine isotropic reductions. 
\end{proposition}
\begin{proof}
	The datum of a family of generalized isotropic reductions on $T\times X$ consists of a collection of maps between vector bundles subject to a number of locally closed conditions. Thus, there is a moduli scheme $M^{iso-red}$ parametrizing such objects, and a family of generalized isotropic reductions over $T$ is the same as a map $T\times X\to M^{iso-red}$. 
	
	By definition, a map $K\to \Hom_{\tilde{\mathcal{R}}^{\gen,\Delta}}(F^{\bullet},E^{\bullet})$ is then the datum of a compatible collection of maps $(\bA/K)(x)\to \mathcal{M}$ over the index category $(\Delta/K)$. By Proposition \ref{prop:geom-real}, this is the same as a map $\bA(K)\to M^{iso-red}$. 
\end{proof}
We will need to consider how affine realizations interact with products and subdivisions
\begin{example}
	Let $K\subset \Delta^N$ be a simplicial set with geometric realization $\bA(K)\subseteq \bA^N$. The product $K\times \Delta^1$ is naturally a subset of $\Delta^{2N+1}$ via the inclusion $\Delta^N\times \Delta^1\to \Delta^{2N+1}$, given by $(i,j)\mapsto i+(N+1)j$ for $i\in [0,N]$ and $j\in [0,1]$. In particular, we have a geometric realization $\bA(K\times\Delta^1)$ for $K\times\Delta^1$. 
	
	On the other hand, we can consider the product $\bA(K)\times \bA(\Delta^1) = \bA(K)\times \bA^1$. We have a comparison map $\Gamma:\bA(K\times\Delta^1)\to \bA(K)\times \bA^1$, obtained as the restriction of the affine linear map $\bA^{2N+1}\to \bA^{N+1}$ sending $\mathsf{v}_{i+(N+1)j}$ to $\mathsf{v}_i\times \{j\}$ for $0\leq i\leq N$ and $j\in \{0,1\}$. The map $\Gamma$ agrees with the map obtained by functoriality of colimits and hence is independent of the choice of embedding $K\subseteq \Delta^N$.
	
	In particular, any morphism $\bA(K)\times \bA^1\to \underline{\Hom}(F^{\bullet}, E^{\bullet})$ induces a morphism $K\times \Delta^1\to \Hom_{\tilde{\mathcal{R}}^{\gen,\Delta}}(F^{\bullet},E^{\bullet})$ by composition with $\Gamma$ and the correspondence of Proposition \ref{prop:geom-real-adjunction}. 
\end{example}

\begin{example}\label{ex:subdivs}
	Let $K\subseteq \Delta^N$ be a simplicial subset. Applying the simplicial subdivision functor gives rise to a simplicial subset 
	\[\Sd(K)\subseteq \Sd(\Delta^N)\subseteq \Delta^M,\]
	where $M = 2^{n+1}-1$. In particular, $\Sd(K)$ possesses a geometric realization, as does any iterated subdivision. 
\end{example}

\begin{lemma}\label{lem:ideal-intersection}
	Let $R$ be a ring and $\mathfrak{J}$ be a set of subsets of $[0,N]$ and let $0\leq n\leq N$. Then we have the following equality of ideals of $R[X_1,\ldots,X_N]$: 
	\begin{equation}\label{eq:ideals-equal}
		(X_{n+1}, \ldots, X_N) + \bigcap_{J\in \mathfrak{J}} (X_i)_{i\in J} = \bigcap_{J\in \mathfrak{J}} \left((X_i)_{i\in J} + (X_{n+1}, \ldots, X_N)\right),
	\end{equation}
	where we set $X_0:= 1 - X_1 -\cdots -X_N$. 
\end{lemma}
\begin{proof}
	The inclusion "$\subseteq$" is straightforward.  To show that it is an equality, we may without loss of generality assume that all $J\in \mathfrak{J}$ satisfy $J\subset [0,n]$, since replacing $\mathfrak{J}$ with $\{J\cap[0,n] | J\in \mathfrak{J}\}$ leaves the right hand side unchanged, while possibly making the left hand side smaller. 
	
	Since $(X_0,\ldots,X_N) = 1$, it suffices to check the equality after localizing at each $X_i$ respectively. For $i\geq n+1$, both sides become the unit ideal. For any $0\leq i\leq n$, the terms in the intersection for which $i\in J$ become the unit ideal after localizing at $X_i$, so we may leave them out. This reduces us to the case that there exists some $i_0\in [0,n]$ such that for all $J\in \mathfrak{J}$, we have $i\not\in J$. By symmetry, we may assume that $i_0=0$. So, suppose that $\mathfrak{J}$ is a collection of sub-sets of $[1,n]$, and let $f\in R[X_1,\ldots,X_N]$ be contained in the right hand side of \eqref{eq:ideals-equal}. We may write 
	\[f = g + \sum_{i=n+1}^N X_i f_i\]
	where $g\in R[X_1,\ldots,X_n]$ is the sum of all monomial terms of $f$ not divisible by any $X_i$ for $n+1\leq i\leq N$. Then, for any given $J\in \mathfrak{J}$, the condition that $f\in (X_i)_{i\in J} + (X_{n+1}, \ldots,X_N)$ is equivalent to every monomial term in $f$ being divisible by some $X_i\in J\cup[n+1,N]$, hence implies that  $g\in (X_i)_{i\in J}$. Thus, 
	\[g\in \bigcap_{J\in \mathfrak{J}} (X_i)_{i\in \mathfrak{J}},\]
	which shows that $f$ is contained in the left hand side of \eqref{eq:ideals-equal} as desired.  	
\end{proof}

\section{The spin functor}\label{sec:spin-functor}
We keep the setting and notation of \S \ref{sec:self-dual-cat}: $X$ is an affine scheme, carrying a fixed quadratic complex $(\bE, \theta)$. As defined there, let $\mathcal{R}$ denote the category of self-dual representatives of $\bE$ (Definition \ref{def:iso-red}), and $\mathcal{R}^{\Delta}$ its simplicial enrichment (Definition \ref{def:simplicial-enrichment-reductions}). We now also assume $\bE$ is equipped with an orientation. 

In this section, we construct a pseudo-functor $\bS:\mathcal{R}\to \Grpd$, which to an object $E^{\bullet}\in \mathcal{R}$ associates the groupoid of \emph{spin structures on $E^{\bullet}$} (Definition \ref{def:spin-functor}). Generalizing the discussion in \ref{sec:sketch}, we define in \S \ref{sec:constr-spin-functor} a spin structure on $E^{\bullet}$ to be a $\Spin^{\bC}$-structure \footnote{See \ref{sec:appendix-spin-groups}. Not working over $\bC$, this should maybe more appropriately be called a \emph{special Lipschitz structure}.} on the $\SO(m)\times \bG_m$-principal bundle associated to the pair 
\[(E^0, \det E_+^{\bullet}).\]
The data of $\mathbb{S}$ gives rise to a category fibered in groupoids
\[\mathcal{S}\to \mathcal{R}\]
and we define a spin structure on $\bE$ to be a section of this category \ref{def:spin-structure-affine}
 
In the rest of the section, we study the structure of the functor $\bS$, or equivalently the fibration $\mathcal{S}\to \mathcal{R}$. In \S \ref{subsec:enriched-functors}, we show that $\bS$ extends to a functor on the simplicially enriched category $\mathcal{R}^{\Delta}$. Since we already showed this is a weakly contractible category, we can quickly conclude in \S \ref{sec:spin-functor-structure} that the functor $\bS$ is essentially constant, or equivalently, that $\mathcal{S}$ is equivalent to a constant fibration (Theorem \ref{thm:spin-functor-structure}). 

\subsection{Construction of the functor}\label{sec:constr-spin-functor}

\begin{definition}	\label{def:pair-cat}
	Let $\mathcal{P}$ denote the category given as follows
	\begin{enumerate}
		\item The objects of $\mathcal{P}$ are pairs $(E, L)$ where $E$ is an $\SO(m)$-bundle for some $m$, and $L$ is a line bundle
		\item A morphism $(F,M)\to (E,L)$ in $\mathcal{P}$ consists of an isotropic sub-space $K\subseteq E$, which is required to be positive if $2k = m$, together with isomorphisms $q:F \to K^{\perp}/K$ and $\varphi: M \to L\otimes \det K$. 
		\item Composition of morphisms $(K,q,\varphi):(F,M)\to (E,L)$ and $(J,p,\psi):(G,N)\to (F,M)$ is defined as follows: Let $\pi_K:K^{\perp}\to K^{\perp}/K$ and $\pi_{J}:J^{\perp}\to J^{\perp}/J$ denote the natural projections. Then the subspace
		\[\Lambda:= \pi_K^{-1}(qJ) \subseteq E\]
		is isotropic, and one has natural isomorphisms $ J^{\perp}/J\to \Lambda^{\perp}/\Lambda$ and $\det\Lambda \to \det K \otimes \det J$. 
		We take the composition $(K,q,\varphi) \circ (J,p,\psi)$ to be $(\Lambda, r, \rho)$, where $r$ is the composition 
		\[G\xrightarrow{p} J^{\perp}/J \to  \Lambda^{\perp}/\Lambda\]
		and where $\rho$ is the composition 
		\[N\xrightarrow{\psi} M\otimes \det J \xrightarrow{\varphi\otimes \det J} L\otimes \det K\otimes \det J \to L\otimes \det \Lambda\] \label{item:pair-catiii}
	\end{enumerate}	
\end{definition}
In particular, each object of $\mathcal{P}$ corresponds to a principal bundle with structure group $\SO(m)\times \mathbb{G}_m$, and a morphism in $\mathcal{P}$ gives rise to a diagram of principal bundles \eqref{diag:principal-bundles-pairs} over the diagram of groups \eqref{eq:group-diag-twisted}.

\begin{construction}\label{constr:to-pair-functor}
	Define a functor $\mathbb{P}:\mathcal{R}\to \mathcal{P}$ as follows:
	\begin{enumerate}
		\item On objects it is given by   
		\[\bP:E^{\bullet}\mapsto (E^0,\det (E_+^{\bullet}[1])).\]
		\item For a morphism $\xi:F^{\bullet}\rightsquigarrow E^{\bullet}$ given by the isotropic reduction 
		\[F^{\bullet} \xleftarrow{f_{\bullet}} A^{\bullet} \xrightarrow{e_{\bullet}} E^{\bullet},\]
		by Lemma \ref{lem:iso-red-gives-iso-red}, we have the isotropic sub-bundle $K:= \Ker(E^0\xrightarrow{e_0^{\vee}} A_0)$, an induced isomorphism $q:K^{\perp}/K\to F^0$ and a natural isomorphism $\varphi:\det F_+^{\bullet}[1] \to \det  E^{\bullet}_+[1] \otimes \det K$.
		We set 
		\[\bP(\xi):= (K,q,\varphi).\]
	\end{enumerate}
	That this defines a functor, i.e. that $\bP(\xi\circ \zeta) = \bP(\xi)\circ \bP(\zeta)$ follows from a straightforward but slightly tedious calculation involving the commutative diagram $\eqref{eq:composition-diag}$ and similar arguments as used in the proof of Lemma \ref{lem:iso-red-gives-iso-red}. 
\end{construction}

\begin{remark}\label{rem:twisted-pair}
	Let $\mathcal{N}$ be a fixed line bundle on $X$. Then one can also consider a twisted variant $\bP^{\mathcal{N}}$, obtained by postcomposing $\bP$ with the functor 
	\begin{align*}
	\mathcal{N}\otimes -: \mathcal{P} & \to \mathcal{P}	\\
		(E,L)&\mapsto (E,\mathcal{N}\otimes L) 
	\end{align*}
\end{remark}

For an object $(E,L)\in \mathcal{P}$, let $\bS(E,L)$ denote the groupoid of $\Spin^{\bC}$-structures over the associated $\SO(m)\times \mathbb{G}_m$-principal bundle. Give a morphism $\xi: (F,M)\rightsquigarrow (E,L)$ in $\mathcal{P}$, we have the associated pullback and pusforward functors $\xi^*$ and $\xi_*$ obtained from the diagram \eqref{diag:spin-structure-transfer}, as explained there. 
\begin{lemma}\label{lem:spin-to-pairs-functor}
	The construction $(E,L)\mapsto \bS(E)$ and $\xi\mapsto \bS(\xi):=\xi^*$ naturally defines a (contravariant) pseudo-functor $\bS:\mathcal{P}\to \Grpd^{\simeq}$. 
\end{lemma}
\begin{proof}
	Given composable morphisms $\xi,\zeta$ in $\mathcal{P}$, we need to specify a $2$-morphism $\mu_{\xi,\zeta}:\zeta^* \xi^* \Rightarrow (\xi \zeta)^*$ (cf. \cite[Part 1, Definition 3.10]{FGAE} One also needs to specify similar data for identity morphisms, which works in the obvious way and is left out here). 
	Say we have $\xi: (F,M)\rightsquigarrow (E,L)$ given by $(K,q,\varphi)$ and $\zeta:(G,N)\rightsquigarrow (F,M)$ given by $(J,p,\psi)$, so that their composition is given by  $(\Lambda, r, \rho)$ constructed as in Definition \ref{def:pair-cat} \ref{item:pair-catiii}.
	Let $m, k, j$ be the ranks of $E,K$ and $J$ respectively, and consider the subgroup $\G(m,k,j)\subseteq \SO(m)$ preserving the two-step filtration $\mathcal{O}^k\times \{0\}\subseteq \mathcal{O}^{k+j}\times \{0\} \subseteq \mathcal{O}^m$. \footnote{For $\SO(m)$ as chosen in \ref{sec:standard-forms}} Then we have a commutative diagram 
	\begin{equation}\label{eq:big-group-diag}
		\begin{tikzcd}[row sep = small,column sep = tiny]
			& & \G(m,k+j)\times \bG_m\ar[dddll, bend right]\ar[dddrr, bend left]& & \\
			& & \G(m,k,j)\times \bG_m\ar[dl,"{\Gprojtw_{m,k,j}}"]\ar[dr]\ar[u]& &  \\
			& \G(n,j)\times \bG_m\ar[dl,"{\Gprojtw_{n,j}}"]\ar[dr,"{\Gsub\times \id_{\bG_m}}"]& &\G(m,k)\times \bG_m\ar[dl,"{\Gprojtw_{m,k}}"]\ar[dr,"{\Gsub\times \id_{\bG_m}}"] & \,\\
			\SO(\ell)\times \bG_m& &\SO(n)\times \bG_m& &\SO(m)\times \bG_m
		\end{tikzcd}
	\end{equation}
	where $n:=m-k$ and $p:=n-j = m-k-j$. 
	Here, all the maps to the lower right and the upward vertical map are an inclusion in the first and identity on the second factor.
	
	The maps to the lower left are the ones of \ref{eq:group-diag-twisted}, except for $\Gprojtw_{m,k,j}$ which is the restriction of $\Gprojtw_{m,k}$.
	We claim that over the diagram \eqref{eq:big-group-diag}, we have a natural diagram of double covers with the $\Spin^{\bC}$-groups in the bottom row. In view of \eqref{eq:diag-groups-extended}, the only thing that needs to be checked, is that the subgroups defining the double covering groups $\widetilde{\G(a,b)\times \G_m}$ pull back to the same subgroup of $\pi_1(\G(m,k,j)\times \bG_m)$, but this already follows from \eqref{eq:diag-groups-extended} together with the commutativity of \eqref{eq:big-group-diag}.	
	
	Now, let $P_E\times P_L$ be the principal bundle $\SO(m)\times \mathbb{G}_m$-bundle associated to $(E,L)$, so that $\bS(E,L)$ is the groupoid of reductions of structure groups to $\Spin^{\bC}(m)$. Taking base changes and induced bundle along the lower zig-zag of \eqref{eq:big-group-diag}, gives the functor  $\zeta^*\xi^*:\bS(E,L)\to \bS(G,N)$, while traversing via the top yields $(\xi\zeta)^*$. 
	
	By the commutativity of \eqref{eq:big-group-diag}, both of these functors are canonically isomorphic to the functor given by pulling back and taking induced bundle from the group $\widetilde{\G(m,k,j)\times \bG_m}$. 
	
	Finally, one needs to check that these $2$-isomorphisms satisfy the associativity constraints for a pseudo-functor. This involves a similar argument using a group $\G(m,k,j,\ell)$ for three-step filtrations which yields a tetrahedral diagram with \eqref{eq:big-group-diag} as faces. The details are left out here.
\end{proof}

\begin{definition}\label{def:spin-functor}
	\begin{itemize}
		\item	The \emph{spin functor} is the pseudo-functor given as the composition \[\mathcal{R}\xrightarrow{\bP} \mathcal{P}\xrightarrow{\bS} \Grpd^{\simeq}.\] By abuse of notation, we again denote it by $\bS$. Note that it depends on $\bE$ as an oriented quadratic complex. For any $E^{\bullet}$, we say that $\sigma\in \bS(E^{\bullet})$ is a \emph{spin-structure} on the representative $E^{\bullet}$.
		\item 	We let $\mathcal{S}\to \mathcal{R}$ denote the category fibered in groupoids associated to the pseudo-functor $\bS$. Concretely, objects of $\mathcal{S}$ are pairs $(E^{\bullet},\sigma)$ with $\sigma\in \bS(E^{\bullet})$. A morphism $\Xi:(F^{\bullet},\tau)\to (E^{\bullet}, \sigma)$ in $\mathcal{S}$ consists of a morphism $\xi:F^{\bullet}\rightsquigarrow E^{\bullet}$ in $\mathcal{R}$ together with an isomorphism $\tau \to \xi^*\sigma$ in $\bS(F^{\bullet})$.		 	
	\end{itemize}	
\end{definition}

Now we can give a local definition of a spin structure on a quadratic complex which does not depend on a choice of representative:

\begin{definition}\label{def:spin-structure-affine}
	We let $\bS(\bE)$ denote the groupoid of sections of $\mathcal{S}\to \mathcal{R}$. We say that objects of $\bS(\bE)$ are \emph{spin structures on $\bE$}. 
\end{definition}

\subsection{Simplicial enrichments}\label{subsec:enriched-functors}
We would like to enhance the functor $\bS:\mathcal{R}\to \Grpd^{\simeq}$ to a functor from $\mathcal{R}^{\Delta}$.  

We have the obvious notion of families of maps in $\mathcal{P}$ analogous to Definition \ref{def:family-reductions}, which allows us to define a simplicial enrichment on $\mathcal{P}$ as in Definitions \ref{def:simplicial-enrichment-reductions} and \ref{def:kan-enrichment-reductions}. 
\begin{definition}	
	For any scheme $T$, a \emph{family of maps $(F, M)\to (E,L)$ in $\mathcal{P}$} is given by an isotropic subspace $K_T\subset \pr_2^*E$, and isomorphisms $K_T^{\perp}/K_T\to \pr_2^*F$, $L\otimes \det K_T\to M$.
\end{definition}

\begin{definition}
	Let $\tilde{\mathcal{P}}^{\Delta}$ denote the simplicial category whose objects are the objects of $\mathcal{R}$, and such that for $(F,M), (E,L)$, the associated $\Hom$-simplicial set $\Hom_{\tilde{\mathcal{P}}^{\Delta}}((F,M),E,L) $is given by the composition
	\[\Delta \xrightarrow{\bA^{\Delta}}\operatorname{Schemes}\xrightarrow{\underline{Hom}_{\mathcal{P}}((F,M)(E,L)) }\operatorname{Set}\]
	We also define 
	\[\mathcal{P}^{\Delta}:=\tilde{\mathcal{P}}^{\Delta,\Ex^{\infty}}.\] 
\end{definition}

\begin{proposition}\label{prop:pairfunc-simplicial}
	The functor $\bP$ of Construction \ref{constr:to-pair-functor} extends naturally to a functor of simplicial categories $\mathcal{R}^{\Delta}\to \mathcal{P}^{\Delta}$.
\end{proposition}
\begin{proof}
	Construction \ref{constr:to-pair-functor} works unchanged with $X$ replaced by $X\times T$ for any affine scheme $T$, and is compatible with pullbacks in $T$. Thus for any $F^{\bullet},E^{\bullet}\in \mathcal{R}$, we have a natural extension of $\bP$ to a functor $\underline{\Hom}_{\mathcal{R}}(F^{\bullet},E^{\bullet})\to \underline{\Hom}_{\mathcal{P}}(\bP(F^{\bullet}), \bP(E^{\bullet}))$, in a way compatible with composition of (families of) morphisms. This induces, in turn, an extension of $\bP$ to 
	\[\tilde{\mathcal{R}}^{\Delta}\to \tilde{\mathcal{P}}^{\Delta},\]
	and then, by functoriality of the $\Ex^{\infty}$-construction, to 
	\[\mathcal{R}^{\Delta}\to \mathcal{P}^{\Delta}.\]
\end{proof}

The slightly more subtle point is to extend the pseudo-functor $\bS:\mathcal{P}\to \Grpd^{\simeq}$ to a functor with domain $\mathcal{P}^{\Delta}$, since it is not immediately clear what a the correct notion of a pseudo-functor from a simplicial category to a $2$-category should be. 
A solution is to replace both types of objects by associated $\infty$-categories, so that a functor between them is now simply given by a map of simplicial sets. 

For $2$-categories, the role of the homotopy-coherent nerve is played by the \emph{Duskin nerve}, which to a $2$-category $\mathcal{D}$ associates an simplicial set $\ND(\mathcal{D})$, and which is functorial for strictly unitary functors of $2$-categories \cite[\href{https://kerodon.net/tag/00AU}{Theorem 00AU}]{kerodon}.\footnote{We silently modify all our pseudo-functors to be strictly unitary as described in \cite[\href{https://kerodon.net/tag/008U}{Tag 008U}]{kerodon}.} If $\mathcal{D}$ is a $(2,1)$-category, in the sense that all its $2$-morphisms are invertible, then its Duskin nerve is an infinity category \cite[\href{https://kerodon.net/tag/00AC}{Theorem 00AC}]{kerodon}. 

Hence, we would like to extend $\bS:\mathcal{P}\to \Grpd^{\simeq}$ in some way to a morphism $\Nhc(\mathcal{P}^{\Delta})\to \ND(\Grpd^{\simeq})$.
Since the target is obtained from a $2$-category, we will need only the $2$-categorical information contained in $\mathcal{P}^{\Delta}$, which is achieved by the following standard construction. 

To any simplicial category $\mathcal{C}^{\Delta}$, we associate its homotopy $2$-category $h_2\mathcal{C}$ as in \cite[\href{https://kerodon.net/tag/02BJ}{Construction 02BJ}]{kerodon}. By definition, it has the same objects as $\mathcal{C}^{\Delta}$, and for any $x,y\in \mathcal{C}^{\Delta}$, the $1$-category $\mathcal{\Hom}_{h_2\mathcal{C}}(x,y)$ is obtained as the homotopy category of the simplicial mapping space $\Hom_{\mathcal{C}^{\Delta}}(x,y)$. 

More concretely, for any $x,y\in \mathcal{C}^{\Delta}$, the category $\Hom_{h_2\mathcal{C}}(x,y)$ is the category whose objects are given by the $0$-simplices of $\Hom_{\mathcal{C}^{\Delta}}(x,y)$, and whose morphism set is freely generated by the set of $1$-simplices of $ \Hom_{\mathcal{C}^{\Delta}}(x,y)$ subject to the relations that $[\sigma_0(f)] = \id_f$ for any vertex $f\in\Hom_{\mathcal{C}^{\Delta}}(x,y)_0$, and $\xi_1\circ \xi_{0} = \xi_2$ for any $2$-simplex in $\Hom_{\mathcal{C}^{\Delta}}(x,y)$ with edges $\xi_0,\xi_1, \xi_2$. 
By \cite[\href{https://kerodon.net/tag/02BM}{Remark 02BM}]{kerodon}, there is a natural comparison map of $\infty$-categories $\Nhc(\mathcal{C}^{\Delta})\to \ND(h_2\mathcal{C})$

\begin{construction}\label{constr:pullback-simplicial-pairs}
	We define a natural extension of the pseudo-functor $\mathbb{S}:\mathcal{P}\to \Grpd^{\simeq}$ to a pseudo-functor of $2$-categories $h_2\mathcal{P}\to \Grpd^{\simeq}$ (also denoted $\bS$). 
	
	Since $\Grpd^{\simeq}$ has only invertible $2$-morphisms, it is equivalent to construct a pseudo-functor $h_2\tilde{\mathcal{P}}\to \Grpd^{\simeq}$ as we show in Corollary \ref{cor:2-cat} below. 
	
	Fix objects $(F,M)$ and $(E,L)$ of $\mathcal{P}$. We specify the functor
	\begin{equation}\label{eq:2catfunc-homcats}
		\bS:\Hom_{h_2\tilde{\mathcal{P}}}((F,M),(E,L))\to \Hom_{\Grpd^{\simeq}}(\mathbb{S}(E,L), \mathbb{S}(F,M) ),
	\end{equation}
	On objects, it is just given by $\xi \mapsto \mathbb{S}(\xi):=\xi^*$. 
	By definition, 
	\[\Hom_{h_2\widetilde{\mathcal{P}}}((F,M),(E,L)) = h_1\left(\Hom_{\widetilde{\mathcal{P}}^{\Delta}}((F,M), (E,L))\right).\]
	
	Thus, we need to specify $\mathbb{S}$ on $1$-simplices of $\Hom_{\mathcal{P}^{\Delta}}((F,M), (E,L))$ and show that the resulting map respects the relations defining the homotopy category. 
	
	A $1$-simplex of $\Hom_{\widetilde{\mathcal{P}}^{\Delta}}((F,M), (E,L))$ is given by a family of maps $\xi_t: (F,M)\to (E,L)$ over $\mathbb{A}^1$. 

	A spin structure $\sigma $ on $E$ gives rise to an $\mathbb{A}^1$-valued family of spin structures $\xi_t^*\sigma$ on $F$, with identifications $\xi_t^*\sigma|_{0}\simeq \xi_0^*\sigma$ and $\xi_t^*\sigma|_1 = \xi_1^*\sigma$. Since pullback gives an isomorphism\footnote{Here we use \'etale cohomology groups and, crucially, that $2$ is invertible on $X$.} 
	 \begin{equation}\label{map:cohom-pullback-iso}
	 \operatorname{H}^1(X, \ZT)\xrightarrow{\sim}	\operatorname{H}^1(\bA^1\times X, \ZT) ,
	 \end{equation} we get that the spin structure $\xi_t^*\sigma$ is isomorphic to $\pr_2^*\tau$ for some $\tau\in \bS(F,M)$. 	
	Restricting to $0$ and $1$, we get an isomorphism 
	\begin{equation}\label{eq:defeq-eta}		\xi_{0}^*\sigma \simeq \xi^*\sigma|_0 \simeq \operatorname{pr}_2^*\tau|_0\simeq \tau\simeq \operatorname{pr}_2^*\tau|_1\simeq  \xi^*\sigma|_1\simeq \xi_1^*\sigma.
	\end{equation}
	We note that this isomorphism is in fact independent of the choice of $\tau$, and functorial in $\sigma$, and hence gives a well-defined natural transformation $\eta_{\xi}:\xi_0^*\Rightarrow \xi_1^*$ of functors $\bS(F,M)\to \bS(E,L)$. The inverse is obtained from the family obtained by reparametrizing $\xi$ via $t\mapsto 1-t$.   
	
	If $\xi$ is a degenerate $1$-simplex, i.e. a constant family of morphisms, then we can take $\tau = \xi_0^*f$, and see that the induced isomorphism is the identity.
	
	Now suppose that we have a $2$-simplex in $\Hom_{\tilde{\mathcal{P}}^{\Delta}}((F,M), (E,L))$ with restrictions $\xi_{01}, \xi_{12}$ and $\xi_{02}$ to the respective edges. This is given by a family of morphisms over $\mathbb{A}^2$. Then, using \eqref{map:cohom-pullback-iso} twice (once with $X$ replaced by $X\times \bA^1$), the family of spin structures $\xi_{\bA^2}^*\sigma$ is isomorphic to $\pr_2^*\tau$ for $\tau\in \bS(F,M)$. We can use the same $\tau$ to define the natural transformations induced by $\xi_{01}, \xi_{12}$ and $\xi_{02}$. Thus, we have the following commutative diagram, showing that $\eta_{\xi_{12}}\circ \eta_{\xi_{01}} = \eta_{\xi_{02}}$.
	\begin{equation*}
		\begin{tikzcd}
			&               & \xi_{0}^*\sigma \arrow[ld, "\sim"] \arrow[rd, , "\sim"']\ar[dddrr, bend left,"{\eta_{\xi_{02}}(\sigma)}"] \ar[dddll, bend right,"{\eta_{\xi_{01}}(\sigma)}"'] &               &                         \\
			&   |[xshift=1.2em,overlay]|  \xi_{01}^*\sigma|_0 \arrow[rd, , "\sim"] &                                       &  |[xshift=-1.2em,overlay]| \xi_{02}^*\sigma|_0 \arrow[ld, , "\sim"'] &                         \\
			& \xi_{01}^*\sigma|_1 \arrow[r, , "\sim"]  & \tau                                    & \xi_{02}^*\sigma|_2 \arrow[l, , "\sim"']  &                         \\
			\xi_1^*\sigma \arrow[ru, , "\sim"] \arrow[r, , "\sim"] \ar[rrrr, bend right,"{\eta_{\xi_{12}}(\sigma)}"]& \xi_{12}^*\sigma|_{1} \arrow[ru, , "\sim"] &                                       & \xi_{12}^*\sigma|_{2} \arrow[lu, , "\sim"'] & \xi_{2}^*\sigma \arrow[lu, , "\sim"'] \arrow[l, , "\sim"']
		\end{tikzcd}
	\end{equation*}
	This concludes the construction of \eqref{eq:2catfunc-homcats}. 
	
	Next, we need to specify the compatibility data for composition of morphisms. That is, we need to show that the $2$-morphisms $\mu_{\xi,\zeta}$ constructed in the proof of Lemma \ref{lem:spin-to-pairs-functor} assemble to a natural transformation of functors $\Hom_{h_2\tilde{\mathcal{P}}}(x,y)\times \Hom_{h_2\tilde{\mathcal{P}}}(y,z)\to \Hom_{\Grpd^{\simeq}}(\bS(x), \bS(z))$.
	Concretely, we need to show that for every morphism $(\eta,\nu):(\zeta_0,\xi_0) \to (\zeta_1,\xi_1) $ in the category $\Hom_{h_2\tilde{\mathcal{P}}}(x,y)\times \Hom_{h_2\tilde{\mathcal{P}}}(y,z)$, the following diagram of $2$-morphisms in $\Grpd$ commutes
	\begin{equation}\label{eq:2cat-comm-diag}
		\begin{tikzcd}
			\bS(\zeta_0)\circ \bS(\xi_0)\ar[r, Rightarrow, "{\mu_{\xi_0,\zeta_0}}"]\ar[d,Rightarrow, "{\bS(\eta)\circ \bS(\nu)}"]
			& \bS(\xi_0\circ \zeta_0) \ar[d, Rightarrow, "{\bS(\nu\circ \eta)}"]\\
			\bS(\zeta_1)\circ \bS(\xi_1)\ar[r, Rightarrow, "{\mu_{\xi_1,\zeta_1}}"] & \bS(\xi_1\circ \zeta_1)
		\end{tikzcd}
	\end{equation}
	Here, we use $\circ$ to denote both compositions of $1$-morphisms and horizontal composition of $1$-morphisms.
	It suffices to show commutativity of \eqref{eq:2cat-comm-diag} for $(\eta,\nu)$ ranging through a generating set of morphisms, so without loss of generality we may assume that either $\eta$ or $\nu$ is an identity morphism and that the other is given by a $1$-simplex of the respective simplicial set $\Hom_{\tilde{\mathcal{P}}^{\Delta}}(-,-)$. Assume, say, that $\nu = \id_{\xi}$ (the other case is exactly analogous) and that $\eta=\eta_{\zeta}:\zeta_0\Rightarrow\zeta_1$ is obtained from a family of maps $\zeta_t$ over $\bA^1$. 
	
	We now need to check that for a given $\sigma\in \bS(z)$, the following diagram commutes 
	\begin{equation}\label{eq:2cat-commdiag2}
		\begin{tikzcd}[column sep= large]
			\zeta_0^*\xi^*\sigma\ar[r,"{\mu_{\xi,\zeta_0}(\sigma)}"]\ar[d,"{\eta_{\zeta}(\xi^*\sigma)}"]& (\xi\circ \zeta_0)^*\sigma\ar[d,"{\eta_{\xi\circ \zeta}(\sigma)}"] \\
			\zeta_1^* \xi^*\sigma \ar[r,"{\mu_{\xi,\zeta_1}(\sigma)}"] & (\xi\circ \zeta_1)^*\sigma
		\end{tikzcd}
	\end{equation}
	
	By considering $\bA^1\times X$ in place of $X$, we have the natural transformation $\mu_{\pr_2^*\xi, \zeta_t}: \bS(\zeta_t)\circ \bS(\pr_2^{*}\xi) \Rightarrow \bS(\pr_2^*\xi \circ \zeta_t)$  of functors $\bS(\pr_2^*z)\to  \bS(\pr_2^*x)$, which restricts to $\mu_{\xi,\zeta_0}$ and $\mu_{\xi, \zeta_1}$ at $0$ and $1$ respectively.

	To compute $\eta_{\zeta}(\xi^*\sigma)$ via the defining formula \eqref{eq:defeq-eta}, we choose an isomorphism $\zeta_t^*(\pr_2^* (\xi^*\sigma)) \simeq \pr_2^*\tau$ for some $\tau\in \bS(y)$, and to compute $\eta_{\xi\circ \zeta}(\sigma)$, we take the composition 
	\[(\pr_2^*\xi\circ \zeta_t)^*\xi \xrightarrow{\mu_{\pr_2^*\xi,\zeta_t}(\sigma)^{-1}} \zeta_t^* (\pr_2^*\xi)^* \pr_2^*\sigma \simeq \pr_2^*\tau.\] 
	We get the following diagram of isomorphisms, which recovers the edges of \eqref{eq:2cat-commdiag2} by traversing between any two adjacent corners, and therefore shows that \eqref{eq:2cat-commdiag2} commutes.
	\begin{equation*}
		\begin{tikzcd}
			\zeta_0^*\xi^*\sigma\ar[r] &\zeta_t^*\xi^*\sigma|_{0} \ar[dr]& & (\xi\circ \zeta_t)^*\sigma|_0\ar[r]& (\xi\circ\zeta_0)^*\sigma\\
			& &\zeta_t^*\tau\ar[ur]\ar[dr]& & &\\
			\zeta_t^*\xi^*\sigma \ar[r] & \zeta_t^*\xi^*\sigma|_{1}\ar[ur]& & (\xi\circ \zeta_t)^*\sigma|_1\ar[r]& (\xi\circ\zeta_1)^*\sigma.
		\end{tikzcd}
	\end{equation*}
\end{construction}

\begin{lemma}\label{lem:homotopy-category-localization}
	Let $Y$ be a simplicial set. Then the natural map of homotopy categories $h_1(Y)\to h_1(\Ex^{\infty}(Y))$ identifies the latter as the localization of $h_1(Y)$ at the class of all morphisms.  
\end{lemma}
\begin{proof}
	In fact, already $h_1(Y)\to h_1(\Ex(Y))$ is the localization, and this is what we will prove. The statement of the lemma follows from this since, as a left adjoint, the homotopy category functor preserves colimits, and since $\Ex^{\infty}(Y)= \lim_{\rightarrow k} \Ex^k(Y)$.  
	Now, by construction, a $1$-simplex of $\Ex(Y)$ consists of a \emph{cospan} in $Y$, i.e. a pair of $1$-simplices $(f,g)$ with the same endpoint as such:
	\begin{equation*}
		\begin{tikzcd}
			\bullet \ar[r,"f"] & \bullet &\bullet\ar[l, "g"']
		\end{tikzcd}
	\end{equation*}
	Here, we view this as a $1$-simplex starting at the source of $f$ and ending at the source of $g$. The $1$-simplices lying in $Y$ are exactly those of the form $(f,\id)$.
	We claim that in $h^1(\Ex(Y))$, we have the relation $[(g,1)]\circ [(f,g)] = [(f,1)]$. Indeed, this is witnessed by the following $2$-simplex in $\Ex(Y)$, given by a map $\operatorname{Sd}(\Delta^2)\to Y$:
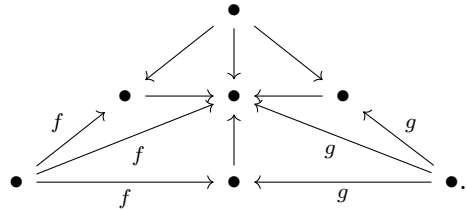
	\begin{equation*}
		\begin{tikzcd}
			&              & {\bullet} \arrow[ld] \arrow[rd] \arrow[d] &              &                                      \\
			& {\bullet} \arrow[r] & {\bullet}                                 & {\bullet} \arrow[l] &                                      \\
			{\bullet} \arrow[rr,"f"'] \arrow[rru,"f"'] \arrow[ru,"f"] &              & {\bullet} \arrow[u]                       &              & {\bullet} \arrow[llu,"g"] \arrow[ll,"g"] \arrow[lu,"g"'].
		\end{tikzcd}
	\end{equation*}
	Taking $f = \id$, this shows that every morphism in $h^1(Y)$ has an invertible image in $h^1(\Ex(Y))$. Thus, the map $h_1(Y)\to h_1(\Ex(Y))$ factors uniquely through the localization. Further, we see that $[(f,g)] = [(g,1)]^{-1}[(f,1)]$, so the map from the localization is surjective on morphisms. To see that it is an isomorphism, we claim that $[(f,g)]\mapsto [g]^{-1}\circ [f]$ gives a well-defined inverse. Concretely, we need to check that it respects the defining relations in $h_1(\Ex(Y))$. For the ones regarding identity morphisms, this is straightforward. Suppose we have a relation $[(f_3,g_3)] = [(f_2,g_2)]\circ [(f_1,g_1)]$ coming from a $2$-simplex $\Delta^2\to \Ex(Y)$, i.e. a map $r:\operatorname{Sd}(\Delta^2)\to Y$. We need to show that, in $h_1(Y)$, we have $[g_3]^{-1}[f_3] = [g_2]^{-1}[f_2][g_1]^{-1}[f_1]$. 
	But $r$ forces a commutative diagram in $h_1(Y)$, which after passing to the localization guarantees exactly this desired relation. 	  
\end{proof}
\begin{corollary}\label{cor:2-cat}
	Let $\mathcal{D}$ be a $2$-category in which all $2$-morphisms are invertible, and let $\mathcal{C}^{\Delta}$ be a simplicial category. Consider the locally Kan replacement $\mathcal{C}^{\Delta,\Ex^{\infty}}$ defined via \eqref{eq:hom-spaces-locally-kan-replacement}. Then any pseudo-functor of $2$-categories
	\[F:h_2(\mathcal{C}^{\Delta}) \to \mathcal{D}\]
	has a unique extension to a pseudo-functor 
	\[h_2(\mathcal{C}^{\Delta,\Ex^{\infty}})\to \mathcal{D}.\]
\end{corollary}
\begin{proof}
	By Lemma \ref{lem:homotopy-category-localization}, the $2$-category $h_2(\mathcal{C}^{\Delta, \Ex^\infty})$ is obtained from $h_2(\mathcal{C}^{\Delta})$ by replacing all $\Hom$-categories by their total localizations. By the assumption on $\mathcal{D}$, and the basic properties of localizations, a straightforward check shows that all the data defining $F$ uniquely extends. 
\end{proof}

\subsection{Structure of the spin functor}\label{sec:spin-functor-structure}
We now use the results of \S \ref{sec:simpl-struc} and \S \ref{subsec:enriched-functors} to show that the functor $\bS$ is essentially constant. We use the notation, $\underline{Z}_x$ to denote the constant functor with value $x$. 
\begin{theorem}\label{thm:spin-functor-structure}
	
	\begin{enumerate}[label = \roman*)]
		\item For any $E^{\bullet}\in \mathcal{R}$, there exists a natural isomorphism of functors $\bS \Rightarrow \underline{Z}_{\bS(E^{\bullet})}$ whose value at $E^{\bullet}\in \mathcal{R}$ is the identity. \label{item:spin-structurei}
		\item There is an equivalence $\Phi:\mathcal{S}\to \bS(E^{\bullet})\times \mathcal{R}$ of categories over $\mathcal{R}$, whose fiber over $E^{\bullet}$ is the identity.\label{item:spin-structureii}
		\item  Any choice of $\Phi$ as constructed in the proof of \ref{item:spin-structureii} is essentially canonical: If $\Phi'$ is obtained from making different choices in the construction, there exists a \emph{canonical} natural transformation $\Phi \Rightarrow \Phi'$. Given three choices, then the resulting triangle of natural isomorphisms commutes. \label{item:spin-structureiii}
	\end{enumerate}
\end{theorem}

\begin{corollary}\label{cor:restriction-iso}
	For any choice of representative $E^{\bullet}$, the natural restriction map $\bS(\bE)\to \bS(E^{\bullet})$ that sends a section of $\mathcal{S}\to \mathcal{R}$ to its value at $E^{\bullet}$ is an equivalence of categories. 
\end{corollary}
\begin{proof}
	Let $\Phi:\mathcal{S}\to \bS(E^{\bullet})\times \mathcal{R}$ be an equivalence as in \ref{item:spin-structureii} of Theorem \ref{thm:spin-functor-structure}. Then the restriction map $\bS(\bE)\to \bS(E^{\bullet})$ factors through postcomposition with $\Phi$. Since $\Phi$ is an equivalence, post-composition with $\Phi$ is an equivalence of the categories of sections. On the other hand, the category of sections of $\bS(E^{\bullet})\times \mathcal{R}$ is naturally isomorphic to the category of functors $\mathcal{R}\to  \bS(E^{\bullet})$, so the result follows from Proposition \ref{prop:simply-connected}. 
\end{proof}
\begin{proof}[Proof of Theorem \ref{thm:spin-functor-structure}]
	Statements \ref{item:spin-structurei} and \ref{item:spin-structureii} are equivalent by the Grothendieck construction. We will prove \ref{item:spin-structurei}, and similarly prove the analogous statement about natural transformations of functors in place of \ref{item:spin-structureiii}. 
	
	Putting together the constructions of \S \ref{subsec:enriched-functors}, we have a commutative diagram of simplicial sets
	
	\begin{equation*}
		\begin{tikzcd}
			\operatorname{N}(\mathcal{R})\ar[d]\ar[r]& \Nhc(\mathcal{R}^{\Delta})\ar[r]\ar[d]& \ND(h_2\mathcal{R})\ar[d] &\\
			\operatorname{N}(\mathcal{P})\ar[r]&\Nhc(\mathcal{P}^{\Delta})\ar[r] & \ND(h_2\mathcal{P})\ar[r,"{\ND(\bS)}"]&\ND(\Grpd^{\simeq})
		\end{tikzcd}
	\end{equation*}
	Here, the vertical maps are obtained from Construction \ref{constr:to-pair-functor} and Proposition \ref{prop:pairfunc-simplicial}, and the commutative squares are due to the compatibility of the respective nerve constructions. The map $\ND(\bS)$ is obtained from the functor $\bS$ of Construction \ref{constr:pullback-simplicial-pairs} and the composite $\operatorname{N}(\mathcal{R})\to \ND(\Grpd^{\simeq})$ agrees with the one induced by the spin functor $\bS:\mathcal{R}\to \Grpd^{\simeq}$ of Definition \ref{def:spin-functor}.

	Since the $2$-category $\Grpd^{\simeq}$ has only invertible morphisms, the Duskin nerve $\ND(\Grpd^{\simeq})$ is a Kan complex, see \cite[\href{https://kerodon.net/tag/019D}{Proposition 019D}, (d)$\implies$(b)]{kerodon}. 
	Consequently, any morphism from a simplicial set $K\to \ND(\Grpd^{\simeq})$ factors through a fibrant replacement $K\to \Ex^{\infty}(K)$, and the space of such extensions is contractible.
	
	In particular, we can choose an extension of $\Nhc(\mathcal{R}^{\Delta})\to \ND(\Grpd^{\simeq})$ to a map 
	\[ e: \Ex^{\infty}\left(\Nhc(\mathcal{R}^{\Delta})\right)\to \ND(\Grpd^{\simeq}).\]

	By Theorem \ref{thm:contractible}, the Kan complex $\Ex^{\infty}\left(\Nhc(\mathcal{R})^{\Delta}\right)$ is contractible. Thus, for any given $E^{\bullet}\in \mathcal{R}$, we can choose a map
	\[c:\Ex^{\infty}\left(\Nhc(\mathcal{R}^{\Delta})\right) \times \Delta^1\to \Ex^{\infty}\left(\Nhc(\mathcal{R}^{\Delta})\right) \]
	so that 
	\begin{itemize}
		\item $c|_{- \times \{0\}}$ is the identity on $\Ex^{\infty}\left(\Nhc(\mathcal{R}^{\Delta})\right)$,
		\item  $c|_{-\times \{1\}}$ is constant with value $E^{\bullet}$ and 
		\item $c|_{\{E^{\bullet}\}\times \Delta^1}$ is constant with value $E^{\bullet}$.
	\end{itemize}
	More precisely, we may define the restriction of $c$ to the sub-set 
	\[\Ex^{\infty}\left(\Nhc(\mathcal{R}^{\Delta})\right)\times \partial\Delta^1 \cup \{E^{\bullet}\}\times \Delta^1 \]
	via these prescriptions, and then obtain $c$ via \cite[\href{https://kerodon.net/tag/0070}{Corollary 0070}]{kerodon} (this also shows that the space of choices of $c$ is contractible). 
	Now, consider the composition 
	\[N(\mathcal{R})\times \Delta^1 \hookrightarrow \Ex^{\infty}\left(\Nhc(\mathcal{R}^{\Delta})\right)\times \Delta^1 \xrightarrow{c} \Ex^{\infty}\left(\Nhc(\mathcal{R}^{\Delta})\right) \xrightarrow{b} \ND(\Grpd^{\simeq}).\]
	By construction, it restricts to $\bS$ on $\operatorname{N}(\mathcal{R})\times \{0\}$, and is constant with value $E^{\bullet}$ on $\operatorname{N}(\mathcal{R})\times \{0\}$ and on $\{E^{\bullet}\}\times \Delta^1$. Hence it defines a natural transformation $\bS\Rightarrow \underline{Z}_{\bS(E^{\bullet})}$ as in \ref{item:spin-structurei}.
	Regarding \ref{item:spin-structureiii}, we observe that the choice involved in the construction is exactly the choice of the pair $(b,c)$, each corresponding to a vertex in a contractible Kan complex. Hence, the space of such choices is itself a contractible Kan complex, and \ref{item:spin-structureiii} follows from this.
\end{proof}

\subsection{Compatibility with pullbacks of schemes}
We discuss how the constructions of this section behave under pullback of schemes. All schemes in this subsection are assumed to be affine\footnote{Or quasi-projective over an affine scheme.}. As throughout, let $\bE$ be an oriented quadratic complex on $X$. Since we want to change the base-scheme, we use the explicit notation $\mathcal{R}_{\bE}$ and $\mathcal{S}_{\bE}$ to denote the categories of Definition \ref{def:iso-red} and Definition \ref{def:spin-functor} respectively. 

Let $a:Y\to X$ be a morphism of schemes. Pullback along $a$ induces a functor $\mathcal{R}_{\bE}\to \mathcal{R}_{a^*\bE}$. If $\bE$ is oriented, this extends to a functor $\mathcal{S}_{\bE}\to \mathcal{S}_{a^*\bE}$, defined as follows: 
The spin structure $\sigma\in \bS(E^{\bullet})$ given by a double cover $P_{\sigma}\to P_{E^0}\times_X P_{\det E^{\bullet}_+}$ is sent to the composition 
\[a^*P_{\sigma}\to a^*P_{E^0}\times_Y a^*P_{\det E^{\bullet}_+}\simeq P_{a^*E_0}\times_Y P_{\det a^*E^{\bullet}_+},\]
where the latter isomorphism is the compatility of pullback with the principal bundle construction. 

The pullback and forgetful maps fit into a (strictly!) commutative diagram 
\begin{equation*}
	\begin{tikzcd}
		\mathcal{S}_{\bE}\ar[r,"a^*"]\ar[d]& \mathcal{S}_{a^*\bE}\ar[d] \\
		\mathcal{R}_{\bE}\ar[r,"a^*"] & \mathcal{R}_{a^*\bE}.
	\end{tikzcd}
\end{equation*}

Let now $\tau\in \bS(a^*\bE)$ and $\sigma\in \bS(\bE)$ be sections of the respective forgetful maps. Composition with the pullback maps gives maps $\tau\circ a^*$ and $a^*\circ \sigma$ from $\mathcal{R}_{\bE}\to \mathcal{S}_{a^*\bE}$, both of which make the following diagram commute
\begin{equation*}
	\begin{tikzcd}
		& \mathcal{S}_{a^*\bE}\ar[d] \\
		\mathcal{R}_{\bE}\ar[r,"a^*"] \ar[ur,dashed]& \mathcal{R}_{a^*\bE}
	\end{tikzcd}
\end{equation*}
\begin{definition}\label{def:spin-category-pullback}
	Let $\mathcal{S}'_{\bE,a}:=\mathcal{R}_{\bE}\times_{\mathcal{R}_{a^*\bE}} \mathcal{S}_{a^*\bE}$ denote the category fibered in groupoids over $\mathcal{R}_{\bE}$ obtained by pullback of $\mathcal{S}_{a^*\bE}$ along $a^*:\mathcal{R}_{\bE}\to \mathcal{R}_{a^*\bE}$. 
	Equivalently, this is the fibered category associated to the functor 
	\[\mathcal{R}_{\bE}\xrightarrow{a^*}\mathcal{R}_{a^*\bE}\xrightarrow{
		\bS}\Grpd^{\simeq}.\]
	We also let $\bS'(\bE,a)$ denote the category of sections of $\mathcal{S}'_{\bE,a}\to \mathcal{R}_{\bE}$. 
\end{definition}
We may and do also interpret $\tau\circ a^*$ and $a^*\circ \tau$ as elements of $\bS'(\bE,a)$, so that we have functors
\begin{equation}\label{eq:pullback-functors}
	\bS(a^*\bE)\xrightarrow{-\circ a^*} \bS'(\bE,a)\xleftarrow{a^*\circ -} \bS(\bE).
\end{equation}

\begin{definition}\label{def:pullback-comparison}
	A \emph{comparison map from $\tau$ to $\sigma$ over $a$} is an isomorphism  $\tau\circ a^*\Rightarrow a^*\circ \tau $ of functors of categories over $\mathcal{R}_{\bE}$, i.e. an isomorphism $\tau\circ a^*\to a^*\circ \sigma$ in $\bS'(\bE,a)$.
\end{definition}
\begin{remark}
	To give such a comparison map, one needs to specify for each $E^{\bullet}\in \mathcal{R}_{\bE}$ an isomorphism $\eta_{E^{\bullet}}: \tau_{a^*E^{\bullet}} \to a^*\sigma_{E^{\bullet}}$ in $\bS(a^*E^{\bullet})$. These need to satisfy that for any morphism $\xi:F^{\bullet}\rightsquigarrow E^{\bullet}$ in $\mathcal{R}_{\bE}$, the following diagram in $\bS(a^*F^{\bullet})$ commutes:
	\begin{equation*}
		\begin{tikzcd}[column sep = large]
			\tau_{a^*F^{\bullet}}\ar[r,"{\eta_{F^{\bullet}}}"] \ar[d,"{\tau_{a^*\xi}}"]& a^*\sigma_{ F^{\bullet}} \ar[d,"{a^*\sigma_{\xi}}"]\\	
			(a^*\xi)^*\tau_{a^*E^{\bullet}}\ar[r,"{(a^*\xi)^*\eta_{E^{\bullet}}}"]& a^*\xi^*\sigma_{E^{\bullet} } 
		\end{tikzcd}
	\end{equation*}
\end{remark}

Let $E^{\bullet}\in \mathcal{R}_{\bE}$ be a given representative. Then restriction to the fiber over $E^{\bullet}$ (respectively over $a^*\bE$) induces the vertical maps in the following commutative diagram extending \ref{eq:pullback-functors}
\begin{equation}\label{eq:pullback-functor-restrictions}
	\begin{tikzcd}
		\bS(a^*\bE)\ar[r]\ar[d]& \bS'(\bE,a)\ar[d]& \bS(\bE)\ar[l]\ar[d]\\
		\bS(a^*E^{\bullet})\ar[r,equal] &\bS(a^*E^{\bullet}) &\bS(E^{\bullet})\ar[l]
	\end{tikzcd}
\end{equation}

The following is a variant of Corollary \ref{cor:restriction-iso}. 
\begin{proposition}\label{prop:restriction-equiv-pullback}
	The map $\bS'(\bE,a)\to \bS(a^*E^{\bullet})$ in \eqref{eq:pullback-functor-restrictions} is an equivalence of categories. 
\end{proposition}

Before we prove this, we collect two consequences.
\begin{corollary}
	The map $\bS(a^*\bE)\to \bS'(\bE,a)$ is an equivalence of categories. 
\end{corollary}
\begin{proof}
	By Corollary \ref{cor:restriction-iso}, the outer vertical maps in \eqref{eq:pullback-functor-restrictions} are equivalences. By Proposition \ref{prop:restriction-equiv-pullback}, so is the middle one. The result follows by the two-out-of-three property for equivalences. 
\end{proof}
Let $\tau\in \bS(a^*\bE)$ and $\sigma\in \bS(\bE)$. Restriction to $E^{\bullet}$ induces a map of sets
\begin{equation}\label{eq:restrict-comparison-map}
	\left\{\mbox{{Comparison maps from $\tau$ to $\sigma$ over $a$}} \right\}\to \operatorname{Isom}_{\bS(a^*E^{\bullet})}(\tau_{a^*E^{\bullet}}, a^*\sigma_{E^{\bullet}}).
\end{equation} 
\begin{corollary}\label{cor:comparison-bijection}
	The map \eqref{eq:restrict-comparison-map} is a bijection. 
\end{corollary}
\begin{proof}
	By definition, the map \eqref{eq:restrict-comparison-map} is given by the application of the functor $\bS(\bE,a)\to \bS(a^*E^{\bullet})$ on the morphism space between two objects. By Proposition \ref{prop:restriction-equiv-pullback}, this functor is an equivalence, hence induces a bijection on morphism sets. 
\end{proof}

\begin{proof}[Proof of Proposition \ref{prop:restriction-equiv-pullback}]
	The proof follows the same lines as the one of Corollary \ref{cor:restriction-iso} using an analogue of Theorem \ref{thm:spin-functor-structure} \ref{item:spin-structurei}. We will only address the parts that are different.  
	
	As remarked in Definition \ref{def:spin-category-pullback}, the fibered category $\mathcal{S}'_{\bE,a}$ is associated to the composition 
	\[\mathcal{R}_{\bE}\to \mathcal{R}_{a^*\bE}\xrightarrow{\bS} \Grpd^{\simeq}.\]
	As in the proof of Theorem \ref{thm:spin-functor-structure}, corresponding to the spin functor $\mathcal{R}_{a^*\bE}\to \Grpd^{\simeq}$ associated to the oriented quadratic complex $a^*\bE$ on $Y$, we have a map of simplicial sets $\operatorname{N}(\mathcal{R}_{a^*\bE})\to \ND(\Grpd^{\simeq})$, and a factorization  
	\[\operatorname{N}(\mathcal{R}_{a^*\bE})\to \Nhc(\mathcal{R}^{\Delta}_{a^*\bE})\to \ND(\Grpd^{\simeq}).\]
	Pullback by $a$ extends naturally to a map of simplicial categories $\mathcal{R}^{\Delta}_{\bE}\to \mathcal{R}^{\Delta}_{a^*\bE}$. Hence, we have a commutative diagram of simplicial sets 
	\begin{equation*}
		\begin{tikzcd}
			\operatorname{N}(\mathcal{R}_{\bE})\ar[d]\ar[r]& \Nhc(\mathcal{R}^{\Delta}_{\bE})\ar[d]&\\
			\operatorname{N}(\mathcal{R}_{a^*\bE})\ar[r]&\Nhc(\mathcal{R}_{a^*\bE}^{\Delta})\ar[r] &\ND(\Grpd^{\simeq})
		\end{tikzcd}
	\end{equation*}
	and the corresponding functor $\mathcal{R}_{\bE}\to \Grpd^{\simeq}$ is $\bS\circ a^*$. By the same argument as in the proof of Theorem \ref{thm:spin-functor-structure}, we conclude that there exists a natural transformation from $\bS \circ a^*$ to the constant functor on $\mathcal{R}_{\bE}$ whose value at $E^{\bullet}$ is the identity on $\bS(a^*E^{\bullet})$.
	From this, one concludes Proposition \ref{prop:restriction-equiv-pullback} using the same argument as in the proof of Corollary \ref{cor:restriction-iso}.
\end{proof}
\begin{variant}
	Let $a:Y\to X$ and $b:Z\to Y$ be morphisms of schemes, quasi-projective over affine schemes and let $\bE$ be an oriented quadratic complex on $X$. Then from the commutative diagram
	\begin{equation*}
		\begin{tikzcd}
			\mathcal{S}_{\bE}\ar[r,"a^*"]\ar[d]& \mathcal{S}_{a^*\bE}\ar[d]  \ar[r,"b^*"] & \mathcal{S}_{b^*a^*\bE}\ar[d]\\
			\mathcal{R}_{\bE}\ar[r,"a^*"] &  \mathcal{R}_{a^*\bE}\ar[r,"b^*"]&\mathcal{R}_{b^*a^*\bE}.
		\end{tikzcd}
	\end{equation*}
	and the identification $b^*a^*\bE\simeq (a\circ b)^*\bE$, one has the commutative diagram
	\begin{equation*}
		\begin{tikzcd}
			\mathcal{S}_{\bE,a\circ b}'\ar[r]\ar[d]& \mathcal{S}'_{a^*\bE, b}\ar[d] \\
			\mathcal{R}_{\bE}\ar[r] & \mathcal{R}_{a^*\bE}
		\end{tikzcd}
	\end{equation*}
	Moreover, this is already a base change diagram, hence we have an induced functor
	\[\bS'(a^*\bE,b)\to \bS'(\bE,a\circ b).\]
	This is an equivalence, which one sees by checking that the following diagram $2$-commutes 
	\begin{equation*}
		\begin{tikzcd}
			\bS(b^*a^*\bE)\ar[r,"\sim"]\ar[d,"\sim"]& \bS((a\circ b)^*\bE)\ar[d,"\sim"] \\
			\bS'(a^*\bE,b)\ar[r] & \bS'(\bE,a\circ b).
		\end{tikzcd}
	\end{equation*}
\end{variant}

\section{Spin Structures on Quadratic Complexes}\label{sec:spin-structures}
We generalize the notion of spin structures beyond affine schemes to algebraic stacks. Let $\mathcal{X}$ be an algebraic stack over $\bZ[1/2]$ and let $\bE$ be an oriented quadratic complex on $\mathcal{X}$. Let $\Aff_{\mathcal{X}}$ denote the category of affine schemes over $\mathcal{X}$.

\subsection{Spin structures on complexes over stacks}

For any object $x:X\to \mathcal{X}$ of $\Aff_{\mathcal{X}}$, we have the categories $\mathcal{S}_{x^*\bE}\to \mathcal{R}_{x^*\bE}$ associated to $x^*\bE$. The collections of these categories can be assembled into fibered categories over $\Aff_{\mathcal{X}}$:
\begin{definition}\label{def:R-cat-stack}
	We let $\underline{\mathcal{R}}_{\bE}$ denote the category over $\Aff_{\mathcal{X}}$ whose objects are pairs $(x, E^{\bullet})$ with $x:X\to \mathcal{X}\in \Aff_{\mathcal{X}}$ and $E^{\bullet}\in \mathcal{R}_{x^*\bE}$. A morphism $(y:Y\to \mathcal{X}, F^{\bullet})\to (x:X\to \mathcal{X},E^{\bullet})$  in $\mathcal{R}_{\mathcal{X}}$ consists of a morphism $a:Y\to X$ of affine schemes over $\mathcal{X}$ together with a morphism  $\xi:F^{\bullet} \rightsquigarrow a^*E^{\bullet}$ in $\mathcal{R}_{y^*\bE}$, where $a^*E^{\bullet}$ is seen as a representative of $y^*\bE$ via the natural isomorphism $a^*x^*\bE\to y^*\bE$ induced by $a$.   
\end{definition}

\begin{definition}\label{def:S-cat-stack}
	We let $\underline{\mathcal{S}}_{\bE}$ denote the category over $\underline{\mathcal{R}}_{\bE}$ whose objects are triples $(x, E^{\bullet}, \sigma)$, where $x:X\to \mathcal{X}\in \Aff$, $E^{\bullet}\in \mathcal{R}_{x^*\bE}$ and $\sigma \in \bS(E^{\bullet})$. 
	A morphism $(y, F^{\bullet}, \tau)\to (x, E^{\bullet}, \sigma)$ in $\mathcal{S}_{\mathcal{X}}$ consists of a morphism  $(a, \xi):(y,F^{\bullet})\to (x,E^{\bullet})$ in $\underline{\mathcal{R}}_{\bE}$ together with an isomorphism of spin structures $\tau \to \xi^*a^*\sigma $.  
\end{definition}

As $\underline{\mathcal{R}}_{\bE}$ and $\underline{\mathcal{S}}_{\bE}$ are obtained from the Grothendieck construction, one has:
\begin{lemma}\label{lem:fibered-cats-stacks}
	\begin{enumerate}[label = \roman*)]
		\item The category $\underline{\mathcal{R}}_{\bE}$ is fibered over $\Aff_{ \mathcal{X}}$. 
		\item The category $\underline{\mathcal{S}}_{\bE}$ is fibered in groupoids over $\underline{\mathcal{R}}_{\bE}$.
	\end{enumerate}
\end{lemma}

\begin{definition}\label{def:spin-structure}
	A \emph{spin structure} on $\bE$ is a section $\underline{\mathcal{R}}_{\bE}\to \underline{\mathcal{S}}_{\bE}$ of the forgetful functor $\underline{\mathcal{S}}_{\bE}\to \underline{\mathcal{R}}_{\bE}$. A morphism of spin structures is a natural transformation of functors over $\underline{\mathcal{R}}_{\bE}$. We denote by $\bS(\bE)$ the category of spin structures on $\bE$. 
\end{definition}
\begin{remark}
	In the case that $\mathcal{X}$ is the functor of points of an affine scheme $X$, we have now defined $\bS(\bE)$ \emph{twice}, although both notions turn out to be equivalent (Theorem \ref{thm:spin-str-defs-compatible}). To avoid confusion, we will for the moment denote the groupoid of Definition \ref{def:spin-structure-affine} by $\bS^{\Aff}(\bE)$.   
\end{remark} 

Effectively, a spin structure on the oriented quadratic complex $\bE$ on the stack $\mathcal{X}$ is a choice of spin structure (in the sense of Definition \ref{def:spin-structure-affine}) on $x^*\bE$ for each $(X,x)\in \Aff_{\mathcal{X}}$, together with suitable compatibility data under pullbacks of affine schemes. To make this more precise, we introduce the following
\begin{definition}
	Let $\underline{\bS}^{\Aff}(\bE)$ denote the category over $\Aff_{\mathcal{X}}$ defined as follows 
	\begin{itemize}
		\item Objects of $\underline{\bS}^{\Aff}(\bE)$ over $x:X\to \mathcal{X}$ are spin structures on $X$, i.e. objects of $\bS^{\Aff}(x^*\bE)$, or more concretely, sections of $\mathcal{S}_{x^*\bE}\to \mathcal{R}_{x^*\bE}$.  
		\item A morphism $(y:Y\to \mathcal{X},\tau) \to (x:X\to \mathcal{X}, \sigma)$ in  $\underline{\bS}^{\Aff}(\bE)$ consists of a morphism $a:Y\to X$ in $\Aff_{\mathcal{X}}$ together with a comparison map from $\tau$ to $\sigma$ over $a$ (Definition \ref{def:pullback-comparison}).
	\end{itemize} 
\end{definition}
\begin{proposition}\label{prop:spin-structure-stack-alternate}
	\begin{enumerate}[label = \roman*)]
		\item The category $\underline{\bS}^{\Aff}(\bE)$ is fibered in groupoids over $\Aff_{\mathcal{X}}$.\label{item:spinstructure-alternatei}
		\item The category of sections of $\underline{\bS}^{\Aff}(\bE)\to \underline{\mathcal{R}}_{\bE}$ is \emph{isomorphic} to the category $\bS(\bE)$.\label{item:spinstructure-alternateii}
	\end{enumerate}
\end{proposition}
\begin{proof}
	Clearly, the fibers of $\underline{\bS}^{\Aff}(\bE)\to \Aff_{\mathcal{X}}$ are groupoids, so for \ref{item:spinstructure-alternatei} it remains to show that it is a fibered category. Let $a:(y:Y\to \mathcal{X})\to(x:X\to\mathcal{X})$ be a given morphism in $\Aff_{\mathcal{X}}$.

	The existence of cartesian arrows over $a$ may be checked after pulling back to the category $\Aff_{X}$. Without loss of generality, we may therefore assume that $\mathcal{X} = X$ is an affine scheme.
	In this case, we may choose a representative $E^{\bullet}\in \mathcal{R}_{\bE}$, which induces a representative $E_y^{\bullet}\in \mathcal{R}_{y^*\bE}$ for every $y:Y\to X$. Consider the fibered category $\underline{\bS}(E^{\bullet})\to \Aff_{X}$ associated to the pseudo-functor $(y:Y\to X)\mapsto \bS(y^*E^{\bullet})$. Then, restriction induces a map $\underline{\bS}^{\Aff}(\bE)\to \underline{\bS}(E^{\bullet})$ of categories over $\underline{\mathcal{R}}_{\bE}$. By Corollary \ref{cor:restriction-iso}, this map is an equivalence fiberwise, hence an equivalence. By \cite[\href{https://stacks.math.columbia.edu/tag/042G}{Lemma 042G}]{stacks-project}, the property of being fibered is preserved under equivalences preserving the base, hence \ref{item:spinstructure-alternatei} follows.
	
	Point \ref{item:spinstructure-alternateii} can be seen by unraveling both definitions and using the properties of fibered categories. One finds that to give an object in either category means to give the following data:
	\begin{itemize}
		\item for each $x:X\to \mathcal{X}$ in $\Aff_{\mathcal{X}}$ and each $E^{\bullet}\in \mathcal{R}_{x^*\bE}$, a choice of $\sigma_{E^{\bullet}}\in\bS(E^{\bullet})$
		\item for each $x:X\to \mathcal{X}$ in $\Aff_{\mathcal{X}}$ and each morphism $\xi:F^{\bullet}\rightsquigarrow E^{\bullet}$ in $\mathcal{R}_{x^*bE}$, an isomorphism $\sigma_{\xi}:\sigma_{F^{\bullet}}\to\xi^*\sigma_{E^{\bullet}}$,
		\item for each morphism $a:(Y,y) \to (X,x)$ in $\Aff_{\mathcal{X}}$, and each $E^{\bullet}$ in $\mathcal{R}_{x^*\bE}$, an isomorphism $\eta_{a,\bE}:\sigma_{a^*E^{\bullet}}\to a^*\sigma_{E^{\bullet}}$.
	\end{itemize}
	The result is immediate from this.
\end{proof}

Now suppose that $\mathcal{X}=X$ is the functor of points of an affine scheme\footnote{Or one quasi-projective over one.}. We examine how $\bS(\bE)$ and $\bS^{\Aff}(\bE)$ are related. 
\begin{definition}
	Let $\underline{\bS}'(\bE)$ denote the fibered category over $\Aff_{X}$ associated to the functor sending $y:Y\to X$ to the groupoid $\bS'(\bE,y)$ of Definition \ref{def:spin-category-pullback}, and sending a morphism $a:(Z,z)\to (Y,y)$ over $X$ to the map $\bS'(\bE,y)\to \bS'(\bE,z)$ induced by pullback along $a$.
\end{definition}

We claim that we have natural maps of categories fibered in groupoids over $\Aff_{X}$
\begin{equation}\label{eq:category-diagram}
	\begin{tikzcd}
		\underline{\bS}^{\Aff}(\bE)\ar[r]&\underline{\bS}'(\bE)& 	\Aff_X\times \bS^{\Aff}(\bE).\ar[l]
	\end{tikzcd}
\end{equation} 
The map $\Aff_{X}\times \bS^{\Aff}(\bE)\to \underline{\bS}'(\bE)$ sends $(y,\sigma)$, where $y:Y\to X$ and $\sigma\in \bS^{\Aff}(\bE)$ to $(y,y^*\sigma)$, where $y^*\sigma\in \bS'(\bE,y)$. 
The map $\underline{\bS}^{\Aff}(\bE)\to \underline{\bS}'(\bE)$ is on fibers over any $(Y,y)\in \Aff_X$ given by the map  $\bS^{\Aff}(y^*\bE)\to \bS'(\bE,y)$ of \eqref{eq:pullback-functors} induced by pre-composition with $y$. To define it on morphisms, let $(a,\eta):(z,\tau)\to (y,\sigma)$ be a morphism in $\underline{\bS}^{\Aff}(\bE)$, i.e. $a:(Z,z)\to (Y,y)$ is a morphism in $\Aff_X$, we have $\tau\in \bS^{\Aff}(z^*\bE)$, $y\in \bS^{\Aff}(y^*\bE)$, and $\eta$ is a comparison map from $\tau$ to $\sigma$ over $a$. By definition, $\eta$ specifies an isomorphism $\tau\circ a^*\to a^*\circ \sigma$ in $\bS'(y^*\bE,a)$. Let $\overline{\eta}$ denote its image under the functor 
\[\bS'(y^*\bE,a)\to \bS'(\bE,a\circ y) = \bS'(\bE,z).\]
Then we take $(a,\overline{\eta})$ to be the image of $(a,\eta)$ in $\underline{\bS}'(\bE)$.

Let $\bS'(\bE)$ denote the category of sections of $\underline{\bS}'(\bE)\to \Aff_X$. Passing to the categories of sections of \eqref{eq:category-diagram}, we obtain the following functors
\begin{equation} \label{eq:compare-spin-definitions}
\begin{tikzcd}
	\bS(\bE) \ar[r]& \bS'(\bE)\ar[d,dotted]& \operatorname{Fun}(\Aff_X,\bS^{\Aff}(\bE))\ar[l]&\bS^{\Aff}(\bE)\ar[l,dashed]\\
	&\bS^{\Aff} & &
\end{tikzcd}
\end{equation}
where the added dashed map is the inclusion of the subcategory of constant functors. The dotted map is only defined if $X$ is affine and is evaluation at $(id_X)\in \Aff_X$.
\begin{theorem}\label{thm:spin-str-defs-compatible}
	Suppose $\mathcal{X} = X$ is a scheme, quasi-projective over an affine scheme.  Then the maps in \eqref{eq:compare-spin-definitions} are equivalences of categories. 
\end{theorem}

\begin{proof}
	
	For $\bS(\bE)\to \bS'(\bE)$: We claim that in fact the map $\underline{\bS}(\bE)\to \underline{\bS}'(\bE)$ is an equivalence of categories over $\Aff_X$. This follows from combining the facts that for each $y:Y\to X$, the maps $\bS(y^*\bE)\to \bS'(\bE,y)$ is an equivalence, and that for each map $a:(Z,z)\to (Y,y)$ in $\Aff_X$, the map $\bS'(y^*\bE,a)\to \bS'(\bE,y\circ a) = \bS'(\bE,z)$ is an equivalence.
	
	Now, choose a representative $E^{\bullet}\in \mathcal{R}_{\bE}$. Then, restriction to $E^{\bullet}$ induces a diagram of maps extending \eqref{eq:category-diagram}
	\begin{equation}\label{eq:category-diagram2}
		\begin{tikzcd}
			\underline{\bS}^{\Aff}(\bE)\ar[r,"\sim"]\ar[d,"\sim"]&\underline{\bS}'(\bE)\ar[d]& 	\Aff_X\times \bS^{\Aff}(\bE)\ar[l]\ar[d,"\sim"]\\
			\underline{\bS}(E^{\bullet})\ar[r,equal]& \underline{\bS}(E^{\bullet})&\Aff_X\times \bS(E^{\bullet})\ar[l]
		\end{tikzcd}
	\end{equation}
	where we have marked the known equivalences. Passing to sections over $\Aff_X$ reduces us to showing the following maps are equivalences:
	\[\operatorname{Sec}(\underline{\bS}(E^{\bullet})/\Aff_X) \leftarrow \operatorname{Fun}(\Aff_X,\bS(E^{\bullet}))\leftarrow \bS(E^{\bullet}).\]
	To see that the map $\bS(E^{\bullet})\to \operatorname{Fun}(\Aff_X,\bS^{\Aff}(E^{\bullet}))$ is an equivalence, one can use that $\Aff_X$ has an inital object (the empty scheme), hence its total localization is equivalent to a point.
	
	Finally, we argue that the composition $\bS(E^{\bullet})\to \operatorname{Sec}(\underline{\bS}(E^{\bullet})/\Aff_X)$ is an equivalence. After choosing a cover of $X$ by affine open subsets, this follows from Zariski descent. \footnote{The proof simplifies a bit if one assumes $X$ is affine rather than just quasi-projective.}
\end{proof}
We conclude by defining the pullback of spin structures along maps of stacks.
\begin{construction}\label{constr:pullback}
	For a map of stacks $f:\mathcal{Y}\to \mathcal{X}$, we have an associated functor $F:\Aff_{\mathcal{Y}}\to \Aff_{\mathcal{X}}$ given by postcomposition with $f$. We have a natural \emph{isomorphism} 
	of categories $\underline{\mathcal{R}}_{F^*\bE}\to F^*\underline{\mathcal{R}}_{\bE}$   induced by the natural isomorphisms $t^*F^*\bE\simeq (F\circ t)^* \bE$ for each $(t:T\to \mathcal{Y})\in\Aff_{ \mathcal{Y}}$. We have a similar isomorphism $\underline{\mathcal{S}}_{F^*\bE}\to F^*\underline{\mathcal{S}}_{\bE}$. We obtain a pullback functor $F^*:\bS(\bE)\to \bS(F^*\bE)$ which sends a section $\mathcal{R}_{\bE}\to \mathcal{S}_{\bE}$ to its $2$-fiber product along $F$.
\end{construction}

\subsection{The gerbe of spin structures}
\begin{definition}\label{def:spin-gerbe}
	Given an algebraic stack $\mathcal{X}$ over $\Spec \bZ[1/2]$ and an oriented quadratic complex $\bE$ on $\mathcal{X}$, let $\underline{\bS}(\bE)$ denote the category fibered over schemes, defined as follows:
	\begin{itemize}
		\item The objects of $\underline{\bS}(\bE)$ are schemes $t:T\to \mathcal{X}$ over $\mathcal{X}$, together with a spin structure $\sigma$ on $t^*\bE$ in the sense of Definition \ref{def:spin-structure} . 
		\item A morphism $(T', \sigma')\to (T,\sigma)$ consists of a morphism $b:T'\to T$ of schemes over $\mathcal{X}$ and an isomorphism $b^*\sigma\to \sigma'$, where $b^*\sigma$ is the pullback defined in Construction \ref{constr:pullback}, which we view as a as a spin structure on $t'^*\bE$ via the identification $b^*t^*\bE\to t'^*\bE$. 
	\end{itemize} 
\end{definition}

\begin{proposition}\label{prop:spin-gerbe}
	The category $\underline{\bS}(\bE)$ fibered in groupoids over $\operatorname{Schemes}$ is a stack.
\end{proposition}
\begin{proof}
	We will check that descent holds for finite families of affine schemes and then conclude by purely formal arguments. 
	Let $V$ be an affine scheme with an oriented quadratic complex $\bE$ and let $(u_i:U_i\to V)_{i=1}^n$ be a smooth cover by affine schemes.
	Assume that for each $i$, we are given an element $\sigma_i\in \bS(\bE_{U_i})$, together with isomorphisms $\varphi_{i,j}:u_j^*\sigma_i\to u_i^*\sigma_j$ in $\bS(u_j^*\bE_{U_i})\simeq \bS(u_i^*\bE_{U_j})$ for each $i,j$, that moreover satisfy the co-cycle condition on triple products. (Here, we denote the base change of $u_i$ and $u_j$ to $U_i\times_V U_j$ also by $u_i$ and $u_j$ respectively). 
	
	Let $E^{\bullet}\in \mathcal{R}_{\bE}$ be a representative for $\bE$. Then as in \eqref{eq:category-diagram2}, we have equivalences 
	\[\bS(\bE)\to \operatorname{Sec}(\underline{\bS}(E^{\bullet})/\Aff_V)\leftarrow \bS(E^{\bullet})\]
	which are moreover compatible with further pullbacks along any map $U\to V$ from an affine scheme. This reduces us to checking that descent is effective for the categories $\bS(E^{\bullet})$, which follows from descent for principal bundles. 
	
	Now let $s:S\to \mathcal{X}$ be an arbitrary scheme, and $t:T\to S$ be a smooth cover. We would like to show that the category $\bS(s^*\bE)$ is equivalent to the category of descent data for  $t$ via the natural map obtained from restriction. Let $\pi_1,\pi_2:T\times_S T\to T$ denote the projection maps on first and second factor respectively, and similarly let $\pi_{i,j}:T\times_S T\times_S T\to T\times_S T$ denote the projection on the $i$th and $j$th factor, where $1\leq i<j\leq 3$. Now consider a descent datum $(\sigma_T, \varphi)$, where $\sigma$ is a spin structure on $t^*\bE$ and $\varphi:\pi_1^*\sigma_T\to \pi_2^*\sigma_T$ is an isomorphism of spin structures on $(s\circ t)^*\bE$, so that $\pi_{2,3}^*\varphi \circ \pi_{1,2}^*\varphi = \pi_{1,3}^*\varphi$ under the natural identifications. 
	
	Let $U\to S$ be a map from an affine scheme, and let $t_U:T_U\to U$ be the base change of $t$ to $U$. Let $V\to T_U$ be a map from an affine scheme so that the composition $v:V\to U$ is a smooth cover -- more concretely, we may take $V$ to be a disjoint union of affine open subsets of $T_U$. Then the descent data for $t$ induces descent data for $v$.  
	This gives rise to an object $\sigma_U\in \bS(u^*s^*\bE)$ by our result for affine schemes. One can check that -- up to canonical isomorphism -- the object $\sigma_U$ is independent of the choice of $V$. In particular, we can pass to the limit over all choices of $V$ to obtain a canonical choice. This construction is functorial in $U$, hence we obtain a functor sending a descent datum to an object of $\bS(s*\bE)$. This is an inverse equivalence to the operation of taking the descent datum associated to $t$. 
	
\end{proof}

\begin{corollary}\label{cor:spin-gerbe}
	For any stack $\mathcal{X}$ with an oriented quadratic complex $\bE$, the category $\underline{\bS}(\bE)\to \mathcal{X}$ is a $\ZT$-gerbe over $\mathcal{X}$, in particular it is an algebraic stack. When $\bE$ is clear from the data, we also denote it by $\mathcal{X}^{\operatorname{sp}}$.
\end{corollary}
\begin{definition}\label{def:obstr-class}
	We let $o_{\operatorname{sp}}(\bE)\in H^2(\mathcal{X},\ZT)$ denote the class of the gerbe $\mathcal{X}^{\operatorname{sp}}$. It is exactly the obstruction class to the existence of a spin structure on $\bE$.  
\end{definition}
\begin{proof}
	Working locally, we may assume that $\mathcal{X}$ is an affine scheme, so that we can pick a representative $E^{\bullet}$ for $\bE$. Etale-locally, we may trivialize the $\SO(m)\times \mathbb{G}_m$ bundle associated to $E^{\bullet}$. In particular, $\bS(\bE) = \bS(E^{\bullet})$ is locally non-empty. Given any two objects $\sigma_1,\sigma_2\in \bS(E^{\bullet})$, we may locally trivialize the associated $\Spin^{\bC}$-bundles (compatibly with a chosen trivialization of the $\SO(m)\times \bC^*$-bundle associated to $E^{\bullet}$), which shows that any two objects are locally isomorphic.
	
	Finally, for any object $x\in \bS(\bE)$ we have $\Aut(x)\simeq \underline{\mathbb{Z}}_2$:  
	This can be checked locally, so we may assume we are working over an affine scheme with chosen representative $E^{\bullet}$. Then $\Aut(x)$ is given by the automorphisms of the double cover $P_{x}\to P_{E}\times P_{\det E_+^{\bullet}}$, which is canonically the sections of the constant sheaf $\ZT$ over $T$.  
\end{proof}
We outline a variant of the constructions here, which may be useful for applications.
\begin{remark}\label{rem:spinc-structures}
	 Let $\mathcal{N}$ be a fixed line bundle on $\mathcal{X}$. Replacing the functor $\bP$ of Construction \ref{constr:to-pair-functor} with its twisted version $\bP^{\mathcal{N}}$ of Remark \ref{rem:twisted-pair} throughout, one obtains a fibered category $\underline{\bS}(\bE, \mathcal{N})$ of \emph{$\Spin^{\bC}$-structures} on the pair $(\bE,\mathcal{N})$. This again forms a $\ZT$-gerbe over $\mathcal{X}$, whose class in $H^2(\mathcal{X},\ZT)$ is 
	 \[o_{\operatorname{sp}}(\bE,\mathcal{N}) = o_{\operatorname{sp}}(\bE) + o_{\mathcal{N}},\]
	 where $o_{\mathcal{N}}$ is the class obstructing the existence of a square root of $\mathcal{N}$. 
	 In particular, one may take $\mathcal{N}$ to be the (pullback of the) universal sheaf on $\mathcal{X}\times B\bG_m$. This gives rise to a $\bG_m$-gerbe over $\mathcal{X}$ which can be interpreted as the gerbe parametrizing $\Spin^{\bC}$-structures on $\bE$ without reference to a fixed line bundle.  
\end{remark}

\section{Application to DT4 theory}
We apply the concepts developed here to stacks carrying an oriented DT4 obstruction theory -- the main example being moduli stacks of sheaves on Calabi--Yau fourfolds. The main result is that a spin structure on the obstruction theory gives rise to a virtual structure sheaf. In the absence of a spin structure, one can work on the $\ZT$-gerbe of Definition \ref{def:spin-gerbe} where one has a \emph{universal} spin structure. Thus, the virtual structure sheaf always exists as a \emph{twisted} sheaf. In this section all stacks are assumed to be locally Noetherian and maps to be locally of finite type.
\subsection{Obstruction theories in DT4 theory}
The obstruction theories appearing in DT4 theory are characterized by their self-dual structure and a more subtle -- but crucial -- \emph{isotropy condition} \cite[Definition 1.10]{park_pull}. In this subsection, we define the notion of DT4 obstruction theory for Artin stacks. To state the isotropy condition in this setting, we work affine locally, which requires the notion of symmetrized pullback. Another way to express the isotropy condition for stacks -- using \emph{higher} stacks -- can be found in \cite[Appendix B]{bkp_counting} via the intrinsic normal cone for stacks of \cite{ap_cone}).  

We first recall the definition of obstruction theory on an Artin stack:

\begin{definition}[{cf. \cite[Definition 2.4.1]{mochi_don}}]\label{def:ob-thy}
	Let $\mathcal{X}$ be an algebraic stack over a base $S$. An \emph{obstruction theory on $\mathcal{X}$ over $S$} consists of an object $\bE\in D^b(\mathcal{X})$ together with a map $\ob:\bE \to L_{\mathcal{X}/S}$ satisfying
	\begin{enumerate}[label = \roman*)]
		\item[(OB)]  the induced map $h^i(\ob)$ on cohomology sheaves is an isomorphism in degrees $i\geq 0$, and surjective for $i=-1$. 
	\end{enumerate}
\end{definition}

In the usual applications of obstruction theories to constructing virtual cycles, such as in Gromov--Witten theory or sheaf counting on varieties in dimension $\leq 3$, one often considers \emph{perfect} obstruction theories, for which $\bE$ is assumed to have perfect amplitude in degrees $[-1,1]$. In this situation, it is well known that one can construct a virtual structure sheaf on $\mathcal{X}$ as originally suggested by Kontsevich \cite{kont_en} (see e.g. \cite[Remark 5.4]{bf_normal} for the case of schemes).

The obstruction theories appearing in DT4 theory are not perfect, but instead are required to satisfy the following symmetry condition:

\begin{definition}[{cf. \cite[Definition C.1]{park_pull}}]\label{def:shifted-symm}
	Let $\mathcal{X}$ be an algebraic stack. 
	A \emph{$-2$-shifted symmetric complex} on $\mathcal{X}$ is a pair $(\bE, \theta)$ where $\bE$ is a perfect complex on $\mathcal{X}$, and where $\theta:\bE^{\vee} \to \bE[-2]$ is an isomorphism satisfying $\theta^{\vee}[-2] = \theta$.
\end{definition}

Combining Definitions \ref{def:ob-thy} and \ref{def:shifted-symm}, we have
\begin{definition}
	Let $\mathcal{X}$ be an algebraic stack over a base $S$.
	A \emph{$-2$-shifted symmetric obstruction theory} on $\mathcal{X}$ is a triple $(\bE,\ob,\theta)$, such that $(\bE,\ob)$ is an obstruction theory for $\mathcal{X}$ over $S$ and such that $(\bE,\theta)$ is a $-2$-shifted symmetric complex. 
\end{definition}

We now describe how a $-2$-shifted symmetric obstruction theory on a stack $\mathcal{X}$ gives rise to $-2$-shifted symmetric obstruction theory on any smooth affine scheme over $\mathcal{X}$. This uses the notion of symmetrized pullback introduced in Park \cite[Theorem 0.1]{park_pull}. An analogue for $-1$-shifted symmetric obstruction theories has been used in \cite{klt_dtpt} and \cite{liu_ss_vw} in the context of wall-crossing for equivariant K-theoretic invariants in DT3 theory. 

\begin{lemma}\label{lem:symm-pull}
	Let $\mathcal{X}$ be an algebraic stack over $S$, and let $(\bE, \ob,\theta)$ be a $-2$-shifted symmetric obstruction theory on $\mathcal{X}$ over $S$. Let $t:T\to \mathcal{X}$ be a smooth map from an affine scheme. Then there exists a $-2$-shifted symmetric obstruction theory $(\bF_T, \ob_T,\theta_T)$ on $T$ over $S$ that is \emph{compatible} with $\bE$, in the following sense: There exists a diagram in which the rows are distinguished triangles and the vertical maps give morphisms of distinguished triangles
	\begin{equation}\label{eq:compatibility-diagram}
		\begin{tikzcd}
			\bD^{\vee}[2]\ar[d,"\beta^{\vee}"]\ar[r,"\alpha^{\vee}"] & \bF_{T} \ar[d,"\alpha"]\ar[r,"\delta"]& \Omega_t\ar[r]\ar[d,equals]& \, \\
			t^*\bE \ar[r, "\beta"]\ar[d,"t^*\ob"]&\bD\ar[r,"\gamma"]\ar[d,"\ob_T'"]& \Omega_t \ar[r]\ar[d,equals]&\, \\
			t^*L_{\mathcal{X}/S}\ar[r]& L_{T/S}\ar[r] & \Omega_t \ar[r]&\, 
		\end{tikzcd}
	\end{equation}
	such that $\ob_T = \ob'_T\circ \alpha$ and where $\alpha^{\vee}$ and $\beta^{\vee}$ are defined using the self-duality isomorphisms $\theta_T$ and $t^*\theta$ respectively. 
	Moreover, the triple $(\bF_T, \ob_T,\theta_T)$ together with choice of diagram \eqref{eq:compatibility-diagram} is unique up to (not necessarily unique) isomorphism.
\end{lemma}
\begin{definition}
	We call $(\bF_T,\ob_T,\theta_T)$ as in Lemma \ref{lem:symm-pull} the \emph{symmetrized pullback obstruction theory.}
\end{definition}
\begin{proof}
	We construct $\bF_T$, by building up the diagram \eqref{eq:compatibility-diagram} from the bottom up. The lowest row is simply the natural triangle of cotangent complexes associated to the morphisms $T\to \mathcal{X}\to S$.  In particular, we have a connecting homomorphism $\Omega_t[-1]\to t^*L_{\mathcal{X}/S}$. By \cite[\S 2.3.3]{klt_dtpt}, using that $T$ is affine and $\ob$ is an obstruction theory, there exists a unique lift $\Omega_t[-1]\to t^*\bE$. We take $\bD$ to be the cone of this map and choose some $\ob_T'$ giving a morphism of the triangles in the lower two rows of \eqref{eq:compatibility-diagram}. By \cite[Lemma 2.3.7]{klt_dtpt}, the choice of $\bD$ is unique. (Note that \cite{klt_dtpt} works with $-1$-shifted symmetric objects, but the argument goes through with obvious changes)
	
	Now, the map $\Omega_t[-1]\to t^*\bE$ induces a map $t^*\bE \simeq t^*\bE^{\vee}[2]\to \Omega_t^{\vee}[3]$, which is an \emph{isotropic subcomplex} of the $-2$-shifted symmetric complex $(t^*\bE,t^*\theta)$ in the sense of \cite[Definition 2.3]{park_pull} (this uses again that $T$ is affine). By \cite[Lemma C.2]{park_pull} this determines the upper two rows in \eqref{eq:compatibility-diagram} uniquely up to isomorphism (here one needs to dualize and shift Park's diagram \cite[(C1)]{park_pull}).
\end{proof}

\begin{definition}\label{def:DT4-ob-thy}
	Let $\mathcal{X}$ be an algebraic stack over a base $S$. A \emph{DT4 obstruction theory} is a $-2$-shifted symmetric obstruction theory $E^{\bullet}$, satisfying the following \emph{isotropy condition}:
	\begin{enumerate}
		\item[(ISO)]\label{item:ISO} For any smooth map $T\to \mathcal{X}$ from an affine scheme, the induced symmetric pullback obstruction theory satisfies the isotropy condition (cf. Definition 4.10 in \cite{park_pull}). 
	\end{enumerate}
\end{definition}

For a \emph{perfect} obstruction theory $\bE$ on a scheme $X$ and any \emph{two-term} resolution $E^{\bullet}=[E^{-1}\to E^0]$ of $\bE$, one obtains a sub-cone $\fC_{X,E}\subset E_1$ which is preserved under the action of $E_0$, and so that $[\fC_{X,E}/E_0]$  is canonically identified with the intrinsic normal cone $\fC_X$ of $X$ \cite[Theorem 4.5]{bf_normal}. Moreover, any morphism of $2$-term resolutions $F^{\bullet}\to E^{\bullet}$ induces (dually) a morphism between associated cones $\fC_{X,E}\to \fC_{X,F}$ compatible with the actions of $E_0$ and $F_0$, so that taking quotients, we have a $2$-commutative diagram
\begin{equation*}
	\begin{tikzcd}
		{\left[\fC_{X,E}/E_0\right]} \ar[dr]\ar[r]& {[\fC_{X,F} /F_0]}\ar[d] \\
		\, & \fC_X
	\end{tikzcd}
\end{equation*} 
We record here that a similar statement is true for more general obstruction theories, which streamlines some of the arguments that follow.
\begin{construction}\label{constr:sub-cone}
	Let $\mathcal{X}$ be an algebraic stack with an obstruction theory $\bE\to L_{\mathcal{X}}$ and let $t:T\to \mathcal{X}$ be a smooth map from an affine scheme. 
	Let $f:\Omega_t[-1]\to t^*\bE$ be the unique lift of the connecting homomorphism $\Omega_t[-1]\to t^*L_{\mathcal{X}}$.
	Then for any representative $A^{\bullet}$ of $\bE$, we naturally obtain a sub-cone $\fC_{\mathcal{X},t,A}\subset A_1$. 
	
	To obtain this, choose a map $a:\Omega_t\to A^1$ representing the lift $f$. Then taking 
	\[A'^0:= \Ker(\Omega_t\oplus A^0 \xrightarrow{\begin{bmatrix}
			f_1& d 
	\end{bmatrix}} A^1),\] the sub-complex 
	\[A'^{\bullet} = \left[\cdots \to A^{-1}\to A'^0\to 0 \right]\]
	of $A^{\bullet}$ is an obstruction theory for $T$. The truncation $\left[A^{-1}\to A'^0\right]$ is a resolution of a perfect obstruction theory, hence we have a natural sub-cone $\fC_{\mathcal{X},t,A}\subset A_1$ invariant under translation $A_0$. 
	Given another representative $B^{\bullet}$ and a map of representatives $g_{\bullet}:B^{\bullet}\to A^{\bullet}$. Choosing a lift $b:\Omega_t\to B^1$ of $a$, one checks that the morphism $A_1\to B_1$ restricts to a morphism of cones \[\fC_{\mathcal{X},t,A}\to \fC_{\mathcal{X},t,B}.\] This is equivariant over the map $A_0\to B_0$. Note that the construction is independent of the choice of lifts $a$, and $b$, and functorial in $T$.
\end{construction}

\begin{remark}\label{rem:sub-cone-iso}
	In the case of Construction \ref{constr:sub-cone}, if $\bE$ is a DT4 obstruction theory, amd one chooses a self-dual representative $E^{\bullet}$ of $t^*\bE$, condition \eqref{item:ISO} of Definition \ref{def:DT4-ob-thy} guarantees that the sub-cone $\fC_{\mathcal{X},t,E}\subset E_1\simeq E^{-1}$ is isotropic. 
	
	If, moreover, 
	\[F^{\bullet} \leftarrow A^{\bullet} \to E^{\bullet}\]
	is an isotropic reduction with $K = \Ker \left(E_{1}\to A_1\right)$, then the induced morphism 
	\[\fC_{\mathcal{X},t,A}\to \fC_{\mathcal{X},t,E}\]
	is an isomorphism, and we have an exact sequence of cones
	\[0\to K \to \fC_{\mathcal{X},t,A} \to  \fC_{\mathcal{X},t,F}\to 0.\] 
\end{remark}
\subsection{Spin Structures and Clifford Representations}
Let $X$ be a scheme and let $(E,q)$ be a quadratic bundle of rank $n$ over $X$. We recall some parts from the appendix that will be needed in what follows

Let $\Cl(E)$ denote the \emph{Clifford Algebra of $E$} (cf. \S \ref{app:cl-algs}).  It is a $\bZ_2$-graded $\mathcal{O}_X$-algebra, and locally free of rank $2^n$ as an $\mathcal{O}_X$-module. We will use the canonical embedding $\iota_E:E\to \Cl(E)$ to view $E$ as a sub $\mathcal{O}_X$-module of $\Cl(E)$. A choice of orientation on $E$ gives rise to a \emph{volume element} $\omega\in \Cl(E)$, which is a certain section of the Clifford algebra satisfying $\omega^2 = 1$. If $n$ is odd, then any volume element lies in the center of $\Cl(E)$. If $n$ is even, we have 
\begin{equation}\label{eq:anti-comm}
	 \omega s = - s \omega \mbox{ for } s\in \Gamma(X,E).
\end{equation}

Now suppose we have fixed an orientation on $E$.  For a line bundle $L$ on $X$, the datum of a $\Spin^{\bC}$-structure $\sigma$ on the pair $(E,L)$ can be characterized in terms of giving a certain module over $\Cl(E)$. More precisely (see Propositions \ref{prop:spin-data-even} and \ref{prop:spin-data-odd} in the appendix)
\begin{proposition}
	\begin{enumerate}
		\item \label{item:even}If $n=2m$ is even, then the category of $\Spin^{\bC}$-structures on $(E,L)$ is equivalent to the category of pairs $(\sS,\eta)$, where 
		\begin{itemize}
			\item $\sS$ is a locally free $\mathcal{O}_X$-module with a left action of $\Cl(E)$,
			\item  $\eta$ is an isomorphism of $R$-modules $\eta:\sS\otimes_{\Cl(E)} \sS \to L$.\footnote{The tensor product is taken with respect to the induced right-module structure induced by the canonical anti-involution on $\Cl(E)$ that acts as $-\id$ on $E$ (see Proposition \ref{prop:clifford-props}).}
		\end{itemize}
		\item \label{item:odd} If $n=2m+1$ is odd, then the category of $\Spin^{\bC}$-structures on $(E,L)$ is equivalent to the category of pairs $(\sS,\eta)$, where 
		\begin{itemize}
			\item $\sS$ is a locally free $\mathcal{O}_X$-module with a left action of $\Cl(E)$ and on which the volume element $\omega$ acts by the identity,
			\item  $\eta$ is an isomorphism of $R$-modules $\eta:\sS\otimes_{\Cl_0(E)} \sS \to L$. 
		\end{itemize}
	\end{enumerate}
	In either \ref{item:even} or \ref{item:odd}, an isomorphism of pairs $(\sS,\eta)\to (\sS',\eta')$ is an isomorphism $\Phi:\sS\to \sS'$ so that $\eta = \eta'\circ (\Phi\otimes \Phi)$. 
\end{proposition}

We note that in either case the module $\sS$ is necessarily locally free of rank $2^m$.  

To apply this to DT4 theory, we recall a construction of Polishchuk and Vaintrob (\cite[\S 2.2]{pv_algebraic}).
\begin{construction}\label{constr:matrix-fac}
	Let $(E,q)$ be an oriented quadratic bundle of even rank on $X$ with a section $s\in \Gamma(X,E)$, and let $\sS$ be a module for $\Cl(E)$. Let $\sS = \sS_+\oplus \sS_-$ be the composition in $\pm 1$-Eigenspaces for the action of the volume element $\omega$. By \eqref{eq:anti-comm}, multiplication with any section of $E$ maps $\sS_+$ into $\sS_-$ and vice-versa. We define a matrix factorization of the element $q(s)$ on $X$
	\[C^{\bullet}(E,s,\sS):= \left[ \cdots \to \sS_- \xrightarrow{s}  \sS_+\xrightarrow{s} \sS_-\to \cdots\right],\] 
	where $\sS_+$ is in degree zero, and all maps are given by multiplication with $s\in \Gamma(X,E)$.
	
	If $s$ happens to be isotropic, i.e. $q(s) = 0$, this is a two-periodic chain complex of $\mathcal{O}_X$-modules. 
\end{construction}
As remarked in \cite[\S 3.1]{pv_algebraic}, if the section $s$ is isotropic, then the cohomology sheaves of $C^{\bullet}(E,s,\sS)$ are supported on the vanishing locus $Z(s)$. 
\begin{definition}\label{def:matfac-cohomologies}
	Suppose that in Construction \ref{constr:matrix-fac}, the section $s$ is isotropic. We define for $i\in \ZT$,
	\[h^i(E,s,\sS):= h^i(C^{\bullet}(E,s,\sS)) \in \operatorname{Coh}(Z(s)).\]
\end{definition}

\begin{remark}
	This construction easily generalizes to give maps 
	\begin{align*}
		D_{Coh}^b(X)&\to \Coh(Z(s))\\
		F^{\bullet}&\mapsto h^i\left(\operatorname{Tot}^{\oplus} \left(C^{\bullet}(E,s,\sS)\otimes F^{\bullet}\right)\right).
	\end{align*}
\end{remark}

\subsection{Local construction of a virtual structure sheaf}
We now turn to applying Construction \ref{constr:matrix-fac} to DT4 theory. In order to apply the construction in the context of stacks, we start with a slightly more general framework.

Let $(E,L)$ be a pair of an oriented quadratic bundle and a line bundle on a scheme $X$. Assume that $E$ is of even rank $n=2m$, and let $\sigma$ be a $\Spin^{\bC}$-structure on the pair $(E,L)$. Let $\sS_{\sigma}$ be the associated $\Cl(E)$-module, with trivialization $\eta_{\sigma}:\sS_{\sigma}\otimes_{\Cl(E)}\sS_{\sigma} \to L$.  

For what follows, we identify $E$ with its total space. Let $\pi_E:E\to X$ denote the projection, and $0_E:X\to E$ the zero section. Let $s_E\in \Gamma(E,\pi_E^*E)$ denote the tautological section. We also regard the quadratic structure on $E$ as a function $q_E\in \Gamma(E,\mathcal{O}_E)$. Via pull-back, the pair $(\pi_E^*\sS,\pi_E^*\eta_{\sigma})$ gives a $\Spin^{\bC}$-structure on the pair $(\pi^*_E E,\pi_E^*L_{\sigma})$ on $E$. 

\begin{construction}\label{constr:matfac-cone}
	Let $\mathfrak{C}\subset E$ be sub-scheme of the total space of $E$ that is \emph{isotropic}, i.e. that satisfies $q|_{\fC} \equiv 0$.  
	Let $E_{\fC},L_{\fC},\sS_{\sigma,\fC}$ and $\eta_{\sigma,\fC}$ denote the respective pullbacks of $E,L,\sS_{\sigma}$ and $\eta_\sigma$ to $\fC$. Let also $s_{\fC}$ denote the restriction of the tautological section $s_E$ to $\fC$. By the assumption on $\fC$, the section $s_{\fC}$ is isotropic. Hence, we can apply Definition \ref{def:matfac-cohomologies}, to obtain 
	\[h^i(E_{\fC},s_{\fC}, \sS_{\sigma,\fC}) \in \Coh(Z(s_{\fC})).\]
	Note that $Z(s_{\fC}) = Z(s_E)\cap \fC$ is contained in the zero-section of $E$ which is identified with $X$. Hence, we may pushforward to obtain sheaves
	\[h^i(E,\fC, \sS_{\sigma}):= h^i(E_{\fC},s_{\fC}, \sS_{\sigma,\fC}) \in \Coh(X).\]  
\end{construction}

\paragraph{Functoriality.}
In the situation of Construction \ref{constr:matfac-cone}, suppose that we are given an isotropic sub-bundle $K\subset E$ of rank $k$ satisfying $\fC \subset K^{\perp}$, and such that $\fC$ is invariant under translation by $K$. In this situation, the sheaves $h^i(E,\fC, \sS_{\sigma})$ behave well under taking isotropic reductions along $K$. 

Let $F:=K^{\perp}/K$, and let $M:= L\otimes \det K$. By  \eqref{diag:spin-structure-transfer}, the $\Spin^{\bC}$-structure $\sigma$ on the pair $(E,L)$ induces a $\Spin^{\bC}$-structure $\tau$ on $(F,M)$ via extension of structure group along $\widetilde{\G(m,k)\times \bG_m}\to \Spin^{\bC}(m-2k)$. 

In terms of the Clifford modules (see \ref{app:iso-red} for details), it sends $\sS_{\sigma}$ to the $\Cl(K^{\perp}/K)$-module $\sS_{\tau}:=\Ann(K)$, which is an $\mathcal{O}_X$-submodule of $\sS_{\sigma}$. By Lemma \ref{lem:isored-orientations}, the decomposition of $\S_{\tau}$ into even and odd part is compatible with the one of $\sS_{\sigma}$ under the inclusion. 

Let $\pi_K:K^{\perp}\to F$ be the projection. By assumption, there is a unique sub-scheme $\overline{\fC}\subset F$, such that $\pi_K^{-1}\overline{\fC} = \fC$ as sub-schemes of $K^{\perp}$.

\begin{construction}\label{constr:matfac-iso}
	From the identification $\sS_{\tau} = \Ann(K)$, we have an inclusion of $\mathcal{O}_X$-modules $\sS_{\tau}\hookrightarrow \sS_{\sigma}$. Moreover, we have a commutative diagram with vertical maps the multiplication actions
	\begin{equation*}
		\begin{tikzcd}
			\Cl(K^{\perp}/K)\otimes \sS_{\tau}\ar[d]&\Cl(K^{\perp})\otimes \sS_{\tau}\ar[l]\ar[r]\ar[d]&\Cl(E)\otimes \sS_{\sigma} \ar[d] \\
			\sS_{\tau}&\sS_{\tau}\ar[l, equal]\ar[r, hook] & \sS_{\sigma}
		\end{tikzcd}
	\end{equation*} 
	Since $\fC\subset K^{\perp}$, the tautological section $s_{\fC}\in \Gamma(\fC,\pi_E^*E)$ lies in $\pi_E^*K^{\perp}$. Let $\overline{s}_{\fC}$ denote its image in $\pi^*_E K^{\perp}/K$. Then multiplication by $\overline{s}_{\fC}$ on $\sS_{\tau}$ as a $\Cl(K^{\perp}/K)$-module agrees with the restriction of multiplication with $s$ on $\sS_{\sigma}$. Since the gradings are compatible, we have an induced map of complexes, which is a degreewise inclusion 
	\begin{equation}\label{eq:complex-inclusion}
		C^{\bullet}(F_{\fC}, \overline{s}_{\fC}, \sS_{\tau, \fC}) \to C^{\bullet}(E_{\fC}, s_{\fC}, \sS_{\sigma,\fC}).
	\end{equation}
	Here, the former complex is canonically identified with 	\[\pi_K^*C^{\bullet}(F_{\overline{\fC}}, s_{\overline{\fC}}, \sS_{\tau,\overline{\fC}}).\] 
	Since $\pi_K: \fC\to \overline{\fC}$ is flat, we get a map on cohomology sheaves 
	\[\pi_K^*h^i(F,\overline{\fC}, \sS_{\tau})\to h^i(E,\fC, \sS_{\sigma}).\]
	By adjunction, this is equivalent to a map of sheaves on $X$
	\begin{equation}\label{eq:sheaf-map}
		h^i(F,\overline{\fC}, \sS_{\tau})\xrightarrow{\nu_{E,\fC,\sS_{\sigma},K}} h^i(E,\fC, \sS_{\sigma}).
	\end{equation}
\end{construction}
\begin{lemma}\label{lem:matfac-functorial}
	Construction \ref{constr:matfac-iso} is functorial, i.e. given a further isotropic subspace $J\subseteq F$ such that $\overline{\fC}\subseteq J^{\perp}$ and $\overline{\fC}$ is $J$-invariant, and letting $\Lambda := \pi_K^{-1}J\subseteq E$, we have an identity
	\[\nu_{E,\fC,\sS_{\sigma},K} \circ \nu_{F,\overline{\fC}, \sS_{\tau}, J} = \nu_{E,\fC,\sS_{\sigma},\Lambda}.\]
\end{lemma}
\begin{proof}[Proof sketch]
	Let $G:=J^{\perp}/J$ and $\upsilon$ the induced $\Spin^{\bC}$-structure. We have that $\sS_{\upsilon}:=\Ann(J)\subseteq \sS_{\tau}$ agrees with $\Ann(\Lambda)$ as a sub $\mathcal{O}_X$-module of $\sS_{\sigma}$. Moreover, the induced structures as $\Cl(J^{\perp}/J)$-module and as $\Cl(\Lambda^{\perp}/\Lambda)$-module are compatible with the natural isomorphism  $J^{\perp}/J\simeq \Lambda^{\perp}/\Lambda$. Letting $\widehat{\fC}:= \overline{\fC}/J = \fC/\Lambda$ one next checks that the following induced diagram commutes
	\begin{equation*}
		\begin{tikzcd}
			\pi_K^*\pi_{J}^*h^i(G,\widehat{\fC}, \sS_{\upsilon})\ar[r]\ar[d,"\sim"]& \pi_K^*h^i(F,\overline{\fC}, \sS_{\tau})\ar[d] \\
			 \pi_{\Lambda}^*h^i(G,\widehat{\fC}, \sS_{\upsilon})\ar[r]& h^i(E,\fC,\sS_{\sigma}).
		\end{tikzcd}
	\end{equation*}
	Then the claim follows from this by push-pull adjunction.
\end{proof}

\begin{proposition}\label{prop:matfac-iso}
	The map \eqref{eq:sheaf-map} is an isomorphism. 
\end{proposition}
\begin{proof}
	This may be checked locally on $X$, so that we may reduce to the case that $E$ is the standard quadratic and that $\sS_{\sigma}$ is the standard $\Cl(E)$-module as in \ref{app:cliff-rep}. Moreover, we may assume that one has an orthogonal splitting $E=F\oplus( K\oplus K^{\vee})$, so that $\fC = \overline{\fC}\oplus K \oplus 0$. We then have 
	\[\Cl(E)\simeq \Cl(F)\otimes_{\bZ_2} \Cl(K\oplus K^{\vee}).\] 
	We may further assume that $\sS_{\sigma} = \sS_{\tau} \otimes_{\bZ_2} \sS_K$ for the $\Cl(K\oplus K^*)$-module $\sS_K = \bigwedge^{\bullet}K^{\vee}$. Here, the action on the tensor product is in the graded sense, i.e. $x\otimes y (s\otimes t) = (-1)^{|s|\cdot|y|} xs\otimes yt$ whenever $y,s$ are homogeneous. 
	In this case, one has $s_{E} =(\pi_K^* s_F, \pi_F^*s_K)$, where $s_K$ is the tautological section of $\pi_K^*K$ on the total space of $K$. 
	One gets an identification of complexes on $\fC$:
	\[C^{\bullet}(E_{\fC}, s_{\fC}, \sS_{\sigma,\fC}) \simeq C^{\bullet}(F_{\fC}, s_E|_{\fC}, \sS_{\tau,\fC}) \otimes_{\operatorname{2-per}} \pi_F^* C^{\bullet}(\pi_K^*(K\oplus K^*), s_K,\pi_K^* \sS_K ),\]
	where the tensor product is in the sense of $2$-periodic complexes.
	Here, the last complex is just the $2$-periodization of the Koszul complex $K^{\bullet}(0_K)$ associated to the zero-section of the total space of $K$. This gives an identification
	\[C^{\bullet}(E_{\fC}, s_{\fC}, \sS_{\sigma,\fC}) \simeq C^{\bullet}(F_{\fC}, s_E|_{\fC}, \sS_{\tau,\fC})\otimes K^{\bullet}(0_K),\]
	where we now take the usual tensor product of complexes. Moreover, the inclusion \eqref{eq:complex-inclusion} is identified with the result of tensoring the (bad) truncation map  
	\[\mathcal{O}_{E}\to \pi_F^* C^{\bullet}(\pi_K^*(K\oplus K^*), s_K,\pi_K^* \sS_K )\]
	with $\pi^*_K C^{\bullet}(F_{\overline{\fC}}, s_{\overline{\fC}}, \sS_{\tau,\overline{\fC}})$.
	
	By adjunction, and the projection formula, the Proposition now follows from the the fact that the Koszul resolution of a regular immersion is a resolution of the structure sheaf.
\end{proof}

The following is the key result of this sub-section. Note that a DT4 obstruction theory is quadratic in our sense only after a shift.
\begin{theorem}\label{thm:DT4-local}
	Let $X$ be a scheme that is quasi-projective over an affine scheme, and let $\bE[1]\to L_X$ be an oriented DT4 obstruction theory on $X$, so that $\bE$ is an oriented quadratic complex. There is a natural functor $\bS(\bE)\to \Coh(X)^{\oplus \bZ_2}$. Over any given component of $X$, it sends the non-trivial automorphism of any object of $\bS(\bE)$ to multiplication by $-1$.  
\end{theorem}
\begin{proof}
	The data of the DT4 obstruction theory provides for each $E^{\bullet}\in \mathcal{R}_{\bE}$ an isotropic cone $\fC_{X,E}\subset E^0$, with the property that for each isotropic reduction $F^{\bullet}\rightsquigarrow E^{\bullet}$, the induced morphism of pairs $(F, \det F_+^{\bullet}[1])\to (E,\det E_+^{\bullet}[1])$ with isotropic subspace $K\subset E$ and identification $K^{\perp}/K\simeq F$ satisfies $K\subset \fC_{X,E}\subset K^\perp$ and identifies $\fC_{X,F}$ with $\fC_{X,E}/K$.
	Thus, we can apply Construction \ref{constr:matfac-iso} to any morphism in $\mathcal{S}_{\bE}$, which by Lemma \ref{lem:matfac-functorial} gives a functor $\mathcal{S}_{\bE}\to \Coh^{\oplus \ZT}$ that by Lemma \ref{prop:matfac-iso} sends any morphism to an isomorphism. 
	
	Hence, we get an induced functor
	\[\bS(\bE)\to \operatorname{Fun}(\mathcal{R}_{\bE}, (\Coh^{\oplus \ZT})^{\simeq}).\]
	By Proposition \ref{prop:simply-connected} any element of the latter functor category is locally constant. Hence, taking the limit gives an equivalence 
	\[\operatorname{Fun}\left(\mathcal{R}_{\bE}, \left(\Coh^{\oplus \ZT}\right)^{\simeq}\right) \to \left(\Coh^{\oplus \ZT}\right)^{\simeq}.\] 
	Composing, we obtain the functor of the theorem statement. 
	The statement about the action of automorphisms follows immediately from the construction.
\end{proof}

\subsection{Global construction of a virtual structure sheaf}\label{sec:str-sheaf}
Now let $\mathcal{X}$ be an algebraic stack with an oriented DT4 obstruction theory $\bE[1]\to L_{\mathcal{X}}$ as in Definition \ref{def:DT4-ob-thy}, with its groupoid of spin structures $\bS(\bE)$ (Definition \ref{def:spin-structure}).

By combining the construction used in Theorem \ref{thm:DT4-local} with descent we can construct a virtual structure sheaf on $\mathcal{X}$. For this, we will identify $\bS(\bE)$ with the category of sections of $\underline{\bS}^{\Aff}(\bE)$ via Proposition \ref{prop:spin-structure-stack-alternate}.  

\paragraph{Value on affines.}

\begin{construction}\label{constr:str-sheaf-affine}
By Construction \ref{constr:sub-cone} and Remark \ref{rem:sub-cone-iso}, to any $t:T\to \mathcal{X}$ in $\Aff_{\mathcal{X}}$ and any self-dual representative $E^{\bullet}$ of $t^*\bE$, we have an induced sub-cone $\fC_{\mathcal{X},t,E}\subset E_1$. For a spin structure $\sigma \in \bS(E^{\bullet})$, we may apply Construction \ref{constr:matfac-cone}. Furthermore, for any morphism in $\mathcal{S}_{t^*\bE}$, we are in the situation of Construction \ref{constr:matfac-iso}. By Lemma \ref{lem:matfac-functorial}, we obtain a functor \[\mathcal{S}_{t^*\bE}\to \left(\Coh(T)^{\oplus \bZ_2}\right)^{\simeq}\]
which induces the first functor in
\[\bS^{\Aff}(t^*\bE)\to \operatorname{Fun}\left(\mathcal{R}_{t^*\bE}, \left(\Coh(T)^{\oplus \bZ_2}\right)^{\simeq}\right) \xrightarrow{\sim}  \left(\Coh(T)^{\oplus \bZ_2}\right)^{\simeq}.\]
The second one is taking the limit over $\mathcal{R}_{t^*\bE}$ and is an equivalence by Proposition \ref{prop:simply-connected}. 
\end{construction}

For any $\sigma_t\in \bS^{\Aff}(t^*\bE)$, we denote the resulting $\bZ_2$-graded coherent sheaf by $\mathcal{O}_{\mathcal{X}}^{\vir}(t,\sigma_t)$. 

\paragraph{Value on morphisms.}
Let $x:(s,S)\to (t,T)$ be a morphism in $\Aff_{\mathcal{X}}$ and let $\sigma_x : \sigma_s\to \sigma_t$ a morphism in $\underline{\bS}^{\Aff}(\bE)$ lying over it. Our goal is to produce an isomorphism 
\[\mathcal{O}_{\mathcal{X}}^{\vir}(\sigma_x): \mathcal{O}_{\mathcal{X}}^{\vir}(s,\sigma_s)\xrightarrow{\sim} x^*\mathcal{O}_{X}^{\vir}(t,\sigma_t).\]

Recall the category fibered in groupoids $\mathcal{S}'_{t^*\bE,x}\to \mathcal{R}_{t^*\bE}$ defined in Definition \ref{def:spin-category-pullback} with $\bS'(t^*\bE,x)$ its category of sections. Then $\sigma_x$ is -- by definition -- an isomorphism of the induced sections $\sigma_s\circ x^* \to x^*\circ \sigma_t$ in $\bS'(t^*\bE,x)$. 
By the same recipe as in Construction \ref{constr:str-sheaf-affine}, we have a natural functor $\mathcal{S}_{t^*\bE,x}'\to (\Coh(S)^{\oplus \ZT})^{\simeq}$, which fits into a natural $2$-commutative diagram 
\begin{equation*}
	\begin{tikzcd}
		\mathcal{S}_{t^*\bE}\ar[r]\ar[d]&\mathcal{S}'_{t^*\bE,x}\ar[r]\ar[d]& \mathcal{S}_{s^*\bE}\ar[d] \\
		(\Coh(T)^{\oplus \ZT})^{\simeq}\ar[r,"x^*"]&(\Coh(S)^{\oplus \ZT})^{\simeq}\ar[r, equal]&  (\Coh(S)^{\oplus \ZT})^{\simeq}
	\end{tikzcd}
\end{equation*}
where the commutativity of the left square is the compatibility of Construction \ref{constr:matfac-cone} with pullbacks, and the right square is commutative on the nose. In the same way as there, we also obtain a functor
\[\bS'(t^*\bE,x)\to (\Coh(S)^{\oplus \ZT})^{\simeq}\]
which we denote $\mathcal{O}_{\mathcal{X}}^{\vir}(t,x, -)$.
We have an induced $2$-commutative diagram of functors
\begin{equation*}
	\begin{tikzcd}
		\operatorname{Fun}\left(\mathcal{R}_{t^*\bE}, \left(\Coh(T)^{\oplus \ZT}\right)^{\simeq}\right)\ar[r,"{x^*\circ -}"]\ar[d]&  \operatorname{Fun}\left(\mathcal{R}_{t^*\bE}, \left(\Coh(S)^{\ZT}\right)^{\simeq}\right)\ar[d] & \ar[l,"{- \circ x^*}"]\ar[d] \operatorname{Fun}\left(\mathcal{R}_{s^*\bE}, \left(\Coh(S)^{\ZT}\right)^{\simeq}\right)\\
		  \left(\Coh(T)^{\ZT}\right)^{\simeq} \ar[r,"x^*"] &  \left(\Coh(S)^{\ZT}\right)^{\simeq}&  \left(\Coh(S)^{\ZT}\right)^{\simeq}\ar[l,equals]
	\end{tikzcd}
\end{equation*}
Here, the vertical maps are the limit functors.
The commutativity of both squares is due to all limits involved being over essentially constant functors by Proposition \ref{prop:simply-connected}.

The sections $\sigma_{s},\sigma_t$ and the morphism $\sigma_x$ between them then give rise to the following isomorphisms (as desired)
\[\mathcal{O}_{\mathcal{X}}^{\vir}(s,\sigma_s)\xrightarrow{\sim} \mathcal{O}_{\mathcal{X}}^{\vir}(t,x,\sigma_s\circ x^*)\xrightarrow{\sim}\mathcal{O}_{\mathcal{X}}^{\vir}(t,x, x^*\sigma_t)\xleftarrow{\sim} x^*\mathcal{O}_{\mathcal{X}}^{\vir}(t,\sigma_t).\]

\begin{remark}\label{rem:ev-pt}
	Instead of taking limits, we may also choose a representative $E^{\bullet}$ of $t^*\bE$ and evaluate at $E^{\bullet}$ and $x^*E^{\bullet}$ respectively. The resulting sheaves fit into a commutative diagram
	\begin{equation*}
		\begin{tikzcd}
			\mathcal{O}_{\mathcal{X}}^{\vir}(s,\sigma_s)\ar[r]\ar[d]& x^*\mathcal{O}_{\mathcal{X}}^{\vir}(t,\sigma_t)\ar[d] \\
			\mathcal{O}_{\mathcal{X}}^{\vir}(s,\sigma_s|_{x^*E^{\bullet}})\ar[r] & x^*\mathcal{O}_{\mathcal{X}}^{\vir} (t,\sigma_t|_{E^{\bullet}})
		\end{tikzcd}
	\end{equation*}
	Here the bottom map is induced by the restriction of $\sigma_x$ to an isomorphism $\sigma_s|_{x^*E^{\bullet}}\to x^*\sigma_t|_{E^{\bullet}}$ via compatibility of Construction \ref{constr:matfac-iso} with pullback of schemes.
\end{remark}
\paragraph{Functoriality.}
We claim that for $x:(S,s)\to (T,t)$ and $y: (T,t)\to (U,u)$ composable morphisms in $\Aff_{\mathcal{X}}$, and for comparison morphisms $\sigma_x:\sigma_s\to \sigma_t$ and $\sigma_{y}:\sigma_t\to \sigma_u$, the diagram 
\begin{equation*}
	\begin{tikzcd}
		\mathcal{O}_{\mathcal{X}}(s, \sigma_s)\ar[r,"\mathcal{O}_{\mathcal{X}}^{\vir}(\sigma_x)"]\ar[dr, "\mathcal{O}_{\mathcal{X}}^{\vir}(\sigma_y\circ \sigma_x)"']& x^*\mathcal{O}_{\mathcal{X}}(t, \sigma_t)\ar[d, "x^*\mathcal{O}_{\mathcal{X}}^{\vir}(\sigma_y)"] \\
		& x^*y^*\mathcal{O}_{\mathcal{X}}(u,\sigma_u)
	\end{tikzcd}
\end{equation*}
commutes. 
To do so, we may choose a representative $E^{\bullet}$ of $u^*\bE$ on $U$. Then we have commutative diagrams as in Remark \ref{rem:ev-pt}, and the result follows from functoriality of Construction \ref{constr:matfac-iso} with respect to pullback of schemes. 
Putting the constructions of this subsection together we have
\begin{theorem}\label{thm:twisted-str-sheaf}
		Let $\mathcal{X}$ be an algebraic stack carrying an oriented DT4 obstruction theory $\bE[1]\to L_{\mathcal{X}}$ of even rank.
\begin{enumerate}[label =\roman*)]
	\item We have a natural functor $\mathcal{O}_{\mathcal{X}}^{\vir}:\bS(\bE)\to (\Coh(\mathcal{X})^{\ZT})^{\simeq}$ compatible with base change in $\mathcal{X}$. For any object of $\bS(\bE)$, the everywhere non-trivial automorphism is sent to multiplication by $-1$.  
	\item The $\bZ_2$-gerbe $\mathcal{X}^{\operatorname{sp}}$ of Corollary \ref{cor:spin-gerbe} carries a canonical $1$-twisted (and $\ZT$-graded) sheaf  $\mathcal{O}_{\mathcal{X}}^{\vir}\in \Coh^1(\mathcal{X}^{\operatorname{sp}})^{\oplus \bZ_2}$. 
\end{enumerate}	 
\end{theorem}

\begin{remark}
	If $\mathcal{N}$ is a fixed line bundle on $\mathcal{X}$, everything should go through if one works with $\Spin^{\bC}$-structures on $(\bE,\mathcal{N})$ in place of spin-structures on $\bE$. Then one obtains a $1$-twisted $\bZ_2$-graded sheaf $\mathcal{O}^{\vir}_{\mathcal{X},\mathcal{N}}$ on the gerbe parametrizing $\Spin^{\bC}$-structures on $(\bE,\mathcal{N})$ described in Remark \ref{rem:spinc-structures}. Morally, this sheaf differs from $\mathcal{O}_{\mathcal{X}}^{\vir}$ by a twist with a square root of $\mathcal{N}$. This can be made more precise, for example by passing to $K$-theory on $X$ with $2$ inverted via Theorem \ref{thm:untw}.  
\end{remark}
\subsection{Comparison with the Oh--Thomas class}
Let $X$ be a quasi-projective scheme over the complex numbers with an oriented DT4 obstruction theory $\bE[1]\to L_X$ of even rank. We exhibit how our twisted virtual structure sheaf $\mathcal{O}_X^{\vir}$ recovers the $K$-theory class constructed by Oh--Thomas. More precisely, we recover their class by first passing to $K$-theory and then applying an untwisted pushforward operation to pass from the spin gerbe back to $X$ -- this is developed more generally in Appendix \ref{app:untw}. We expect that a similar result holds in the torus equivariant case if one takes care of the additional subtleties that already appear in the Oh--Thomas's construction \cite[p. 53 and below]{ot_counting}. 

\begin{theorem}\label{thm:str-shvs-compatible}
	Let $\pi_X:X^{\operatorname{sp}}\to X$ denote the spin gerbe and let $\mathcal{O}_X^{\vir}$ denote the $1$-twisted virtual structure sheaf of Theorem \ref{thm:twisted-str-sheaf}. Let 
	\[\widetilde{\pi}_{X,*}:K(X^{\operatorname{sp}}) \to K(X)[1/2] \]
	denote the untwisted pushforward operation of Theorem \ref{thm:untw}.
	Then 
	\[\widetilde{\pi}_{X,*}\left[\mathcal{O}_X^{\vir}\right] \]
	is the Oh--Thomas twisted virtual structure sheaf of \cite[Definition 5.9]{ot_counting}.
\end{theorem}
\begin{proof}
	Since we are on a quasi-projective scheme, we may calculate with a fixed resolution $E^{\bullet}$ of $\bE$, say
	\[E^{\bullet} = \left[B_{-1}\to E\to B^1\right].\] 
	Both constructions work with the induced isotropic cone 
	\[\fC_{X,E}\subset E.\]
	From here, Oh--Thomas proceed by applying a \emph{K-theoretic localized square-root Euler class} and twisting by $\sqrt{\det B^1}$ (which is well-defined after inverting $2$ in $K$-theory), while our construction proceeds by passing to the gerbe parametrizing $\Spin^{\bC}$-structures on the pair $(E,\det B^1)$ and applying Construction \ref{constr:matfac-cone} to the universal $\Cl(E)$-module on that gerbe. Either of these constructions is compatible with flat pullback, so as in \cite[Definition 5.7]{ot_counting} we may pass to a suitable bundle $\tilde{X}\to X$ of isotropic flags associated to $E$ and do the comparison there. In particular, we may assume that $E$ carries a maximal isotropic subbundle $\Lambda$ compatible with the orientation. In this case, via $\Lambda$, we obtain a Clifford module $\sS_{\Lambda}$ for $E$ satisfying $\sS_{\Lambda}\otimes_{\Cl(E)}\sS_{\Lambda}\simeq \det \Lambda$. Hence, the existence of a $\Spin^{\bC}$ structure on $(E,\det B^1)$ is obstructed precisely by $\sqrt{\det \Lambda}^{-1}\otimes \sqrt{\det B^1}$. By the construction of $\widetilde{\pi}_{X}$ and by Construction \ref{constr:matfac-cone}, we find that 
	\[\widetilde{\pi}_{X,*}\left[\mathcal{O}_X^{\vir}\right] = \left(h^i(E,\fC_{X,E}, \sS_{\Lambda})\right)_{i=0,1} \otimes \sqrt{\det B^1}\otimes \sqrt{\det \Lambda}.\]
	On the other hand, the expression on the right hand side agrees with the class of Oh--Thomas due to results of Kim--Oh and Oh--Sreedhar \cite{ko_loc}\cite{os_loc}. This is explained in detail in \cite[Proposition 3.8]{kr_mag}.
\end{proof}

\appendix

\section{Quadratic forms and associated groups}
We collect some standard facts about the algebraic groups $\OO(n), \SO(n)$ and $\Spin(n)$, the categories of torsors over them and the relationship to Clifford algebras. A reference is \cite{calfas_groupes} particularly Chapter 4. For conventions regarding determinants (though not for orientations!) we follow \cite{ot_counting}. Throughout, we let $R$ denote a ring over $\bZ[1/2]$ and $X$ a scheme over $\Spec R$. As we are outside characteristic $2$, we will identify a quadratic form with its associated bilinear form. 
\subsection{Standard forms of $\OO(n)$ and $\SO(n)$}\label{sec:standard-forms}
We fix conventions regarding the orthogonal and special orthogonal groups.

For $n\geq 0$, let 
\[B_n := \begin{bmatrix}
	& & 1 \\
	& \iddots &\\ 
	1 & &
\end{bmatrix}\]
be the $n\times n$ matrix with entries $1$ along the anti-diagonal. We let $q_n$ denote the associated bilinear form on $R^n$. Explicitly:
\[q_n(\sum x_i e_i, \sum y_je_j) = x_1y_n + \cdots +x_n y_1.\]

\paragraph{Orthogonal groups.}
We define $\OO(n)$ to be the closed subgroup of $\GL(n)$ of matrices preserving $q_n$. Equivalently
\[\OO(n):= \{ A\in \GL(n) | A^T B_n A = B_n\},\]
understood as an equality of $X$-valued points for all $R$-schemes $X$.

\begin{definition}
	A \emph{quadratic bundle of rank $n$} on a scheme $X$ is a pair $(M,q_M)$, where $M$ is a locally free $\mathcal{O}_X$-module of rank $n$ and where $q_M:M\otimes M\to \mathcal{O}_X$ is a non-degenerate symmetric bilinear form on $M$. An isomorphism of quadratic modules is an isomorphism of their underlying $\mathcal{O}_X$-modules that preserves the bilinear forms. 
\end{definition}
We let the \emph{standard quadratic bundle of rank $n$} be the pullback of $(R^n, q_n)$ along the structure map $X\to \Spec R$, and denote it by $(\mathcal{O}_X^n,q_n)$. 

\begin{proposition}\label{prop:orthogonoal-torsors}
	For any $R$-scheme $X$ we have an equivalence of categories 
	\[\left(\mbox{Quadratic bundles of rank $n$ on $X$ }\right) \xrightarrow{\sim} \left(\mbox{$\OO(n)$-torsors over $X$}\right)\]
	sending a quadratic bundle $(M,q_M)$ to the isom-scheme $\operatorname{Isom}\left((M,q_M), (\mathcal{O}_X^n,q_n)\right)$. 
\end{proposition}
\begin{proof}
	Since we are working away from characteristic $2$, any non-degenerate quadratic bundle on a scheme $X$ is \'etale locally isomorphic to the standard quadratic bundle by the Gram--Schmidt process. The statement then follows, since essentially by definition, $\OO(n)$ parametrizes the automorphisms of the standard quadratic bundle. For details see Proposition 4.1.0.4 in \cite{calfas_groupes}.
\end{proof}

\paragraph{Special orthogonal groups.}

We define the special orthogonal group as the intersection $\SO(n):=\OO(n)\cap \SL(n)$ of closed sub-groupschemes of $\GL(n)$. Equivalently
\[\SO(n):= \Ker\left(\OO(n)\xrightarrow{\det} \mathbb{G}_m \right).\]
is the sub-group of $\OO(n)$ of elements with determinant $1$. 

To any quadratic bundle $(M,q_M)$ one has the associated isomorphism 
\begin{align*}
	\theta_M: M & \to M^{\vee}\\
			  x & \mapsto q_M(x, -),
\end{align*}
which induces an isomorphism of line bundles
\[\det(\theta_M): \det M \to \det M^{\vee} \simeq \left(\det M\right)^{\vee}.\]
Here, the latter isomorphism is induced from the pairing 
\begin{align*}
	\det M \otimes \det M^{\vee}& \to \mathcal{O}_X\\
	(x_1\wedge \cdots \wedge x_n) \otimes (\varphi_1 \wedge\cdots \wedge \varphi_n) & \mapsto \sum_{\sigma \in S_n} \operatorname{sgn}(\sigma)\,  \varphi_1(x_{\sigma n}) \cdots \varphi_n(x_{\sigma_1}).
\end{align*}
Tensoring $\det \theta_M$ from the left with $\det M$, we obtain an isomorphism
\begin{equation}\label{eq:detqm}
	\det(q_M) : (\det M)^{\otimes 2} \to \mathcal{O}_X.
\end{equation}
\begin{remark}
	From these conventions, one has that for a (local) orthonormal basis $u_1,\ldots,u_n$ of $M$, 
	\[\det(q_M)\left((u_1\wedge \cdots \wedge u_n)^{\otimes 2}\right) = (-1)^{n(n-1)/2}.\]
\end{remark}
\begin{definition}
	Let $(M,q_M)$ be a rank $n$ quadratic bundle on $X$. An \emph{orientation} on $M$ consists of an isomorphism $o_M:\mathcal{O}_X\to \det M$ for which the composition 
	\[\mathcal{O}_X \simeq \mathcal{O}_X^{\otimes 2} \xrightarrow{o_M^{\otimes 2}} \left(\det M\right)^{\otimes 2} \xrightarrow{\det q_M} \mathcal{O}_X \]
	is equal to the identity. An isomorphism between oriented quadratic bundles is an isomorphism of quadratic bundles whose determinant commutes with the orientations. 
\end{definition}
\begin{remark}
	Letting $\varphi_M:= o_M^{-1}: \det M\to \mathcal{O}_X$, this is equivalent to the requirement that 
	\[\varphi_M^{\vee}\circ \varphi_M = \det \theta_M\]
\end{remark}

Consider the standard quadratic bundle $(R^n,q_n)$ and let $e_i$ be the standard basis vectors of $R^n$. Let also $f_i:=e_{n-i}$, so that $q_n(e_i,f_j) = \delta_{i,j}$. If $n = 2m$ is even, we define the \emph{standard orientation} on $q_n$ to be 
\begin{align*}
	o_n: R &\to \det R^n\\
	1& \mapsto \left(e_1\wedge \cdots \wedge e_m\right)\wedge \left(f_m\wedge \cdots \wedge f_1\right) = \left(e_1\wedge f_1\right) \wedge \cdots \wedge \left(e_m\wedge f_m\right).
\end{align*}
If $n=2m+1$ is odd, let $u:=e_{m+1}$, so that $q_n(u,u)=1$. In this, case we define the \emph{standard orientation} to be
\begin{align*}
	o_n: R &\to \det R^n\\
	1& \mapsto \left(e_1\wedge \cdots \wedge e_m\right)\wedge \left(f_m\wedge \cdots \wedge f_1\right)\wedge u = \left(e_1\wedge f_1\right) \wedge \cdots \wedge \left(e_m\wedge f_m\right)\wedge u.
\end{align*}
\begin{remark}
	If $\Spec R$ is connected, there are exactly two orientations on the standard quadratic bundle. For $n>0$, the resulting oriented quadratic bundles are isomorphic
\end{remark}

On any scheme $X$, the \emph{standard oriented quadratic bundle} $(\mathcal{O}_X^n,q_n,o_n)$ is defined as the pull-back of the standard quadratic bundle with its standard orientation.

\begin{proposition}\label{prop:special-orthogonal-torsors}
	For any $R$-scheme $X$ and $n>0$ we have an equivalence of categories 
	\[\left(\mbox{Oriented quadratic bundles of rank $n$ on $X$ }\right) \xrightarrow{\sim} \left(\mbox{$\SO(n)$-torsors over $X$}\right)\]
	sending an oriented quadratic bundle $(M,q_M, o_M)$ to the isomorphism scheme \[\operatorname{Isom}\left((M,q_M, o_M), (\mathcal{O}_X^n,q_n,o_n)\right).\] 
\end{proposition}
\begin{proof}
	As in the proof of Proposition \ref{prop:orthogonoal-torsors}, locally any quadratic bundle is isomorphic to the standard quadratic bundle and one can always arrange the isomorphism to preserve orientations.
	Then the result is due to the fact that $\SO(n)$ is exactly the sub-group of $\OO(n)$ preserving the standard orientation. 
	For more details, see Proposition 4.3.0.30 and Proposition 4.3.0.31 in \cite{calfas_groupes}.
\end{proof}

\begin{example}
	Let $K$ be a locally free sheaf of rank $k$ on $X$. The \emph{hyperbolic quadratic bundle} associated to $K$ is the quadratic bundle $(K\oplus K^{\vee}, q_K)$, where $q_K(x,\lambda):=\lambda(x)$. It has a natural orientation given locally by $1\mapsto (e_1\wedge \cdots \wedge e_k)\wedge (f_k\wedge \cdots \wedge f_1)$ for any local choice of basis $e_1,\ldots, e_k$ of $K$ with dual basis $f_1,\ldots,f_k$.  
\end{example}
\subsection{Clifford Algebras}\label{app:cl-algs}

Let $M$ be a rank $n$ locally free $R$-module and $q_M:M\otimes M\to R$ a symmetric bilinear pairing on $M$ (not assumed to be non-degenerate here).
\begin{definition}
	The \emph{Clifford algebra associated to $(M,q_M)$}, denoted $\Cl(M)$ is the quotient of the tensor algebra $\oplus_{i\geq 0} T^i(M)$ by the two-sided ideal generated by the relations $v\otimes v - q_M(v,v)$ for $x\in M$. Note that we have a canonical morphism $\iota_M:M\to \Cl(M)$ via the identification $T^1(M)=M$.
\end{definition}
One defines the analogous notion for a pair $(M,q_M)$ over a scheme $X$ by glueing the local construction. The Clifford algebra has the following universal property 
\begin{itemize}
	\item Let $\mathcal{A}$ be a sheaf of $\mathcal{O}_{X}$-algebras, and let $\iota_{\mathcal{A}}: M\to \mathcal{A}$ be a morphism of $\mathcal{O}_X$-modules, such that for each section $m$ of $M$, one has $m*_{\mathcal{A}} m = q_M(m,m)$. Then there exists a unique homomorphism of $\mathcal{O}_X$-algebras $\Phi:\Cl(M)\to \mathcal{A}$ such that $\iota_{\mathcal{A}} = \Phi\circ \iota_M$. 
\end{itemize}

One has the following properties:
\begin{proposition}\label{prop:clifford-props}
	Let $(M,q_M)$ be a locally free $\mathcal{O}_X$-module together with a symmetric bilinear pairing. 
	
	\begin{itemize}
		\item The Clifford algebra $\Cl(M)$ is an $\mathcal{O}_X$-algebra whose underlying module is locally free of rank $2^{n}$. 	
		\item The Clifford algebra has a natural $\ZT$-grading 
		\[\Cl(M) = \Cl_0(M)\oplus \Cl_1(M)\]
		induced by the grading of the tensor algebra into even and odd parts.
		\item There is a natural filtration $F^i\Cl(M)$ on the Clifford algebra, whose $i$-th piece is the image of $\oplus_{j\leq i} T^i (M)$ under the natural quotient map. 
		\item The map $\iota_M$ is injective, and we have natural identifications $F^0\Cl(M) = \mathcal{O}_X$ and $F^1\Cl(M)\simeq \mathcal{O}_X \oplus M$, where $M$ is identified with its image under $\iota_M$. 
		\item There is a natural isomorphism of $\mathcal{O}_X$-modules to the exterior algebra 
		\[\Gamma_M:\Cl(M)\to \bigwedge^{\bullet}M\]
		which is compatible with filtrations. 
		If $u_1,\ldots,u_n$ is a local orthogonal basis, it is given by 
		\[u_{i_1}u_{i_2}\cdots u_{i_k}\mapsto u_{i_1}\wedge \cdots \wedge u_{i_k},\]
		whenever $1\leq i_1 < \cdots <i_k\leq n$. 
		\item There is a unique anti-automorphism $\sigma_M:\Cl(M)\to \Cl(M)$ that acts as multiplication by $-1$ on the image of $\iota_M$. 
		\item The construction of the Clifford algebra is functorial: If $f:M_1\to M_2$ is a morphism of $\mathcal{O}_X$ modules that respects given pairings $q_1$ and $q_2$, there is an induced homomorphism of algebras
		\[\Cl(f):\Cl(M_1)\to \Cl(M_2)\]
		that restricts to $f$ on $M_1$. 
	\end{itemize}
\end{proposition}

Using the isomorphism $\Gamma_M$, we have a natural inclusion 
\begin{align*}
	\det M &\to \Cl(M) 
\end{align*}
given as the composition $\det M = \bigwedge^n M \xrightarrow{\Gamma_M^{-1}} \Cl(M)$. We always have that $(\det M)^2 \subset \mathcal{O}_X$, where the product is taken in the Clifford algebra. 

Now we return to the case where the bilinear pairing is non-degenerate.
\begin{definition}
	Let $(M,q_M)$ be a quadratic bundle. A \emph{volume element} for $\Cl(M)$ is a global section $\omega$ of $\Cl(M)$ that factors through the inclusion $\det M\to \Cl(M)$ and satisfies $\omega^2 = 1\in \mathcal{O}_X\subseteq \Cl(M)$. 
\end{definition}

If $q_M$ is non-degenerate, and choosing a local orthonormal basis $u_1,\ldots, u_n$, one calculates that in $\Cl(M)$:
\[(u_1\cdots u_n)^2 = (-1)^{n(n-1)/2} u_1^2\cdots u_n^2 = (-1)^{n(n-1)/2}.\]
As a consequence, we conclude the following
\begin{lemma}
	For a quadratic bundle $(M,q_M)$, the induced map 
	\[\det M \otimes \det M \to \mathcal{O}_X \]
	given by multiplication inside the Clifford algebra is identical to $\det(q_M)$ of \eqref{eq:detqm}. In particular, a choice of volume element for the Clifford algebra is the same as a choice of orientation for $(M,q_M)$. 
\end{lemma}
\begin{example}
	If $n = 2m$ is even, the volume element associated to the standard oriented quadratic form can be written in terms of the basis $e_1,\ldots,e_m,f_m,\ldots,f_1$ as 
	\[\omega_n = (e_1 f_1-1)\cdots (e_m f_m - 1).\]
	
	If $n=2m+1$ is odd, one has instead
	\[\omega_n  = (e_1 f_1-1)\cdots (e_m f_m - 1) u.\]
\end{example}

\subsection{The Spin groups}\label{sec:appendix-spin-groups}
Consider the standard quadratic bundle $(R^n, q_n)$ over $R$, and let 
\[\Cl(n):=\Cl(R^n)\]
denote the associated Clifford algebra. 
There is an associated \'etale sheaf of groups $\lip_n\subset \Cl_0(n)^{\times}$, defined as a certain subsheaf of the group of multiplicative units of the even part of the Clifford algebra. Over a scheme $X$, the sections of $\lip_n$ are those sections of $\Cl^0(n)^{\times}$ that preserve the subspace $R^{\oplus n}\subset \Cl(n)$ (after any further possible restriction along a map of schemes $Y\to X$) under the conjugation action.

The sheaf $\lip_n$ is representable by an affine group scheme over $R$, see \cite{calfas_groupes} after Def. 4.5.1.2 (see also Proposition 2.4.2.1 there). We denote this group scheme also by $\lip_n$, and will call it the \emph{special Lipschitz group}.

Since $\lip_n$ preserves the sub-space $R^n\subseteq \Cl(n)$, any section of $\lip_n$ induces a section of $\GL(n)$ by restriction. It can be shown to preserve the quadratic form and the orientation, hence one has a group homomorphism
\[\lip_n\to \SO(n).\]

Using the anti-automorphism $\sigma$, one defines the spinor norm of an element of $\lip_n$:
\[\operatorname{sn}(x):= \sigma(x)x.\]
\begin{lemma}[cf. Def. 4.5.1.8 in \cite{calfas_groupes}]
	The spinor norm of any element of $\lip_n$ takes values in $R^{\times}\subset R\subset \Cl(n)$ and defines a group homomorphism 
	\begin{equation}\label{eq:spinor-norm}
		\operatorname{sn}: \lip_n \to \mathbb{G}_m.
	\end{equation} 
\end{lemma}

\begin{definition}
	We define the \emph{spin group} as the kernel of the spinor norm:
	\[\Spin(n):=\Ker(\lip_n \xrightarrow{\operatorname{sn}} \mathbb{G}_m).\] 
\end{definition}

As expected, the spin groups give double covers of $\SO(n)$ \cite[Propositions 4.6.0.8 and 4.6.0.9]{calfas_groupes}.
\begin{proposition}
	We have an exact sequence of algebraic groups 
	\[0\to \bZ_2 \to \Spin(n)\to \SO(n)\to 0,\]
	where the map to $\SO(n)$ is induced from the conjugation action.
\end{proposition}

\begin{definition}
	We define the group $\Spin^{\bC}(n)$ as the quotient
	\[\Spin^{\bC}(n):= \Spin(n)\times_{\bZ_2}\bG_m,\]
	along the unique inclusions of $\bZ_2$ as the normal subgroup in $\Spin(n)$ and $\bG_m$ respectively.
\end{definition}
As one can take the quotient of $\Spin(n)\times \bG_m$ by the larger subgroup $\bZ_2\times \bZ_2$, we have the exact sequence
\[0\to \bZ_2 \to \Spin^{\bC}(n) \to SO(n)\times \bG_m\to 0.\]

Note that the map 
\begin{align*}
	\Spin(n)\times \bG_m &\to \Cl_0(n)^{\times}\\
	(\alpha, \lambda)&\mapsto \lambda\alpha
\end{align*}
defines a group homomorphism whose kernel is exactly the diagonal subgroup $\bZ_2$. This induces an injective homomorphism $\Spin^{\bC}(n)\to \Cl_0(n)^{\times}$ identifying $\Spin^{\bC}(n)$ with the special Lipschitz group $\lip_n$. 
\subsection{Clifford representations}\label{app:cliff-rep}
We discuss the representations of the Clifford algebra, which gives rise to another characterization of the $\Spin^{\bC}$-groups. For this discussion, it is convenient to distinguish the even and odd dimensional cases. We will also restrict the discussion mostly to the case of the standard Clifford algebra. 

\paragraph{Even rank.}
Let $n = 2m$ be even and consider the standard quadratic bundle $R^{n}$. Let $\Lambda^{\vee} = \langle f_1,\ldots f_m \rangle$ be the maximal isotropic subspace of $R^n$ generated by the last $m$ basis vectors, which we regard as dual to $\Lambda:= \langle e_1,\ldots,e_m\rangle$. We define an action of $\Cl(n)$ on the exterior algebra $\bigwedge^{\bullet}\Lambda^{\vee}$ via
\begin{equation}\label{eq:cliff-action-formulas}
\begin{aligned}
	e_i (f_{i_1}\wedge \cdots \wedge f_{i_k}) := &2 e_i \lrcorner (f_{i_1}\wedge \cdots \wedge f_{i_k})\\
	f_i(f_{i_1}\wedge \cdots \wedge f_{i_k}) := &f_i\wedge f_{i_1}\wedge \cdots \wedge f_{i_k}.
\end{aligned}
\end{equation}
for $1\leq i\leq m$. 
Here $\lrcorner$ denotes the contraction, which by definition is given by 
\[y \lrcorner (x_1\wedge \cdots \wedge x_k) := \sum_{j=1}^k (-1)^{j-1} q_n(y,x_j)\,  x_1\wedge \cdots \wedge \widehat{x_j}\wedge \cdots \wedge x_k,\]
where the hat symbol denotes leaving out the corresponding term. 
We denote the resulting $\Cl(n)$-module by $\sS_n$. As an $R$-module it is free of rank $2^m$.
We have a decomposition 
\[\bigwedge^{\bullet} \Lambda^{\vee}  = \bigwedge^{ev} \Lambda^{\vee} \oplus \bigwedge^{odd}\Lambda^{\vee}\]
into even and odd wedge powers, which gives a $\ZT$-grading such that the Clifford algebra acts in a graded way. Note that the volume element $\omega = e_1\cdots e_mf_m\cdots f_1$ corresponding to the standard orientation acts with weight $1$ on the even part and with weight $-1$ on the odd part. 
The action of the Clifford algebra induces homomorphism of graded $R$-algebras
\[\Cl(n) \to \End_{R}(\sS_n),\]
which is in fact an isomorphism. Using the involution $\sigma$, any left Clifford module carries a structure of right module. Using this to form the tensor product, we have mutually inverse isomorphisms of $R$-modules 
\begin{align*}
	R &\to \sS_n \otimes_{\Cl(n)} \sS_n & \sS_n\otimes_{\Cl(n)}\sS_n &\to \bigwedge^m\langle f_1,\ldots,f_m\rangle \simeq  R\\ 
	1 & \mapsto 1 \otimes f_{1}\wedge \cdots \wedge f_m & x\otimes y & \mapsto (-1)^{|x|}\operatorname{rev}(x) \wedge y. 
\end{align*}
Here, we implicitly project into the degree $m$ part of the exterior algebra and $\operatorname{rev}$ is the automorphism of the wedge product that inverts the order of factors. On degree $i$ it acts by $(-1)^{i(i-1)/2}$.  
We now give an alternate description of the group $\Spin^{\bC}(n)$.
Let $\Aut(R^n,q_n,o_n,\sS_n)$ denote the sheaf of groups over $R$ parametrizing pairs $(\varphi,\Phi)\in \SO(n)\times \GL(\sS_n)$ for which the following diagram commutes 
\begin{equation}\label{diag:cl-act-comm}
	\begin{tikzcd}
		\Cl(n) \otimes \sS_n \ar[r]\ar[d,"{\Cl(\varphi)\otimes \Phi}"]& \sS_n\ar[d,"\Phi"] \\
		\Cl(n) \otimes \sS_n\ar[r] & \sS_n,
	\end{tikzcd}
\end{equation} 
where the horizontal arrows are the action map. 

\begin{proposition}\label{prop:spinc-autom-even}
	The map $\Spin^{\bC}(n)\to \SO(n)\times \GL(\sS_n)$ given in the first coordinate by the natural map $\Spin^{\bC}(n)\to \SO(n)$, and in the second coordinate by the Clifford action is an isomorphism onto $\Aut(R^n,q_n,o_n,\sS_n)$. 
\end{proposition}
\begin{proof}
	Let $(\varphi_{\alpha},\Phi_{\alpha})$ denote the image of a section $\alpha$ of $\Spin^{\bC}(n)$. To see that this lies in $\Aut(R^n,q_n,o_n,\sS_n)$, note that $\Cl(\varphi_{\alpha})$ is nothing but the conjugation action of $\alpha$ on $\Cl(n)$. Thus the commutativity of \eqref{diag:cl-act-comm} reduces to the formula
	\[(\alpha v \alpha^{-1})(\alpha s) = \alpha(v s).\]
	Injectivity holds, since $\Spin^{\bC}(n)$ is a multiplicative sub-monoid of $\Cl(n)$, and the map $\Cl(n)\to \End(\sS_n)$ is injective.
	
	To see surjectivity, let $(\varphi,\Phi)\in \Aut(R^n,q_n,o_n,\sS_n)$ be an arbitrary section. Since $\Spin^{\bC}(n)\to \SO(n)$ is surjective, we may (locally) choose a lift of $\varphi$. This reduces us to the case that $\varphi = 1$. Then $\Phi$ is an endomorphism of $\sS_n$ that respects the action of the Clifford algebra, hence equivalently, the action of the matrix algebra $\End(\sS_n)$. This shows that $\Phi$ must be scaling by some $\lambda\in R$, hence $(1,\Phi)$ is the image of $\lambda\in R^{\times }\subset \Spin^{\bC}(n)$.   
\end{proof}

Given any pair $(\varphi,\Phi) \ in \Aut(R^n,q_n,o_n,\sS_n)$ one has that $\Phi$ is also compatible with $\Cl(\varphi)$ with respect to the \emph{right} $\Cl(n)$-modules structure on $\sS_n$ (this uses that $\sigma\circ \Cl(\varphi)\circ \sigma$ is the identity). Hence we obtain a homomorphism of invertible $R$-modules 
\[R\simeq \sS\otimes_{\Cl(n)}\sS \xrightarrow{\Phi\otimes_{\Cl(\varphi)}\Phi} \sS\otimes_{\Cl(n)} \sS \simeq R.\]
In other words $(\varphi,\Phi)\mapsto \Phi\otimes_{\Cl(\varphi)} \Phi$ defines a group homomorphism
\begin{equation}\label{eq:spinor-norm-hom}
	\Spin^{\bC}(n)\simeq \Aut(R^n,q_n,o_n,\sS_n)\to \bG_m
\end{equation}
\begin{proposition}\label{prop:spinor-norm-even}
	The homomorphism \eqref{eq:spinor-norm-hom} agrees with the spinor norm $\operatorname{sn}$ defined in \eqref{eq:spinor-norm}. 
\end{proposition}
\begin{proof}
	Let $(\varphi_{\alpha},\Phi_{\alpha})$ be the image of some $\alpha\in \Spin^{\bC}(n)$. 
	Then for any section of the form $s_1\otimes s_2$ of $\sS_n\otimes_{\Cl(n)}\sS_n$, we have
	\[\Phi_{\alpha}\otimes\Phi_{\alpha}(s_1\otimes s_2) = (\alpha s_1) \otimes (\alpha s_2) = (\sigma(\alpha) \alpha s_1)\otimes s_2 = \operatorname{sn}(\alpha) s_1\otimes s_2.\]
\end{proof}

\paragraph{Odd rank.}

Let $n = 2m+1$, and let again $\Lambda = \langle e_1,\ldots,e_m\rangle$ and $\Lambda^{\vee} = \langle f_1,\ldots,f_m \rangle$, which are maximal isotropic subspaces of $R^n$. We again define an action of $\Cl(n)$ on $\bigwedge^{\bullet}\Lambda^{\vee}$ via the formulas \eqref{eq:cliff-action-formulas} for $1\leq i\leq m$, and with 
\[u (f_{i_1}\wedge \cdots \wedge f_{i_k}):= (-1)^k(f_{i_1}\wedge \cdots \wedge f_{i_k}). \]
We let $\sS_n$ denote the resulting $\Cl(n)$-module. We let $\sS_n'$ denote the $\Cl(n)$-module on which the odd elements of $\Cl(n)$ act with the opposite sign, i.e. with the same underlying $R$-module as $\sS_n$, but with multiplication defined by $(x,s)\mapsto (-1)^{|x|}xs$ for a homogeneous element $x$ of $\Cl(n)$. As $R$-modules, both $\sS_n$ and $\sS'_n$ are locally free of rank $2^m$. The volume element $\omega\in \Cl(n)$ defined by the standard orientation acts as the identity on $\sS_n$ and as minus the identity on $\sS_n'$. 
The actions of the Clifford algebra on $\sS_n$ and $\sS_n'$ induce an isomorphism of $R$-algebras
\[\Cl(n)\to \End_R(\sS_n)\times \End_R(\sS_n').\]

We have an isomorphism of $\Cl(2m)\to \Cl_0(2m+1)$ onto the \emph{even} part of $\Cl(n)$. It is obtained from the universal property of the Clifford algebra, applied to the map
\begin{align*}
	R^m\oplus R^m & \to \Cl_0(2m+1)\\
	(e, f) & \mapsto eu - fu.
\end{align*}
Here, the first summand is the subspace $\Lambda$ and the second is $\Lambda^{\vee}$, cf. \cite[Theorem 4.2.0.16]{calfas_groupes}.
Restricting the action of $\Cl(n)$ on $\sS_n$ to the even part of the Clifford algebra, we obtain an isomorphism
\[\Cl_0(n)\to \End_R(\sS_n).\]
We have mutually inverse isomorphisms of $R$-modules, where we take the tensor product only over the even part of the Clifford algebra
\begin{align*}
	R & \to \sS_n \otimes_{\Cl_0(n)} \sS_n &  \sS_n\otimes_{\Cl_0(n)} \sS_n &\to \bigwedge^m\langle f_1,\ldots, f_m \rangle \simeq R\\
	1&\mapsto 1\otimes f_1\wedge \cdots \wedge f_m, & x\otimes y &\mapsto (-1)^{m|x|} \operatorname{rev}(x)\otimes y.
\end{align*}

Exactly as in the even rank case, we define the group $\Aut(R^n,q_n,o_n,\sS_n)\subset \SO(n)\times \GL(\sS_n)$.
We have the analogue of Proposition \ref{prop:spinc-autom-even} with exactly the same statement, and essentially the same proof (one only needs to change $\Cl(n)$ to $\Cl_0(n)$ in the proof of injectivity).
\begin{proposition}\label{prop:spinc-autom-odd}
	The map $\Spin^{\bC}(m)\to \SO(n)\times \GL(\sS_n)$ given in the first coordinate by the natural map $\Spin^{\bC}(n)\to \SO(n)$, and in the second coordinate by the Clifford action is an isomorphism onto $\Aut(R^n,q_n,o_n,\sS_n)$. 
\end{proposition}

Given $(\varphi,\Phi)\in \Aut(R^n,q_n,o_n,\sS_n)$ we may again form the tensor product of $\Phi$ with itself, this time over $\Cl_0(n)$. This gives a homomorphism
\[R\simeq \sS_n\otimes_{\Cl_0(n)} \sS_n\xrightarrow{\Phi\otimes_{\Cl_0(\varphi)}\Phi}\sS_n\otimes_{\Cl(n)}\sS_n \simeq R.\]
Hence, we again obtain a group homomorphism
\begin{equation}\label{eq:spinor-norm-hom-odd}
	\Spin^{\bC}(n)\simeq \Aut(R^n,q_n,o_n,\sS_n) \to \bG_m.
\end{equation}
\begin{proposition}\label{prop:spinor-norm-odd}
	The homomorphism \eqref{eq:spinor-norm-hom-odd} agrees with the spinor norm $\operatorname{sn}$.
\end{proposition}

\subsection{Isotropic reductions}\label{app:iso-red}
Let $(M,q_M)$ be a rank $n$ quadratic bundle on $X$, and let $K\subseteq M$ be a rank $k$ isotropic sub-bundle. Let $K^{\perp}$ denote the orthogonal complement to $K$ in $M$. Then $K\subset K^{\perp}$, and the locally free sheaf $K^{\perp}/K$ inherits a symmetric bilinear pairing $\overline{q}_M$ satisfying $\overline{q}_M(\overline{x},\overline{y}) = q_M(x,y)$ for any pair of sections $x,y$ of $K^{\perp}$. 

An orientation $o_M:\mathcal{O}_X\to \det M$ on $(M,q_M)$ induces an orientation $\overline{o}_M$ on $(K^{\perp}/K, \overline{q}_M)$ via the formula
\[\mathcal{O}_X\xrightarrow{o_M} \det M \simeq \det(K^{\perp}/K)\otimes \det (K\oplus K^{\vee}) \xrightarrow{1\otimes o_K^{-1}}\det K^{\perp}/K.\]
Here, $o_K$ is the canonical orientation on the hyperbolic space $K\oplus K^{\vee}$, and the middle isomorphism is induced by writing locally $M =  K^{\perp}/K\oplus K \oplus K^{\vee}$.

By functoriality of Clifford algebras, we have morphisms 
\[\Cl(K^{\perp}/K) \twoheadleftarrow \Cl(K^{\perp})\hookrightarrow \Cl(M).\]

The kernel of the quotient map $\Cl(K^{\perp})\to \Cl(K^{\perp}/K)$ is given by the two-sided ideal generated by the sub-space $K\subset \Cl(K^{\perp})$.

\paragraph{Associated structure groups.}
Consider the standard quadratic bundle $R^n$. Let $K = R^{\oplus k}\times \{0\}\subset R^n$ be the subspace spanned by the first $k$ basis vectors $e_1,\ldots, e_k$ for some integer $k$ with $2k\leq n$. 
Then $K^{\perp}/K$ with its induced quadratic structure and orientation is canonically identified with the standard oriented quadratic bundle $R^{n-2k}$, via the isomorphism
\begin{equation}\label{eq:reduction-iso} 
\begin{aligned}
	R^{n-2k} & \to K^{\perp}/K\\
	e_i&\mapsto \overline{e}_{i + k}.
\end{aligned}  
\end{equation}

\begin{definition}
	We let $\G(n,k)$ denote the subgroup of $\SO(n)$ of elements preserving the subspace $K\subset R^n$ of the first $k$ coordinates. We define $\widetilde{\G(n,k)\times \bG_m}$ as the base change
	\begin{equation*}
		\begin{tikzcd}
			\widetilde{\G(n,k)\times \bG_m}\ar[r]\ar[d]& \Spin^{\bC}(n)\ar[d] \\
			\G(n,k)\times \bG_m\ar[r] & \SO(n)\times \bG_m
		\end{tikzcd}
	\end{equation*}
\end{definition}
where the right vertical map is the natural double covering.
Note that any element of $\G(n,k)$ automatically preserves $K^{\perp}$, and that we have an induced homomorphism $\G(n,k)\to \SO(n-2k)$.

Using Propositions \ref{prop:spinc-autom-even} and \ref{prop:spinc-autom-odd}, we find that $\widetilde{\G(n,k)\times \bG_m}$ is identified with the subgroup of $\SO(n)\times \GL(\sS_n)$ consisting of pairs $(\varphi,\Phi)$ for which $\varphi\in \G(n,k)$ and which satisfy the compatibility \eqref{diag:cl-act-comm}.

Consider the subspace $\Ann(K) \subset \sS_n$ consisting of those elements which are annihilated by every element of $K$. Explicitly, with $m= \lfloor n/2\rfloor$, one has that
\[\Ann(K) =\left\langle f_{i_1}\wedge \cdots \wedge f_{i_\ell}~|~\substack{\ell \geq 0 \\ k+1\leq i_1<\cdots <i_{\ell} \leq m }\right\rangle,\]
so $\dim_R \Ann(K) = 2^{m-k}$, and we have a natural isomorphism of $R$-modules 
\begin{equation}\label{eq:annihil-iso}
\Ann(K) = \bigwedge^{\bullet}\langle f_{k+1}, \ldots,f_m \rangle\simeq \bigwedge^{\bullet}\langle f_1,\ldots, f_{m-k}\rangle = \sS_{n-2k}.
\end{equation}

Since every element of $K^{\perp}$ anti-commutes with every element of $K$, the action of $\Cl(K^{\perp})$ perserves $\Ann(K)$, and the two sided ideal generated by $K\subset \Cl(K^{\perp})$ acts trivially. 
We obtain an induced action of $\Cl(K^{\perp}/K)$ on $\Ann(K)$. 
The isomorphisms \eqref{eq:reduction-iso} and \eqref{eq:annihil-iso} are compatible with the Clifford actions, in that they make the following diagram commute
\begin{equation*}
	\begin{tikzcd}
		\Cl(n-2k)\otimes \sS_{n-2k}\ar[r]\ar[d]& \sS_{n-2k}\ar[d] \\
		\Cl(K^{\perp}/K)\otimes \Ann(K)\ar[r] & \Ann(K).
	\end{tikzcd}
\end{equation*}

Let $(\varphi,\Phi)\in \widetilde{\G(n,k)\times \bG_m}$. For any $v\in K$ and $s\in \Ann(K)$, write $w:=\varphi^{-1}v\in K$, so  that
\[v \Phi(s) = (\varphi w)\Phi(s) = \Phi(w s) = 0.\]
It follows that $\Phi$ preserves $\Ann(K)$. Since $\Phi$ is compatible with $\varphi$ under the Clifford action, it follows that $\Phi|_{\Ann(K)}$ is compatible with the action of $\overline{\varphi}|_{K^{\perp}/K}$. 
We obtain an induced homomorphism
\[\widetilde{\G(n,k)\times \bG_m} \xrightarrow{\widetilde{\Gprojtw_{n,k}}} \Spin^{\bC}(n-2k).\]

In the following, let $\Gprojtw_{n,k}: \G(n,k)\times \bG_m \to \SO(n-2k)$ denote the map given by 
\[\Gprojtw_{k,m}: \left(\begin{bmatrix}
	A & * & * \\
	0 & B & * \\
	0 & 0 & (A^{-1})^t
\end{bmatrix},\,  t\right) \mapsto \left(B, t\cdot \det A\right).\]

\begin{proposition}
	The diagram 
	\begin{equation*}
		\begin{tikzcd}
			\widetilde{\G(n,k)\times \bG_m}\ar[r,"{\widetilde{\Gprojtw_{n,k}}}"]\ar[d]&\Spin^{\bC}(n-2k) \ar[d] \\
			\G(n,k)\times \bG_m\ar[r,"\Gprojtw_{n,k}"] & \SO(n-2k) \times \bG_m 
		\end{tikzcd}
	\end{equation*} 
	commutes. Here the vertical maps are the natural projections onto the first factor
\end{proposition}
\begin{proof}
	It suffices that both maps give the same result after composing with the projection to $\SO(n-2k)$ and to $\bG_m$ respectively. For the composition to $\SO(n-2k)$, this is immediate from the definition of $\Gprojtw_{n,k}$. We consider the projection to $\bG_m$.
	Let $(\varphi,\Phi)\in\widetilde{\G(n,k)\times \bG_m}\subseteq \SO(n)\times \GL(\sS_n)$.
	Following the square along the lower left corner and projecting to $\bG_m$, we obtain $\det \varphi|_K \cdot \operatorname{sn}(\varphi,\Phi)\in \bG_m$.
	Following the square along the upper right, we instead get $\operatorname{sn}(\varphi|_{K^{\perp}/K}, \Phi|_{\Ann(K)})$. 
	To compare the two expressions we use the characterization of the spinor norm given in Propositions \ref{prop:spinor-norm-even} and \ref{prop:spinor-norm-odd}. 
	Let $\alpha\in \widetilde{\G(n,k)\times \bG_m}$ corresponding to $(\varphi,\Phi)\in \SO(n)\times\GL(\sS_n)$. Assume first that $n$ is even.
	We have an isomorphism of $R$-modules
	\[F:\Ann(K)\otimes_{\Cl(K^{\perp})}\Ann(K)\to \sS_n\otimes_{\Cl(n)}\sS_n,\]
	which for homogeneous elements $x,y\in \Ann(K)$ is given by $x\otimes y\mapsto (-1)^{k|x|}x\otimes f_1\cdots f_k y$. 
	Let $x\otimes y\in \Ann(K)\otimes_{\Cl(K^{\perp})}\Ann(K)$ with $x$ homogeneous. We have 
	\[\sn(\alpha) F(x\otimes y) = (\Phi\otimes \Phi) F(x\otimes y) =(-1)^{k|x|} \Phi(x)\otimes \Phi(f_1\cdots f_k y) = (-1)^{k|x|} \Phi(x)\otimes \varphi(f_1)\cdots \varphi(f_k) \Phi(y) \]
	On the other hand, we have 
	\[\sn(\widetilde{\Gprojtw_{n,k}}(\alpha)) F(x\otimes y ) =  F(\Phi|_{\Ann(K)} (x) \otimes \Phi|_{\Ann(K)} (y))= (-1)^{k|x|} \Phi(x) \otimes f_1\cdots f_k \Phi(y).\]
	Here, the second equality uses that $\Phi$ respects degree modulo two. 
	It remains to compare these two terms. For this, we make the following
	\begin{claim}
		Let $z,w\in \Ann(K)$, and let $y_1,\ldots,y_k\in R^n$. Then the element 
		\[z\otimes y_1\cdots y_k w\in \sS_n\otimes_{\Cl(n)}\sS_n \]
		depends on $y_1,\ldots,y_k$ only via their image $\bar{y_1}\wedge\cdots \wedge \bar{y_k}\in \wedge^k (R^n/K^{\perp})$. 
	\end{claim}
	Assuming this claim, we have 
	\[(-1)^{k|x|} \Phi(x)\otimes \varphi(f_1)\cdots \varphi(f_k) \Phi(y) =\det(\varphi|_K)^{-1} (-1)^{k|x|} \Phi(x)\otimes f_1\cdots f_k \Phi(y),\]
	and hence $\sn(\alpha) \det(\varphi|_K)= \sn(\widetilde{\Gprojtw_{n,k}}(\alpha))$ as desired.
	\begin{proof}[Proof of the Claim]
		First, we show that if any $y_i\in K^{\perp}$, then the expression is zero. For this, we may introduce a new grading on $\sS_n=\bigwedge^{\bullet} \Lambda^{\vee}$ by setting the degree of $f_1,\ldots,f_k$ to be one and the degree of $f_{k+1},\ldots,f_m$ to be zero. Then, in the map $\sS_{n}\otimes_R\sS_n \to \sS_n\otimes_{\Cl(n)}\sS_n\simeq R$, all elements of degree $\neq k$ in $\sS_n\otimes_R\sS_n$ are sent to zero. Now, by assumption, $z,w$ have degree zero in this grading, and all $y_i$ are of degree $\leq 1$, so if any of them has degree zero, we find that the product vanishes. 
		We have shown that the element only depends on the images of $y_1,\ldots,y_k$ in $R^n/K^{\perp}$. Without loss of generality, we may therefore assume that $y_1,\ldots,y_k\in \langle f_1,\ldots,f_k\rangle$. But the sub-algebra $\Cl(\langle f_1,\ldots,f_k\rangle)$ is the exterior algebra -- the result follows from this.
	\end{proof}
	
	Now assume that $n$ is odd. Here, one considers instead
	\begin{align*}
		F:\Ann(K)\otimes_{\Cl_0(K^{\perp})} \Ann(K)&\to \sS_n\otimes_{\Cl_0(n)} \sS_n \\ 
		x\otimes y &\mapsto x \otimes f_1\cdots f_k y
	\end{align*} 
	without intervention of signs and where we take the tensor product only over the \emph{even} parts of the Clifford algebra. The rest of the argument proceeds exactly analogously, replacing tensor products over $\Cl(n)$ by those over $\Cl_0(n)$ where appropriate.	
\end{proof}

\begin{lemma}\label{lem:isored-orientations}
	The volume element $\omega_n$ preserves $\Ann(K)\subset \sS_n$. Its action agrees with the action of the volume element $\omega_{n-2k}$ of $\Cl(K^{\perp}/K)$.  
\end{lemma}
\begin{proof}
	To calculate this, we may use the splitting $R^n = K^{\perp}/K \oplus (K\oplus K^{\vee})$, and the induced factorization
	\[\Cl(n) = \Cl(K^{\perp}/K)\otimes_{\bZ_2}\Cl(K\oplus K^{\vee}).\] 
	Explicitly, we can write $\omega_{n} = \omega_{n-2k} \omega_K$, where $\omega_K=(e_1f_1-1)\cdots (e_kf_k-1)$. 
	Then, it suffices to show that $\omega_K$, or equivalently each $(e_if_i-1)$ acts as the identity on $\Ann(K)$ for $i=1,\ldots,k$. 
	This holds, since
	\[(e_if_i-1) = (1 - f_ie_i),\]
	and $e_i\in K$.
\end{proof}

\subsection{Torsors and Clifford modules}
We characterize torsors for the $\Spin^{\bC}$-groups. It is again most natural to split the discussion into even and odd rank, but we first discuss some common properties. 

Recall that torsors for the group $\SO(n)$ are equivalent to oriented quadratic bundles via Proposition \ref{prop:special-orthogonal-torsors}. 
A torsor $\sigma$ for the group $\Spin^{\bC}(n)$ gives rise to a torsor over $\SO(n)$ via the natural map $\Spin^{\bC}(n)\to \SO(n)$, and hence to an oriented quadratic bundle, which we denote $V_{\sigma}$. 

The group $\Spin^{\bC}(n)$ acts on $\Cl(n)$ by conjugation, in a way that respects the algebra structure. Thus, the associated vector bundle $P_{\sigma}\times_{\Spin^{\bC}(n)} \Cl(n)$ carries an associative algebra structure. As an $\mathcal{O}_X$-algebra, it is naturally isomorphic to $\Cl(V_{\sigma})$. 

Finally, let $\sS$ be any module for the Clifford algebra $\Cl(n)$. Then via the composition $\Spin^{\bC}(n)\to \Cl(n)\to \End_R(\sS)$, we may again form the associated bundle $\sR(\sS,{\sigma}):=P_{\sigma}\times_{\Spin^{\bC}(n)} \sS$. Since the $\Spin^{\bC}(n)$-action on $\sS$ is compatible with the conjugation action on $\Cl(n)$ under the Clifford multiplication, we find that $\sR(\sS,\sigma)$ is naturally a module under the algebra $\Cl(V_{\sigma})$. 

\paragraph{Even rank.}
Suppose that $n=2m$. 
\begin{definition}\label{def:spindata-even}
	Let $\mathcal{S}\mathit{pin}^{\bC}_n$ denote the \emph{category of spin data} whose objects are tuples 
	\[(M, q_M,o_M, \sS_M),\]
	where 
	\begin{itemize}
		\item $(M, q_M,o_M)$ is an oriented rank $n$ quadratic bundle on $X$,
		\item $\sS_M$ is a module for $\Cl(M)$, locally free of rank $2^m$ as an $\mathcal{O}_X$-module,
	\end{itemize}
	A morphism $(M, q_M,o_M, \sS_M) \to (N, q_N,o_N, \sS_N)$ in this category is given by a pair $(\varphi, \Phi)$, where 
	\begin{itemize}
		\item $\varphi: M\to N$ is an isomorphism of oriented quadratic bundles, and
		\item $\Phi:\sS_M\to \sS_N$ is an isomorphism of $R$-modules that is compatible over $\Cl(\varphi)$ with the Clifford actions. 
	\end{itemize}
\end{definition}
In Proposition \ref{prop:spinc-autom-even} we have exhibited $\Spin^{\bC}(n)$ as the group of automorphisms of the standard object $(R^n,q_n,o_n,\sS_n)$ via the map $\Spin^{\bC}(n)\to \SO(n)\times\GL(\sS_n)$. Given any (right-) torsor $P_{\sigma}$ for the group $\Spin^{\bC}(n)$ over a scheme, we have the associated oriented quadratic bundle $V_{\sigma}$, and the $\Cl(V_{\sigma})$-module $\sS_\sigma:=\sR(\sS_n,\sigma)$. This gives a functor
\begin{equation}\label{eq:spin-functor-todata-even}
	\left(\Spin^{\bC}(n) \mbox{-torsors} \right) \to \Spincat^{\bC}_n.
\end{equation}

Conversely, given a spin datum $(M,q_M,o_M,\sS_M)$, one may consider the sheaf of groups 
\[\operatorname{Isom}\left((M,q_M,o_M,\sS_M),(\mathcal{O}_X^n,q_n,o_n,\sS_n)\right).\]
parametrizing automorphisms to the standard spin datum. 
Since the automorphism group of the standard datum is $\Spin^{\bC}(n)$, and since locally such an isomorphism always exists, this is a (right-) torsor for $\Spin^{\bC}(n)$.
This construction provides a functor
\begin{equation}\label{eq:spin-functor-fromdata-even}
	\Spincat^{\bC}_n \to \left(\Spin^{\bC}(n)\mbox{-torsors}\right).
\end{equation}
We record without proof (see \cite[Proposition 4.5.1.6]{calfas_groupes} for a similar statement) 
\begin{proposition}\label{prop:spin-data-even}
	The maps \eqref{eq:spin-functor-todata-even} and \eqref{eq:spin-functor-fromdata-even} give mutually inverse equivalences of categories. 
\end{proposition}
Consider the exact sequence of groups 
\[1\to \bZ_2\to \Spin^{\bC}(n)\xrightarrow{(\operatorname{can}, \sn)} \SO(n)\times \bG_m \to 1.\]

By taking extensions of structure groups, this gives a functor
\[\left(\Spin^{\bC}(n) \mbox{-torsors}\right) \to \left(\SO(n)\times \bG_m \mbox{-torsors}\right).\]
On the level of Spin data and associated bundles, it sends a spin datum $\sigma = (M,q_M,o_M,\sS_M)$ to the pair $((M,q_M,o_M), L_{\sigma})$, where 
\[L_{\sigma} := \sS_M\otimes_{\Cl(M)} \sS_M.\]
Conversely, suppose we have a $\SO(n)\times \bG_m$-torsor given by a pair $(M, L)$, where $M$ carries an oriented quadratic structure. Then a reduction of structure group to $\Spin^{\bC}(n)$ is the same as giving a $\Cl(M)$-module $\sS$, together with an isomorphism of $\mathcal{O}_X$-modules $\sS\otimes_{\Cl(M)}\sS \to L$.

\paragraph{Odd rank.}
Now let $n=2m+1$ be odd. We have the following variant of Definition \ref{def:spindata-even}, the main difference being that we now require the Clifford module to be \emph{positive}.
\begin{definition}
	Let $\Spincat^{\bC}_n$ denote the category whose objects are tuples
	\[(M, q_M,o_M, \sS_M),\]
	where 
	\begin{itemize}
		\item $(M, q_M,o_M)$ is an oriented rank $n$ quadratic bundle,
		\item $\sS_M$ is a module for $\Cl(M)$, locally free of rank $2^m$ as an $R$-module, and on which the volume element $\omega\in \Cl(M)$ associated to the orientation acts with weight $1$.
	\end{itemize}
	A morphism $(M, q_M,o_M, \sS_M) \to (N, q_N,o_N, \sS_N)$ in this category is given by a pair $(\varphi, \Phi)$, where 
	\begin{itemize}
		\item $\varphi: M\to N$ is an isomorphism of oriented quadratic bundles, and
		\item $\Phi:\sS_M\to \sS_N$ is an isomorphism of $R$-modules that is compatible over $\Cl(\varphi)$ with the Clifford actions. 
	\end{itemize}
\end{definition}
In the same way as in the even rank case, we have maps 
\begin{equation}\label{eq:spin-torsor-maps-odd}
	\left(\Spin^{\bC}(n)\mbox{-torsors}\right) \rightleftharpoons \Spincat^{\bC}_n
\end{equation}
Given by taking the isomorphism scheme to the trivial $\Spin^{\bC}$-datum, and the associated bundle construction respectively. 
\begin{proposition}\label{prop:spin-data-odd}
	The maps \eqref{eq:spin-torsor-maps-odd} are inverse equivalences of categories. 
\end{proposition}

Analogously, we have the map 
\[\left(\Spin^{\bC}(n) \mbox{-torsors}\right) \to \left(\SO(n)\times \bG_m \mbox{-torsors}\right).\]
In this case, it is given on spin data by sending $\sigma = (M,q_M, o_M, \sS_M)$ to the pair $((M,q_M,o_M),L_{\sigma})$ where, due to the characterization of the spinor norm in the odd case, we have 
\[L_{\sigma}:=\sS_M\otimes_{\Cl_0(M)}\sS_M.\]
Given a pair $(V,L)$ of an oriented quadratic bundle and a line bundle, a reduction of structure group to $\Spin^{\bC}(n)$ of the corresponding $\SO(n)\times \bG_m$-torsor is the same as giving a $\Cl(V)$-module $\sS$ on which the volume element acts with weight $+1$, and an isomorphism $\sS\otimes_{\Cl_0(n)}\sS\simeq L$. 
\section{Untwisted pushforward}\label{app:untw}

Let $X$ be a Noetherian algebraic space\footnote{I.e. quasi-compact, quasi-separated and locally Noetherian, cf. \cite[\href{https://stacks.math.columbia.edu/tag/03E9}{Section 03E9}]{stacks-project}}.
Let $\pi:\mathcal{X}\to X$ be a $\mu_n$-gerbe over $X$. Recall that the $K$-group of $\mathcal{X}$ carries a $\bZ/n\bZ$-weight grading (see \cite[Theorem 4.7]{besch_decompositions} for the statement about categories of sheaves):
\[K_0(\mathcal{X}) = \bigoplus_{i=0}^{n-1} K_0^i(\mathcal{X}).\]
Moreover, pullback defines an identification $\pi^*:K_0(X)\to K_0^0(\mathcal{X})$, while pushforward acts as projection onto the summand $K_0^0(\mathcal{X})$ under this identification.

Our goal is to define a different pushforward map $K_0(\mathcal{X})\to K_0(X)[1/n]$, which -- roughly speaking -- acts on $K_0^i(\mathcal{X})$ by undoing the twist before applying the usual pushforward map. It turns out that this is well-defined after inverting $n$ in the coefficients. Our starting point is the following observation which we learnt from \cite[\S 5.1]{ot_counting}:

\begin{lemma}\label{lem:pullback-lem}
	Suppose there exists a section $\sigma:X\to \mathcal{X}$. Then, the map 
	\[\sigma^*:K(\mathcal{X})\to K(X)[1/n].\]
	is independent of the choice of $\sigma$.  
\end{lemma} 
\begin{proof}
	The existence of a section implies $\mathcal{X}\simeq X\times B\mu_n$. Choosing such an isomorphism gives an identification of $H^1(X, \mu_n)$ with the set of sections of $\mathcal{X}$. 
	
	Let $\mathcal{L}$ be the fixed weight $1$ line bundle on $B\mu_n$ corresponding to the standard representation of $\mu_n$. This gives a trivialization $K_0(\mathcal{X}) = \bigoplus_{i=0}^n \mathcal{L}^i \pi^*K_0(X)$ is a direct sum of copies of $K_0(X)$. Since $\mathcal{L}^{\otimes n}$ is trivial, for any section $\sigma$ the pullback $\sigma^*\mathcal{L}$ is an $n$-th root of $\mathcal{O}_X$. Here, we view $\mathcal{L}$ as an element of $K^{\circ}(\mathcal{X})$ (the Grothendieck group of perfect complexes) acting on $K_0(\mathcal{X})$ via the tensor product. 
	
	We claim that any $n$-th root of $\mathcal{O}_X$ acts as $1 = [\mathcal{O}_X]$ on $K_0(X)[1/n]$. Clearly, it is enough to show this for $p$-th power roots of unity for any prime $p|n$. 
	Therefore, suppose that $z^{q} = 1$ in $K^{\circ}(X)[1/n]$ where $q=p^r$ for some prime divisor $p$ of $n$, and that $z$ has rank one. Then $x:=z-1$ is contained in the augmentation ideal of $K^{\circ}(X)[1/n]$, and satisfies $(1+x)^q = 1$, therefore
	\[0 = x \left(\sum_{i=1}^q \binom{q}{i} x^{i-1}\right) = x (1 +\frac{1}{p^r}\sum_{i=1}^{q-1}\binom{q}{i+1} x^i).\]
	Since the augmentation ideal acts nilpotently on $K_0(X)$ \cite[Lemma 2.4]{edgr_riemroch}, the term in parentheses acts invertibly, and therefore $x=0$ as an endomorphism of $K_0(X)$. We conclude that $z = 1+x = 1 $ in $\End(K_0(X)[1/n])$. 
\end{proof}

\begin{remark}
	The conditions on $X$ are only needed to ensure that elements of $1 + I$ act as units on $K_0(X)$, where $I$ is the augmentation ideal. Thus, more generally, the conclusion of the Lemma still holds for an arbitrary algebraic stack $X$ if one replaces $K^{\circ}(X)$ and $K_0(X)$ by their completions along the augmentation ideal. 
\end{remark}

Lemma \ref{lem:pullback-lem} assumes that the class of the gerbe in $H^2(X,\mu_n)$ vanishes. We relax this assumption to require only that the its image $H^2(X,\bG_m)$ vanishes.

Note that for a prime number $p$ the formal identity $(1+X)^p = (1+Z)$ can be restated as
\[Z/p = F(X) = X +O(X^2)\]
when $p$ is assumed invertible. Hence, the series $F$ has a compositional inverse $G$ such that $X = G(Z/p)$. If $z$ is a nilpotent element of any ring in which $p$ is invertible, we can construct a $p$'th root of $1+z$ as $1 + G(z/p)$. Iterating this construction, we obtain a canonical $n$-th root of $1+z$ whenever $z$ is nilpotent and $n$ is invertible. 

\begin{construction}\label{constr:untwisted}
	Suppose that there exists a $1$-twisted line bundle $\mathcal{L}$ on $\mathcal{X}$. Then $\mathcal{L}^{\otimes n} = \pi^*L$ for some line bundle $L$ on $X$. We define $\widetilde{\pi}_{X,*}$ via its action on each summand:
	\begin{align*}
		\widetilde{\pi}_{X,*}: K_0^i(\mathcal{X})& \to K_0(X)[1/n]\\
		F& \mapsto L^{i/n} \otimes \pi_{X,*}(\mathcal{L}^{-i} \otimes F) 
	\end{align*} 
	Here $L^{1/n}$ denotes the $n$-th root construction described above, where we use that $L = 1 + (L-1)$, and that the augmentation ideal acts nilpotently on $K_0(X)$.
\end{construction}
\begin{lemma}
	The map $\widetilde{\pi}_{X,*}$ of Construction \ref{constr:untwisted} is independent of the choice of $\mathcal{L}$.
\end{lemma}
\begin{proof}
	Let $\mathcal{L}'$ be a different choice with $L'= (\mathcal{L}')^{\otimes n}$. Then $\mathcal{N}:=\mathcal{L}'\otimes \mathcal{L}^{\vee} = \pi_X^*N$ for some line bundle $N$ on $X$. We then have $(L')^{1/n} = L^{1/n}\otimes N$ as operators on $K_0(X)[1/n]$.  For any $F\in K_0^i(\mathcal{X})$, we calculate 
	\[(L')^{i/n}\otimes \pi_{X,*}(\mathcal{L}'^{-i} \otimes F) = L^{i/n} \otimes N^i \otimes \pi_{X,*}(\pi_X^*N^{-i} \otimes \mathcal{L}^{-i}\otimes F) = L^{1/n}\otimes \pi_{X,*}(\mathcal{L}^{-i}\otimes F),\]
	where the last equality used the projection formula.
\end{proof}

We are ready to state the main result.
\begin{theorem}\label{thm:untw}
	There is a unique collection of maps $\tilde{\pi}_{X,*}:K(\mathcal{X})\to K(X)[1/n]$ defined for any $\mu_n$-gerbe $\pi_X:\mathcal{X}\to X$ over a Noetherian algebraic space $X$ satisfying the following properties:
	\begin{enumerate}
		\item Whenever $\mathcal{X}$ admits a $1$-twisted line bundle, then $\widetilde{\pi}_{X,*}$ is given by Construction \ref{constr:untwisted}. 
		\item \label{item:untwii} For any 2-cartesian diagram 
		\begin{equation*}
			\begin{tikzcd}
				\mathcal{Z}\ar[r]\ar[d,"\pi_Z"]& \mathcal{X}\ar[d,"\pi_X"] \\
				Z\ar[r,"f"] & X
			\end{tikzcd}
		\end{equation*} 
		in which the vertical maps are $\mu_n$-gerbes over Noetherian algebraic spaces, the maps $\widetilde{\pi}_{X,*}$ and $\widetilde{\pi}_{Z,*}$ commute with flat pullback (resp. proper pushforward) if $f$ is flat (resp. proper). 
	\end{enumerate}
\end{theorem}

\begin{proof}
	We will first define the operation $\widetilde{\pi}_{X,*}$ and check the claimed compatibilities under the following assumptions on $\pi_X$. Once this is done, we check the claimed compatibilities. 
	\begin{enumerate}[label = \roman*)]
		\item \label{item:untw-constri} $\mathcal{X}$ admits a $1$-twisted line bundle.
		\item \label{item:untw-constrii} The class of $\pi_X$ in $H^2(X,\bG_m)$ is representable by an Azumaya algebra -- more generally, there exists a flat and proper map $b:B\to X$ from an algebraic space satisfying $Rb_*\mathcal{O}_B = \mathcal{O}_X$, such that the class of $\pi_X$ in $H^2(X,\bG_m)$ vanishes after pullback with $b$.  
		\item \label{item:untw-constriii} $X$ has the resolution property -- more generally, there exists a vector bundle torsor $g:A\to X$ on which the image of the class $\pi_X$ in $H^2(A,\bG_m)$ is representable by an Azumaya algebra.
		\item $X$ is an arbitrary Noetherian algebraic space.
	\end{enumerate}
	
	\paragraph{i):} We define $\widetilde{\pi}_{X,*}$ by Construction \ref{constr:untwisted}, which is well-defined by Lemma \ref{lem:pullback-lem}. If $f:Z\to X$ is a morphism as in \ref{item:untwii}, then the gerbe $\mathcal{Z}$ inherits a $1$-twisted line bundle. In this case, the construction of $\widetilde{\pi}_{X,*}$ is compatible with flat pullback (resp. proper pushforward) along $f$ and its base change, as one checks directly from the defining formula and the fact that these operations are preserve the decomposition into weights.
	
	\paragraph{ii):} If the class of $\pi_X$ in $H^2(X,\bG_m)$ is represented by an Azumaya algebra, let $b:B\to X$ be the associated Brauer--Severi variety. Then the morphism $b$ is smooth and proper and satisfies $Rb_*\mathcal{O}_B = \mathcal{O}_X$, hence $b^*:K_0(X)\to K_0(B)$ is injective and split by the proper pushforward $b_*$. 
	Moreover, the gerbe $\mathcal{B}:=B\times_X \mathcal{X}$ over $B$ has vanishing class in $H^2(B,\bG_m)$, thus admits a $1$-twisted line bundle, and we have a well-defined untwisted pushforward $\widetilde{\pi}_B: K_0(\mathcal{B})\to K_0(B)$ as in \ref{item:untw-constri}. Let $a:\mathcal{B} \to \mathcal{X}$ denote the projection.
	
	The flat pullback $b^*:K_0(X)\to K_0(B)$ is an injection and split by the proper pushforward $b_*:K_0(B)\to K_0(X)$. The analogous statement is true for the $K$-groups of the gerbes and $b'$ in place of $b$. We define $\widetilde{\pi}_{X,*}:K_0(\mathcal{X})\to K_0(X)[1/n]$ as the composition 
	\[K_0(\mathcal{X})\xrightarrow{a^*} K_0(\mathcal{B})\xrightarrow{\widetilde{\pi}_{B,*}} K_0(B)[1/n]\xrightarrow{b_*} K_0(X)[1/n].\]
	
	This is independent of choices: In fact, if $b_1:B_1\to X$ and $b_2:B_2\to X$ are arbitrary smooth and proper maps satisfying $Rb_{i,*}\mathcal{O}_{B_i} = \mathcal{O}_{X_i}$ and such that the class of $\pi_X$ in $H^2(X,\bG_m)$ vanishes after pullback along $b_i$, then the map $b_3:B_1\times_X B_2\to X$ satisfies the same properties. Then the result follows from the compatibility with flat pulllback and proper pushforward under assumption \ref{item:untw-constri}. In particular, if $\pi_X$ already satisfies \ref{item:untw-constri}, we recover the construction in that case. 
	
	Finally, suppose that $f:Z\to X$ is a flat (resp. proper) morphism. Then we may take the base change $b_Z:Z\times_X B \to Z$ of $b$ and use it to define $\widetilde{\pi}_{Z,*}$. From the compatibility shown in case \ref{item:untw-constri} and the fact that proper pushforward and flat pullback are mutually compatible, we obtain the desired compatibility of $\widetilde{\pi}_{X,*}$ and $\widetilde{\pi}_{Z,*}$.  
	
	\paragraph{iii)} 
	Now suppose that $X$ satisfies the resolution property. Using the equivariant Jouanolou trick of \cite[\S 3]{tot_resol}\footnote{See \cite[\S 2.5]{klt_dtpt} for a more detailed explanation.}, we can find a Joanolou device $g:A\to X$. By this, we mean that $g$ is a torsor for a vector bundle, and that $A$ is an affine scheme. For a vector bundle torsor, the flat pullback map $g^*:K_0(X)\to K_0(A)$ is an isomorphism (see e.g. \cite[Theorem 3.17]{khan_derived}). By Gabbers theorem \cite{dJ_result}, on an affine scheme any gerbe can be represented by an Azumaya-algebra. Therefore, in any case we have a well-defined map $\rho_{A,*}:K_0(\mathcal{A})\to K_0(A)[1/n]$ via \ref{item:untw-constrii}. Let $g':\mathcal{A}\to \mathcal{X}$ denote the base change of $g$. We define $\widetilde{\pi}_{X,*}$ as the composition
	\[K_0(\mathcal{X})\xrightarrow{g'^*} K_0(\mathcal{A})\xrightarrow{\widetilde{\pi}_{A,*}} K_0(A)[1/n] \xrightarrow{(g^*)^{-1}}K_0(X)[1/n].\]
	This is independent of choice of $g$, as the fiber product of any two vector bundle torsors is again one and so one may compare two different choices of $g$ by passing to their fiber product (which may not be an affine scheme, but will still satisfy the more general requirement). In particular, if the class of $\pi_X$ in $\bG_m$ can already be represented by an Azumaya algebra, one choose the trivial vector bundle torsor to conclude that the construction is compatible with the one in case \ref{item:untw-constrii}.
	
	For $f:Z\to X$ a flat (resp. proper morphism), and provided a vector bundle torsor on $X$ on which the class of $\pi_X$ becomes representable by an Azumaya algebra, we may use the its base change along $f$ to define $\widetilde{\pi}_{Z,*}$. Thus, compatibility of untwisted pushforward in this setting follows from the compatibility in case \ref{item:untw-constrii} and the fact that flat pullback and proper pushforward commute with flat pullback. 
	\paragraph{iv)} 
	Finally, we allow $X$ to be an arbitrary Noetherian algebraic space. By Lemma \ref{lem:envelope} below, we can find an \emph{envelope} (see Definition \ref{def:envelope}) $q:W\to X$ such that $W$ satisfies the resolution property. Let $\mathcal{W}:= W\times_X \mathcal{X}$. In particular $q$ is proper, and we have the untwisted pushforward for $\pi_W$ and $\pi_{W\times_X W}$ via \ref{item:untw-constriii}.
	
	Then by compatibility of the construction in \ref{item:untw-constriii}, we have a commutative diagram of solid arrows
		\begin{equation*}
		\begin{tikzcd}
			K(\mathcal{W}\times_{\mathcal{X}}\mathcal{W})\ar[r]\ar[d]& K(\mathcal{W})\ar[d]\ar[r]& K(\mathcal{X})\ar[d, dashed]\ar[r]& 0 \\
			K(W\times_X W)\ar[r] & K(W)\ar[r] & K(X)\ar[r]&0
		\end{tikzcd}
	\end{equation*}
	
	The rows are exact due to Lemmas \ref{lem:gillet} and \ref{lem:gillet-gerbe} below. Hence, there is a unique dashed arrow filling in the diagram, which we define to be $\widetilde{\pi}_{X,*}$. To see that this is independent of choices (and compatible with the definition in \ref{item:untw-constriii}), one uses that the fiber product of two envelopes is again one, and the same arguments as in the previous points. Similarly, compatibility with proper pushforward and flat pullback follows from the fact that the base change of an envelope is again an envelope and similar arguments as in the previous points. 
	
\end{proof}

\paragraph{Gillet--Envelopes}
\begin{definition}\label{def:envelope}
	Let $X$ be a Noetherian algebraic space. A proper morphism $q:X'\to X$ of Noetherian algebraic spaces is called an \emph{envelope} of $X$, if every point of $X$ admits a lift to $X'$. 
\end{definition}

\begin{lemma}\label{lem:envelope}
	Let $X$ be a Noetherian algebraic space. Then there exists an envelope of $X$ that has the resolution property. 
\end{lemma}
\begin{proof}
	This follows from repeated application of \cite[\href{https://stacks.math.columbia.edu/tag/0GV5}{Lemma 0GV5}]{stacks-project} by Noetherian induction. 
\end{proof}

\begin{lemma}[Gillet's exact sequence]\label{lem:gillet}
	Let $X$ be a Noetherian algebraic space and let $q:X'\to X$ be an envelope. Let $p_i:X'\times_X X'\to X'$ denote the projection onto the $i$-th factor. Then the following sequence is exact: 
	\begin{equation}\label{eq:gillet}
		K_0(X'\times_X X')\xrightarrow{p_{1,*}-p_{2,*}} K_0(X')\xrightarrow{q_*} K_0(X)\to 0.
	\end{equation}
\end{lemma}

We first begin by some easier observations
\begin{lemma}\label{lem:gillet-easy}
	The sequence \eqref{eq:gillet} is a complex and exact on the right. 
\end{lemma}
\begin{proof}
	That the sequence is a complex is immediate from $q\circ p_1 = q\circ p_2$ and functoriality of proper pushforwards. We argue that it is exact on the right. If not, there exists a class $x\in K_0(X)$ not in the image of $q_*$, and supported on a closed sub-space $Z\subseteq X$. Since $X$ is Noetherian, we may choose the pair $(x,Z)$ so that $Z$ is minimal. 
	Write $x = [F] - [G]$ with $F,G$ coherent sheaves on $Z$. Without loss of generality, we may assume $[G]=0$. Letting $Z_i$ run through the irreducible components of $Z$, we may write $[F] = \sum [F_i]$ with $F_i$ supported on $Z_i$. Thus, without loss of generality we may assume that $Z$ is irreducible.
	Let $\widetilde{Z}\subset W$ be an integral subspace that maps birationally onto $Z$ (which exists by assumption on $W$), and let $q_Z:\widetilde{Z}\to Z$ be the restriction of $q$. Then $[F]-q_*[q_Z^*F]$ is supported on a subspace strictly contained in $Z$, yielding a contradiction (Note that $q_Z^*$ really denotes the \emph{underived} pullback functor between sheaves).
\end{proof}

\begin{lemma}\label{lem:vanishing-generic}
	Let $Z$ be an irreducible and reduced Noetherian algebraic space with $\eta:\Spec K\to Z$ the generic point. Let $x\in K_0(Z)$ with $\eta^*x =0$. Then $x$ is supported on a proper closed subspace $Z'\subsetneq Z$. 
\end{lemma}
\begin{proof}
	Write $x = [F]-[G]$. Then $F$ and $G$ necessarily have the same rank $r$. Let $\mathcal{H}:=\mathcal{H}om(F,G)$ be their sheaf-hom, which is hence generically locally free of rank $r^2$. By \cite[\href{https://stacks.math.columbia.edu/tag/06NH}{Proposition 06NH}]{stacks-project}, the algebraic space $Z$ has a dense open subspace $U$ which is a scheme. Note that we have a map $\mathcal{H}\to \mathcal{H}om(\wedge^r F,\wedge^r G)$, where the latter is generically rank $1$. By possibly shrinking $U$, we may assume that we have a section $\varphi_U\in \Gamma(U,\mathcal{H}|_U)$ that maps to a generator of $\mathcal{H}om(\wedge^r F,\wedge^r G)|_U$, hence gives an isomorphism $\mathcal{F}|_U\to \mathcal{G}|_U$. 
	
	Let $\mathcal{I}$ be the ideal sheaf of the complement of $U$. By \cite[\href{https://stacks.math.columbia.edu/tag/07UM}{Lemma 07UM}]{stacks-project} there is a map $\varphi:\mathcal{I}^n \to \mathcal{H}$ for some $n$ that restricts to $\varphi|_U$ on $U$. This gives a morphism $\mathcal{I}^n\otimes F\to  G$ that is an isomorphism on $U$. Hence $[G]-[\mathcal{I}^n\otimes F]$ is supported on the complement of $U$. Since the natural map $\mathcal{I}^n\otimes F\to F$ is also an isomorphism on $U$, the same is true for $[\mathcal{I}^n\otimes F]-[F]$, and we are done.
\end{proof}

\begin{lemma}\label{lem:gillet-descent-step}
	Let $x\in K_0(X')$ be a class satisfying $q_*x =0$, supported on some non-empty closed subspace $Y\subset X'$ (wlog reduced) with image $Z\subset X$. Then there is a class $y\in x +\operatorname{Im}(p_{1,*} -p_{2,*})$ supported on a closed subspace $Y'$ with image $Z'\subsetneq Z$. 
\end{lemma}
\begin{proof}
	Write $x = [F] -[G]$ for coherent sheaves supported on $Y$. We may assume that at least one of $F,G$ has support strictly $Y$ (otherwise we are done).  Letting $Y_i$ denote the irreducible components of $Y$, we may write 
	\[[F] = \sum [F_i] \qquad \mbox{ and }\qquad [G] = \sum [G_i],\]
	where -- for each $i$ -- $F_i$ and $G_i$ are supported on $Y_i$.
	Let $Z_j$ be the irreducible components of $Z$ and for each $j$, and let $\tilde{Z}_j$ be a sub-space of $X'$ mapping birationally onto $Z_j$. For each $j$, there exists some $i$ such that $Y_i$ maps surjectively onto $Z_j$. 
	
	Assume first that -- for some $i$ -- there exists a surjective map $Y_i\to Z_j$  with $Y_i\neq\widetilde{Z}_j$. Consider the subspace $Y'_i:= Y_i\times_{Z_i} \widetilde{Z}_i \subseteq X'\times_X X'$.  Then the restriction $p_1|_{Y_i'}: Y_i'\to Y_i$ is birational. We also have the restriction $p_2|_{Y_i'}:Y_i'\to \widetilde{Z}_i$. 
	Let $y_i := [p_1|_{Y_i'}^*F_i]-[p_1|_{Y_i'}^*G_i]$. Then we may replace $x$ by $x-(p_{1,*}-p_{2,*})y_i$ replacing $Y$ by the union of $Y_j$ for $j\neq i$, and also $\widetilde{Z}_i$ and possibly further proper closed subspaces of $Y_i$ (where the difference of $[F_i]-[G_i]-p_{1,*}y_i$ is supported). After doing this for each such $Y_i$ at most once, we may assume that the connected components of $Y$ are exactly the $\widetilde{Z}_j$ and some other components that don't dominate any component of $Z$. 
	
	We claim that the difference $[F_1]-[G_1]$ is supported on a proper closed subspace of $\tilde{Z}_1$, which gives what we want. To see this, note that -- by assumption -- we have that
	\[\tilde{\eta}^*\left([F_1]-[G_1]\right) = \eta^*q_*y =0,\]
	where we use $\eta$ and $\tilde{\eta}$ to denote the inclusion of the generic point of $Z_1$ and $\tilde{Z}_1$ respectively. Now the result follows from Lemma \ref{lem:vanishing-generic}.  	
\end{proof}
\begin{proof}[Proof of Lemma \ref{lem:gillet}]
	By Lemma \ref{lem:gillet-easy}, all that is left to show is that \eqref{eq:gillet} is exact in the middle. Assume not. Then there exists a class $x\in \Ker(q_*)$, supported on some non-empty reduced closed subspace $Z\subseteq X$ that does not lie in the image of $p_{1,*}-p_{2,*}$. By the Noetherian assumption, we may choose an example with $Z$ minimal, so that there is no example supported on a proper subspace of $Z$, contradicting Lemma \ref{lem:gillet-descent-step}.  
\end{proof}
We also need the following strengthening of Lemma \ref{lem:gillet}.

\begin{lemma}[Gillet's exact sequence]\label{lem:gillet-gerbe}
	Let $X$ be a Noetherian algebraic space, let $q:X'\to X$ be an envelope and let $\pi_X:\mathcal{X}\to X$ be a gerbe for a diagonalizable group scheme on $X$. Let $q:\mathcal{X}'\to \mathcal{X}$ be the base change of $q$, and let $p_i:\mathcal{X}'\times_{\mathcal{X}} \mathcal{X}'\to \mathcal{X}'$ denote the projection onto the $i$-th factor. Then the following sequence is exact: 
	\begin{equation*}
		K_0(\mathcal{X}'\times_{\mathcal{X}} \mathcal{X}')\xrightarrow{p_{1,*}-p_{2,*}} K_0(\mathcal{X}')\xrightarrow{q_*} K_0(\mathcal{X})\to 0.
	\end{equation*}
\end{lemma}
\begin{proof}
	 The proof of Lemma \ref{lem:gillet} goes through with mostly straightforward changes (e.g. replacing ``subspace'' by ``substack'' and ``generic point'' with ``residual gerbe at the generic point'' etc.). In the proof of the analogue of Lemma \ref{lem:vanishing-generic}, one needs to first decompose $F$ and $G$ according to weight. Then $\mathcal{H}$ and $\mathcal{H}om(\wedge^rF, \wedge^r G)$ are untwisted sheaves and the same proof goes through. 
\end{proof}

\phantomsection
\addcontentsline{toc}{section}{References}

\begin{small}
\bibliographystyle{alpha}
\bibliography{paper}
\end{small}

\end{document}